\pdfoutput=1
\RequirePackage{ifpdf}
\ifpdf % We~are running pdfTeX in pdf mode
\documentclass[pdftex]{sigma}%,draft
\else
\documentclass{sigma}
\fi

\numberwithin{equation}{section}

\newtheorem{Theorem}{Theorem}[section]
\newtheorem*{Theorem*}{Theorem}
\newtheorem{Corollary}[Theorem]{Corollary}

\newtheorem{Proposition}[Theorem]{Proposition}
 { \theoremstyle{definition}
\newtheorem{Definition}[Theorem]{Definition}

\newtheorem{Remark}[Theorem]{Remark}}

\newcommand{\Zint}{\mathbb {Z}}

\newcommand{\Rea}{\mathbb {R}} % Real number field
\newcommand{\Cplx}{\mathbb {C}} % Complex number field

\begin{document}

%\allowdisplaybreaks

\newcommand{\arXivNumber}{2209.02227}

\renewcommand{\PaperNumber}{037}

\FirstPageHeading

\ShortArticleName{On $q$-Middle Convolution and $q$-Hypergeometric Equations}

\ArticleName{On $\boldsymbol{q}$-Middle Convolution\\ and $\boldsymbol{q}$-Hypergeometric Equations}

\Author{Yumi ARAI and Kouichi TAKEMURA}

\AuthorNameForHeading{Y.~Arai and K.~Takemura}

\Address{Department of Mathematics, Ochanomizu University,\\ 2-1-1 Otsuka, Bunkyo-ku, Tokyo 112-8610, Japan}
\Email{\href{mailto:araiyumi.math@gmail.com}{araiyumi.math@gmail.com}, \href{mailto:takemura.kouichi@ocha.ac.jp}{takemura.kouichi@ocha.ac.jp}}

\ArticleDates{Received October 14, 2022, in final form May 19, 2023; Published online June 05, 2023}

\Abstract{The $q$-middle convolution was introduced by Sakai and Yamaguchi. In this paper, we reformulate $q$-integral transformations associated with the $q$-middle convolution.
In particular, we discuss convergence of the $q$-integral transformations. As an application, we obtain $q$-integral representations of solutions to the variants of the $q$-hypergeometric equation by applying the $q$-middle convolution.}


\Classification{33D15; 39A13}

\section{Introduction}

The Gauss hypergeometric series is defined by
\begin{align}
&{}_2 F_1 (\alpha,\beta ;\gamma ;z)=1+ \frac{\alpha \beta}{ \gamma} z + \frac{\alpha (\alpha +1) \beta (\beta +1 )}{2! \gamma (\gamma +1)} z^2 + \dots +\frac{(\alpha )_n (\beta )_n}{n!(\gamma )_n} z^n + \cdots,
\label{eq:GaussHGs}
\end{align}
where $(\lambda )_n= \lambda (\lambda +1) \cdots (\lambda +n-1) $, and it frequently appears in mathematics and physics.
The Gauss hypergeometric series satisfies the hypergeometric differential equation
\begin{align}
& z(1-z) \frac{{\rm d}^2y}{{\rm d}z^2} + ( \gamma - (\alpha + \beta +1)z ) \frac{{\rm d}y}{{\rm d}z} -\alpha \beta y=0,
\label{eq:GaussHGE}
\end{align}
which is a standard form of second order Fuchsian differential equation with three regular singularities $\{ 0,1,\infty \}$.
The global structure of the solutions to the hypergeometric differential equation had been studied very well, and the integral representations of the solutions had been applied for the study.
An integral representation of the solutions to the hypergeometric differential equation is written as
\begin{equation}
y= \int_{C} w^{\alpha -\gamma} (1-w)^{\gamma -\beta -1} (z-w)^{-\alpha} \, {\rm d}w,
\label{eq:GHGintrep0}
\end{equation}
where $C$ is an appropriate integral contour.
It had been explained in several textbook, e.g.,~see Whittaker and Watson~\cite{WW}, and equation~\eqref{eq:GHGintrep0} is called Euler's integral representation.

In a modern theory, the integral representation of the solutions to the hypergeometric differential equation as equation~\eqref{eq:GHGintrep0} is related to the middle convolution, which was introduced by Katz~\cite{Katz}.
Dettweiler and Reiter~\cite{DR1,DR2} reformulated it for the Fuchsian system of differential equations written as
\begin{equation}
\frac{\rm d}{{\rm d}x} Y(x) =\bigg( \frac{A_1}{x-t_1}+\frac{A_2}{x-t_2}+\dots + \frac{A_r}{x-t_r} \bigg) Y(x), \label{eq:original}
\end{equation}
where $Y(x)$ is a column vector with $n$ entries and $A_1, A_2, \dots,A_r$ are constant matrices of size $n\times n$.
Roughly speaking, the middle convolution is a correspondence of the Fuchsian system of differential equations which is induced by the Euler's integral transformation.
For example, the function $u= w^{\alpha -\gamma} (1-w)^{\gamma -\beta -1} $ appears partially in equation~\eqref{eq:GHGintrep0}, and it satisfies the scalar linear differential equation ${\rm d}u/{\rm d}w= \{(\alpha -\gamma )/w + (\gamma -\beta -1 )/(w-1)\} u$.
By applying the middle convolution to this scalar equation, we obtain a differential equation of the form in equation~\eqref{eq:original} with $r=2$ and $n=2 $, and each element of $Y$ satisfies the hypergeometric differential equation~\eqref{eq:GaussHGE} with some suitable parameters.
The integral representation in equation~\eqref{eq:GHGintrep0} can be obtained as the correspondence of the solutions with respect to the middle convolution.
It is known that the hypergeometric differential equation is an example of the rigid Fuchsian system of differential equations, and the theory of the middle convolution is related with integral representations of solutions to the rigid Fuchsian system of differential equations (see~\cite{Har} for details).

The basic hypergeometric series
\begin{align}
&_2 \phi_1 \biggl( {a,b\atop{c}} ;q, z \biggr) = \sum_{n=0}^{\infty} \frac{(a ;q)_n (b ;q)_n}{(q;q)_n ( c ;q)_n} z^n \label{eq:2phi1}
\end{align}
was introduced by Heine in 1846.
Here, $(\lambda ; q)_n $ is the $q$-Pochhammer symbol defined by ${(\lambda ; q)_0 \!=\!1}$ and
\begin{equation*}
(\lambda ; q)_n= (1- \lambda )(1- \lambda q) \big(1- \lambda q^2\big) \cdots \big(1- \lambda q^{n-1}\big)
\end{equation*}
for the positive integer $n$.
In this paper, we assume that $q$ is a complex number such that $0 <|q|<1$.
The basic hypergeometric series is a $q$-analogue of the hypergeometric series, which is confirmed by setting
\[
a= q^{\alpha},\qquad b= q^{\beta}\qquad \text{and}\qquad c= q^{\gamma}.
\]
In fact, every term $z^n (q^{\alpha};q)_n \big(q^{\beta};q\big)_n / ((q;q)_n (q^{\gamma} ;q)_n)$ in the basic hypergeometric series converges to the term $z^n (\alpha)_n (\beta)_n / (n! (\gamma)_n)$ in the hypergeometric series (equation~\eqref{eq:GaussHGs}) as $q \to 1$.
The basic hypergeometric series satisfies the $q$-difference hypergeometric equation
\begin{equation}
(x-q) f(x/q) - ((a+b)x -q- c )f(x)+ (ab x- c )f(q x)=0.
\label{eq:qhyp}
\end{equation}
It also tends to the hypergeometric equation~\eqref{eq:GaussHGE} as $q \to 1$.
The Jackson integral representation of the basic hypergeometric series~\eqref{eq:2phi1} is known (cf.~\cite{Mim}), and we obtain an integral representation of the Gauss hypergeometric series~\eqref{eq:GaussHGs} as $q \to 1$.
In~\cite{HMST}, the variant of the $q$-hypergeometric equation of degree two was introduced as
\begin{align}
& \big(x-q^{h_1 +1/2} t_1\big) \big(x - q^{h_2 +1/2} t_2\big) g(x/q) + q^{k_1 +k_2} \big(x - q^{l_1-1/2}t_1 \big) \big(x - q^{l_2 -1/2} t_2\big) g(q x) \nonumber
\\
& \qquad{}-\big[ \big(q^{k_1} +q^{k_2} \big) x^2 +E x + p \big( q^{1/2}+ q^{-1/2}\big) t_1 t_2 \big] g(x) =0, \label{eq:varqhgdeg2intro}
\\[1mm]
& p= q^{(h_1 +h_2 + l_1 + l_2 +k_1 +k_2 )/2}, \qquad E= -p \big\{ \big(q^{- h_2}+q^{-l_2}\big)t_1 + \big(q^{- h_1}+ q^{- l_1}\big) t_2 \big\},\nonumber
\end{align}
where $0 \neq t_1 \neq t_2 \neq 0$, which is a $q$-analogue of the second order Fuchsian differential equation with three singularities $\{ t_1, t_2, \infty \}$.
Although the second order Fuchsian differential equation with three singularities $\{ t_1, t_2, \infty \}$ is transformed to the hypergeometric equation by an affine transformation, the variant of the $q$-hypergeometric equation of degree two would not be transformed to the $q$-hypergeometric equation given in equation~\eqref{eq:qhyp} by any simple transformation.
Several solutions of the variant of the $q$-hypergeometric equation of degree two were obtained in~\cite{HMST,MST} and were also obtained by Fujii and Nobukawa in~\cite{Fuj,FN}.
Similarly, the variant of the $q$-hypergeometric equation of degree three was introduced as a $q$-analogue of the second order Fuchsian differential equation with three singularities $\{t_1,t_2,t_3\}$ which do not include the points~$0$ and~$\infty$ (see equation~\eqref{eq:varqhgdeg3}).

A $q$-deformation of the middle convolution was constructed by Sakai and Yamaguchi~\cite{SY}.
The targets of the $q$-middle convolution are the linear systems of the $q$-difference equations written as
\begin{equation}
Y(q x) = B(x)Y(x), \qquad
B(x) = B_{\infty} + \sum^{N}_{i = 1}\frac{B_{i}}{1 - x/b_{i}},
\label{eq:YqxBYx0}
\end{equation}
where $B_{\infty}, B_{1},\dots,B_N $ are the square matrices of the same size and $b_1, b_2, \dots,b_N$ are the non-zero complex numbers which are different from one another.
It was shown in~\cite{SY} that the $q$-hypergeometric equation is written in the form of equation~\eqref{eq:YqxBYx0} for the case $N=1$.
The $q$-convolution of Sakai and Yamaguchi~\cite{SY} is related to the Jackson integral of the form
\begin{equation}
 \int^{\infty}_{0}f(s) \, {\rm d}_{q}s = (1-q)\sum^{\infty}_{n=-\infty}q^{n} f\big(q^n \big).
\label{eq:SYJint}
\end{equation}
See Section~\ref{sec:qmc} for the $q$-convolution and the $q$-middle convolution.

In this paper, we investigate $q$-deformed integral representations of solutions to the $q$-hyper-\\
geometric equation and its variants by using the $q$-middle convolution.
For this purpose, we need to improve the theory of Sakai and Yamaguchi on the $q$-convolution, and it is another important topic in this paper.
As was pointed out in~\cite{STT2}, we can replace the $q$-integral as equation~\eqref{eq:SYJint} in the theory of Sakai and Yamaguchi with the Jackson integral of the form
\begin{equation}
 \int^{\xi \infty}_{0}f(s) \, {\rm d}_{q}s = (1-q)\sum^{\infty}_{n=-\infty}q^{n} \xi f\big(q^n \xi \big), \qquad \xi \in \Cplx \setminus \{0\},
\label{eq:Jint}
\end{equation}
which had been frequently handled by Aomoto (see~\cite{Aom} and references therein).
We need to take care for convergence of the $q$-integral representation by the $q$-middle convolution to obtain rigorous results.
It is also necessary to make a foundation for application to the special functions.

In Section~\ref{sec:qintqhg}, we apply the $q$-convolution discussed in Section~\ref{sec:qmc} to obtain several solutions to the $q$-hypergeometric equation \eqref{eq:qhyp}.
Although these solutions have been known, the method by the $q$-convolution would be new.

In Section~\ref{sec:qintvarqhg}, we obtain $q$-integral representations of solutions to the variants of the $q$-hypergeometric equation by applying the $q$-middle convolution.
We think that these solutions would be new, while Fujii and Nobukawa obtained some results on solutions to the variants of the $q$-hypergeometric equation in~\cite{Fuj,FN}, which would not rely on the $q$-middle convolution.
We~display some of our results on the variants of the $q$-hypergeometric equation of degree two.
We~use the notations
\begin{align*}
& (a ; q)_{\infty} = \prod_{j=0}^{\infty} \big(1-q^j a\big), \qquad
(a_1, a_2, \dots,a_N ; q)_{\infty} = (a_1 ; q)_{\infty} (a_2 ; q)_{\infty} \cdots (a_N ; q)_{\infty}.
\end{align*}
The function $y(x) = x^{\mu} (\alpha_1 x, \alpha_2 x;q)_{\infty} / (\beta_1 x, \beta_2 x;q)_{\infty} $ satisfies the single $q$-difference equation $y(qx)= (B_{\infty} + B_1/(1- \alpha_1 x )+ B_2 /(1- \alpha_2 x)) y(x)$ for some constants $B_{\infty}$, $B_{1} $ and $B_{2}$ (see equation~\eqref{eq:gqxBxgxqJPN2} for details).
By applying the $q$-middle convolution to this single $q$-difference equation, we obtain the $q$-integral representations which we display as follows.
Set $\lambda = (h_1+h_2-l_1-l_2-k_1 +k_2 +1)/2 $.
We show in Section~\ref{sec:qmcvqhgdeg2}, that if $\lambda +k_1 -k_2 >0$, then the function
\begin{align*}
 g(x) & = x^{\lambda -k_2} \frac{\big( q^{\lambda -h_1 +1/2} x /t_1 ; q\big)_{\infty}}{\big( q^{ -h_1 +1/2} x /t_1 ; q\big)_{\infty}} {}_3\phi_2 \biggl(\!\! \begin{array}{c} q^{ -h_1 +1/2} x /t_1, q^{\lambda -h_1 + l_1},q^{\lambda -h_1 + l_2} t_2 /t_1 \\
 q^{\lambda -h_1 +1/2} x /t_1, q^{ 1 -h_1 + h_2} t_2 / t_1 \end{array}\! ;q,q^{\lambda +k_1 -k_2} \biggr)
\end{align*}
satisfies the variant of $q$-hypergeometric equation of degree two (equation~\eqref{eq:varqhgdeg2intro}).
Another solution is given by
\begin{align*}
g(x) = {}&x^{ -k_1} \frac{\big( q^{- \lambda +h_1 +3/2} t_1 / x, q^{ -\lambda + h_2 + 3/2}t_2 / x ; q\big)_{\infty}}{\big( q^{l_1 + 1/2} t_1 / x,q^{l_2 +1 /2} t_2 / x ; q\big)_{\infty}}
\\
& \times \!_3\phi_2 \biggl(\!\! \begin{array}{c} q^{l_1 + 1/2} t_1 / x,q^{l_2 +1 /2} t_2 / x, q^{1-\lambda} \\
 q^{- \lambda +h_1 +3/2} t_1 / x, q^{ -\lambda + h_2 + 3/2}t_2 / x \end{array}\! ;q,q^{\lambda + k_1 -k_2} \biggr).
\end{align*}
These solutions are obtained as some particular forms of the Jackson integral given in equation~\eqref{eq:Jint} with specific substitution for the parameter $\xi$.
By another choice of the parameter~$\xi$, we obtain the function
\begin{align}
 g(x) ={}& (q-1) x ^{-k_2} \frac{\big(q^{\lambda + l_1 + 1/2}t_1 / x, q^{l_1 -l_2 +1} t_1 / t_2, q;q\big)_{\infty}}{\big(q^{ l_1 + 1/2}t_1 / x, q^{\lambda -h_1 + l_1}, q^{\lambda - h_2 + l_1} t_1 /t_2 ;q\big)_{\infty}} \nonumber
\\[1mm]
&\times {}_3\phi_2 \biggl(\!\!\begin{array}{c} q^{ l_1 + 1/2}t_1 / x, q^{\lambda -h_1 + l_1}, q^{\lambda - h_2 + l_1} t_1 / t_2 \\ q^{\lambda + l_1 + 1/2} t_1 / x, q^{l_1 -l_2 +1} t_1 / t_2 \end{array}\! ;q,q \biggr).
\label{eq:anothersol}
\end{align}
However, it does not satisfy the variant of $q$-hypergeometric equation of degree two, because the general theory obtained in Theorem~\ref{thm:qcint} is not applicable.
Nevertheless, we can show that the function $g(x)$ in equation~\eqref{eq:anothersol} satisfies the variant of the $q$-hypergeometric equation of degree two with a non-homogeneous term, which is written as
\begin{align}
\label{eq:varqhgdeg2intrononhom}
\begin{split}
& \big(x-q^{h_1 +1/2} t_1\big) \big(x - q^{h_2 +1/2} t_2\big) g(x/q) + q^{k_1 +k_2} \big(x - q^{l_1-1/2}t_1 \big) \big(x - q^{l_2 -1/2} t_2\big) g(q x)
\\
& \qquad{}-\big[ \big(q^{k_1} +q^{k_2} \big) x^2 +E x + p \big( q^{1/2}+ q^{-1/2}\big) t_1 t_2 \big] g(x)
\\
& \qquad{}-(1-q)\big(1-q^{\lambda}\big) q^{l_2 + k_1 -k_2 -1/2} t_2 x^{1-k_2} =0,
\\[1mm]
& p= q^{(h_1 +h_2 + l_1 + l_2 +k_1 +k_2 )/2}, \qquad
E= -p \big\{ \big(q^{- h_2}+q^{-l_2}\big)t_1 + \big(q^{- h_1}+ q^{- l_1}\big) t_2 \big\} .
\end{split}
\end{align}
Note that a delicate aspect on convergence of the $q$-middle convolution, which we refine in this paper, is applied.
We also obtain that the function
\begin{align*}
g(x) = {}&(q-1)\frac{q^{-\lambda - l_1 + 1/2}}{t_1} x ^{1-k_2} \frac{\big(q^{-\lambda - l_1 + 3/2} x/ t_1, q^{-\lambda -l_2 +3/2} x / t_2, q ;q\big)_{\infty}}{\big( q^{-h_1 +1/2} x /t_1, q^{ - h_2 + 1/2} x /t_2, q^{1 -\lambda} ;q\big)_{\infty}}
 \\
&\times{}_3\phi_2 \biggl(\!\! \begin{array}{c} q^{ -h_1 +1/2} x /t_1, q^{- h_2 + 1/2} x / t_2, q^{1-\lambda}
\\
 q^{ -\lambda - l_1 + 3/2} x/ t_1, q^{-\lambda -l_2 +3/2} x /t_2 \end{array}\! ;q,q \biggr)
\end{align*}
satisfies equation~\eqref{eq:varqhgdeg2intrononhom}.

This paper is organized as follows.
In Section~\ref{sec:qmc}, we discuss the $q$-middle convolution and the convergence.
In Section~\ref{sec:qintqhg}, we discuss $q$-integral representations of solutions to the $q$-hypergeometric equation and the $q$-Jordan Pochhammer equation.
In Section~\ref{sec:qintvarqhg}, we discuss $q$-integral representations of solutions to the variants of the $q$-hypergeometric equation.

Throughout this paper, we assume that the complex number $q$ satisfies $0 <|q|<1$.

\section[q-middle convolution and convergence]{$\boldsymbol{q}$-middle convolution and convergence} \label{sec:qmc}
We review the definitions of the $q$-convolution and the $q$-middle convolution introduced by Sakai and Yamaguchi~\cite{SY}.
Although these definitions are unchanged in this paper, we modify discussions on the $q$-integrals.

Let ${\bf B}= (B_{\infty};B_{1},\dots,B_N)$ be the tuple of the square matrices of the same size and ${\bf b}= (b_1, b_2, \dots,b_N)$ be the tuple of the non-zero complex numbers which are different one another.
We denote by $E_{{\bf B}, {\bf b}}$ the linear $q$-difference equations
\begin{equation}
Y(q x) = B(x)Y(x), \qquad
B(x) = B_{\infty} + \sum^{N}_{i = 1}\frac{B_{i}}{1 - x/b_{i}} .
\label{eq:YqxBYx}
\end{equation}
\begin{Definition}[$q$-convolution~\cite{SY}] \label{def:qc}
Let ${\bf B}= (B_{\infty}; B_{1},\dots,B_N)$ be the tuple of $m\times m$ matrices and $\lambda \in \Cplx$.
Set $B_0 = I_m - B_\infty - B_{1} - \dots -B_N$.
We define the $q$-convolution $c_\lambda\colon (B_{\infty};B_{1},\dots,B_N)\allowbreak \mapsto (F_{\infty};F_{1},\dots,F_N)$ as follows:
\begin{align}\label{eq:bF}
\begin{split}
 &{\bf F}= ( F_\infty ; F_1, \dots, F_N) \mbox{ \rm is a tuple of $(N+1)m \times (N+1)m$ matrices,} \\
& F_i= \begin{pmatrix}
 {} & {} & O & {} & {}
 \\%[5pt]
 B_0 & \cdots & B_i - \big(1-q^\lambda\big)I_m & \cdots & B_N
 \\%[5pt]
 {} & {} & O & {} & {}
 \end{pmatrix}
 ((i+1)\text{-st}), \qquad 1\leq i \leq N,
 \\[1mm]
& F_\infty= I_{(N+1)m} - \widehat{F}, \qquad
 \widehat{F} = \begin{pmatrix}
 B_0 & \cdots & B_N \\%[5pt]
 \vdots & \ddots & \vdots \\%[5pt]
 B_0 & \cdots & B_N
 \end{pmatrix}\!.
 \end{split}
\end{align}
\end{Definition}
The $q$-convolution in Definition~\ref{def:qc} induces the correspondence of the linear $q$-difference equations
\begin{align*}
& Y(q x) = B(x)Y(x) \mapsto \widehat{Y}(q x) = F(x)\widehat{Y}(x),
\\
& B(x) = B_{\infty} + \sum^{N}_{i = 1}\frac{B_{i}}{1 - x/b_{i}}, \qquad
F(x) = F_{\infty} + \sum^{N}_{i = 1}\frac{F_{i}}{1 - x/b_{i}},
\end{align*}
and it is related with a $q$-integral transformation, which was established by Sakai and Yamaguchi~\cite{SY}.
We modify the relationship with the $q$-integral transformation.
Let $\xi \in \Cplx \setminus \{ 0 \}$.
The Jackson integral on the open interval $(0,\infty )$ is defined as the infinite sum
\begin{equation*}
 \int^{\xi \infty}_{0}f(s) \, {\rm d}_{q}s = (1-q)\sum^{\infty}_{n=-\infty}q^{n} \xi f\big(q^n \xi \big).
\end{equation*}
Note that the value of the Jackson integral may depend on the value $\xi $.
Set
\begin{align}
& P_{\lambda}(x, s) = \frac{\big(q^{\lambda +1} s/x;q\big)_{\infty}}{(q s/x;q)_{\infty}} . \label{eq:Plaxss/x}
\end{align}
We are going to consider the $q$-integral transformation associated to the $q$-convolution with paying attention to convergence.
Let $Y(x)$ be a solution to $E_{{\bf B}, {\bf b}}$ (see equation~\eqref{eq:YqxBYx}).
Set $b_0 =0$ and
\begin{equation*}
\widehat{Y}_{i}^{[K,L]}(x) = (1-q) \sum_{n=K}^{L} \frac{P_{\lambda}(x, s)}{s-b_{i}} s Y(s) \bigg|_{s= q^n \xi}, \qquad i=0,1,\dots,N,
\end{equation*}
where $K$ and $L$ are integers.
\begin{Proposition} \label{prop:convKL}
Let $Y(x)$ be a solution to $E_{{\bf B}, {\bf b}}$.
For $i=0,1,\dots,N $, we have
\begin{align*}
 \widehat{Y}_{i}^{[K,L]}(qx) ={}& \widehat{Y}_{i}^{[K,L]}(x) - \sum^{N}_{j = 0}B_{j}\widehat{Y}_{j}^{[K,L]}(x)
 \\
 &+ \frac{1}{1 - x/b_{i}} \biggl\{ -\big(1 - q^{\lambda}\big)\widehat{Y}_{i}^{[K,L]}(x)+ \sum^{N}_{j = 0}B_{j}\widehat{Y}_{j}^{[K,L]}(x) \biggr\}
 \\
& + \frac{(1-q) x}{x - b_{i}} \big( P_{\lambda}\big(x, q^{K-1} \xi\big)Y\big(q^{K} \xi\big) - P_{\lambda}\big(x, q^{L} \xi\big)Y\big(q^{L+1} \xi\big) \big) .
\end{align*}
\end{Proposition}

\begin{proof}\allowdisplaybreaks
It follows from the definition of $P_{\lambda}(x, s)$ that
\begin{equation}
P_{\lambda}(qx, s) = P_{\lambda}(x, s/q) = \frac{x - q^{\lambda}s}{x - s}P_{\lambda}(x, s). \label{eq:Plambda}
\end{equation}
Hence
\begin{equation*}
\frac{P_{\lambda}(qx, s)}{s - b_{i}} =
\frac{x - q^{\lambda}b_{i}}{x - b_{i}}\frac{P_{\lambda}(x, s)}{s - b_{i}} +
\frac{x}{x - b_{i}}\frac{P_{\lambda}(x, s/q) - P_{\lambda}(x, s)}{s} .
\end{equation*}
By multiplying by $s Y(s)$ and taking summation, we have
\begin{gather*}
\widehat{Y}_{i}^{[K,L]}(qx) =
\biggl\{ 1 + \frac{\big(1 - q^{\lambda}\big)b_{i}}{x - b_{i}} \biggr\} \widehat{Y}_{i}^{[K,L]}(x) \\
\hphantom{\widehat{Y}_{i}^{[K,L]}(qx) =}{}
+ \frac{(1-q) x}{x - b_{i}}\sum_{n=K}^{L} (P_{\lambda}(x, s/q) - P_{\lambda}(x, s)) Y(s)|_{s= q^n \xi}.
\end{gather*}
Since the function $Y(x)$ satisfies $E_{{\bf B}, b}$, it follows that
\begin{align*}
& Y(q s) -Y(s) = \biggl( B_{\infty} + \sum^{N}_{j = 1}\frac{B_{j}}{1 - s/b_{j}} - I_{m} \biggr)Y(s) = \sum^{N}_{j = 0} \frac{-s}{s - b_{j}}B_{j} Y(s) .
\end{align*}
Then
\begin{align*}
& \sum_{n=K}^{L} (P_{\lambda}(x, s/q) - P_{\lambda}(x, s)) Y(s)|_{s= q^n \xi}
 \\
&\qquad = P_{\lambda}\big(x, q^{K-1} \xi\big)Y\big(q^{K} \xi\big) - P_{\lambda}\big(x, q^{L} \xi\big)Y\big(q^{L+1} \xi\big) + \sum_{n=K}^{L} P_{\lambda}(x, s)( Y(q s) -Y(s))\bigg|_{s= q^n \xi}
 \\
&\qquad = P_{\lambda}\big(x, q^{K-1} \xi\big)Y\big(q^{K} \xi\big) - P_{\lambda}\big(x, q^{L} \xi\big)Y\big(q^{L+1} \xi\big) +\sum_{n=K}^{L} \sum^{N}_{j = 0}\frac{- s P_{\lambda}(x, s)}{s - b_{j}}B_{j} Y(s) \bigg|_{s= q^n \xi}
 \\
&\qquad = P_{\lambda}\big(x, q^{K-1} \xi\big)Y\big(q^{K} \xi\big) - P_{\lambda}\big(x, q^{L} \xi\big)Y\big(q^{L+1} \xi\big) - \frac{1}{1-q}\sum^{N}_{j = 0} B_{j} \widehat{Y}_{j}^{[K,L]}(x) .
\end{align*}
Therefore, we have
\begin{align*}
 \widehat{Y}_{i}^{[K,L]}(qx) ={}& \biggl\{ 1 + \frac{\big(1 - q^{\lambda}\big)b_{i}}{x - b_{i}} \biggr\} \, \widehat{Y}_{i}^{[K,L]}(x) - \frac{x}{x - b_{i}}\sum^{N}_{j = 0} B_{j} \widehat{Y}_{j}^{[K,L]}(x)
\\
& + \frac{(1-q) x}{x - b_{i}} \big( P_{\lambda}\big(x, q^{K-1} \xi\big)Y\big(q^{K} \xi\big) - P_{\lambda}\big(x, q^{L} \xi\big)Y\big(q^{L+1} \xi\big) \big)
\\
={}& \widehat{Y}_{i}^{[K,L]}(x) - \sum^{N}_{j = 0}B_{j}\widehat{Y}_{j}^{[K,L]}(x)
\\
&+ \frac{1}{1 - x/b_{i}}
\biggl\{ -\big(1 - q^{\lambda}\big)\widehat{Y}_{i}^{[K,L]}(x) + \sum^{N}_{j = 0}B_{j}\widehat{Y}_{j}^{[K,L]}(x) \biggr\}
\\
& + \frac{(1-q) x}{x - b_{i}} \big(P_{\lambda}\big(x, q^{K-1} \xi\big)Y\big(q^{K} \xi\big) - P_{\lambda}\big(x, q^{L} \xi\big)Y\big(q^{L+1} \xi\big)\big).\tag*{\qed}
\end{align*}
\renewcommand{\qed}{}
\end{proof}

\begin{Corollary}[{cf.~\cite[Theorem~2.1]{SY}}] \label{cor:qcint}
Let $Y(x)$ be a solution of $E_{{\bf B}, {\bf b}}$.
Set $b_0 =0$.
Assume that every component of $\widehat{Y}_{i}^{[K,L]}(x)$ converges as $K \to -\infty$ and $L \to +\infty $ for $i=0,1, \dots,N$ and
\begin{equation*}
\lim_{K \to -\infty} P_{\lambda}\big(x, q^{K-1} \xi\big)Y\big(q^{K} \xi\big) =0, \qquad
\lim_{L \to +\infty} P_{\lambda}\big(x, q^{L} \xi\big)Y\big(q^{L+1} \xi\big) =0 .
\end{equation*}
Define the function $\widehat{Y}(x) $ as
\begin{align*}
& \widehat{Y}_{i}(x) = \int^{\xi \infty}_{0}\frac{P_{\lambda}(x, s)}{s-b_{i}}Y(s) \, {\rm d}_{q}s, \quad i=0,\dots,N, \qquad
\widehat{Y}(x) = \begin{pmatrix} \widehat{Y}_{0}(x) \\ \vdots \\ \widehat{Y}_{N}(x) \end{pmatrix}\! . %\label{eq:qcint}
\end{align*}
Then the function $\widehat{Y}(x)$ satisfies the equation $E_{{\bf F}, {\bf b}}$, i.e.,
\begin{equation}
\widehat{Y}(q x) = \biggl( F_{\infty} + \sum^{N}_{i = 1}\frac{F_{i}}{1 - x/b_{i}} \biggr) \widehat{Y}(x). \label{eq:hYqx}
\end{equation}
\end{Corollary}

\begin{proof}
It follows from the assumption and Proposition~\ref{prop:convKL} that
\begin{equation*}
\widehat{Y}_{i}(x) = \lim_{K \to -\infty \atop{L \to +\infty}} \widehat{Y}_{i}^{[K,L]}(x) \qquad
\biggl( = \int^{\xi \infty}_{0} \frac{P_{\lambda}(x, s)}{s-b_{i}}Y(s) \, {\rm d}_{q}s \biggr)
\end{equation*}
and
\begin{align}
& \widehat{Y}_{i}(qx) = \widehat{Y}_{i}(x) - \sum^{N}_{j = 0}B_{j}\widehat{Y}_{j}(x) + \frac{1}{1 - x/b_{i}}
\biggl\{ -\big(1 - q^{\lambda}\big)\widehat{Y}_{i}(x) + \sum^{N}_{j = 0}B_{j}\widehat{Y}_{j}(x) \biggr\} \label{eq:hYiqx}
\end{align}
for $i=0, \dots,N$.
Since equation~\eqref{eq:hYiqx} is equivalent to equation~\eqref{eq:hYqx} by equation~\eqref{eq:bF}, the function $\widehat{Y}(x)$ satisfies the equation $E_{{\bf F}, b}$.
\end{proof}

\begin{Remark}
The function $(x/s)^{\lambda} (x/s;q)_{\infty} / \big(q^{-\lambda} x/s ;q\big)_{\infty}$ also satisfies equation~\eqref{eq:Plambda}.
Proposition~\ref{prop:convKL} and Corollary~\ref{cor:qcint} also hold true if we replace the function $P_{\lambda}(x, s)$ in equation~\eqref{eq:Plaxss/x} with
\begin{align*}
& P_{\lambda}(x, s) = (x/s)^{\lambda} \frac{(x/s;q)_{\infty}}{(q^{-\lambda} x/s ;q)_{\infty}}. %\label{eq:Plaxsx/s}
\end{align*}
\end{Remark}

We give a sufficient condition that the assumption of Corollary~\ref{cor:qcint} holds.
\begin{Proposition} \label{prop:qcintconv}
Let $[Y(s )]_k $ be the $k$-th component of $Y(s ) (\in \Cplx ^{m} )$.
We assume that the following conditions $(a)$ and $(b)$ hold:
\begin{enumerate}
\item[$(a)$] There exist $\varepsilon_1, C_1 \in \Rea_{>0} $ and $M_1 \in \Zint $ such that $|[Y(s)]_k | \leq C_1 |s |^{\varepsilon_1}$ for any $s \in \{ q^{n}\xi \mid n \geq M_1, \, n \in \Zint \} $ and $k=1,\dots,m$.

\item[$(b)$] There exist $\varepsilon_2, C_2 \in \Rea_{>0} $ and $M_2 \in \Zint $ such that $|[Y(s )]_k | \leq C_2 |s | ^{\lambda - \varepsilon_2} $ for any $s \in \{ q^{n}\xi \mid n \leq M_2, \, n \in \Zint \} $ and $k=1,\dots,m$.
\end{enumerate}
Then every component of $\widehat{Y}_{j}^{[K,L]}(x) $ converges absolutely as $K \to -\infty $ and $L \to +\infty $ for $j=0,1,\dots,N$ and
\begin{equation*}
\lim_{K \to -\infty} P_{\lambda}\big(x, q^{K-1} \xi\big)Y\big(q^{K} \xi\big) =0, \qquad
\lim_{L \to +\infty} P_{\lambda}\big(x, q^{L} \xi\big)Y\big(q^{L+1} \xi\big) =0 .
\end{equation*}
\end{Proposition}

\begin{proof}
It follows from equation~\eqref{eq:Plaxss/x} that $P_{\lambda}(x, s) \to 1$ as $s \to 0$ for each $x$.
Therefore, there exists the integer $M'_1$ such that
\begin{equation*}
\bigg| \frac{P_{\lambda}(x, s)}{s-b_{j}} s [Y(s)]_k \bigg| \leq 2 C_1 |s |^{\varepsilon_1}, \qquad
s= q^{n}\xi
\end{equation*}
for any integer $ n$ such that $n \geq M'_1 $, $k \in \{1, \dots, m\}$ and $j \in \{0, 1, \dots, N\}$.
Absolute convergence of the summation
\begin{equation*}
\sum_{n=M'_1}^{\infty} \frac{ P_{\lambda}(x, s)}{s-b_{i}} s[Y(s)]_k \bigg|_{s= q^{n}\xi}
\end{equation*}
follows from the convergence of the majorant series $2 C_1 |\xi |^{\varepsilon_1} (|q|^{\varepsilon_1})^n $.
It is also shown that $\lim_{L \to +\infty} P_{\lambda}\big(x, q^{L} \xi\big)Y\big(q^{L+1} \xi\big) =0 $.

Since $P_{\lambda}(x, s/q) = P_{\lambda}(x, s) \big(x - q^{\lambda}s\big)/(x - s)$, we have $|P_{\lambda}(x, s/q) / P_{\lambda}(x, s) | \leq |q|^{\lambda - \varepsilon_2 /2}$ for sufficiently large $s$ for each $x$.
Then there exists $M'_2 \in \Zint $ such that $\big|P_{\lambda}\big(x, q^{M'_2 + n'} \xi \big) \big| \leq \big(|q|^{\lambda - \varepsilon_2 /2}\big)^{-n'} \big| P_{\lambda}\big(x, q^{M'_2} \xi \big) \big|$ for any non-positive integer $n' $, and it follows that $|P_{\lambda}(x, q^{ n} \xi ) | \leq C' \big(|q|^{\lambda - \varepsilon_2 /2}\big)^{-n}$ for any integer $n$ such that $n \leq M'_2$ by setting $C'= |q|^{M'_2 (\lambda - \varepsilon_2 /2)} \big| P_{\lambda}\big(x, q^{M'_2} \xi \big) \big| $.
Hence, there exist $M_3 \in \Zint $ and $C'' \in \Rea_{>0} $ such that
\begin{equation*}
\biggl|\frac{P_{\lambda}(x, s)}{s-b_{j}}s[Y(s)]_k \biggr| \leq C'' \big(|q|^{\lambda - \varepsilon_2 /2}\big)^{- n} |s |^{\lambda - \varepsilon_2}, \qquad
s= q^{n} \xi
\end{equation*}
for any integer $n$ such that $n \leq M_3$, $k \in \{1, \dots, m\}$ and $j \in \{0, 1, \dots, N\}$.
Absolute convergence of the summation
\begin{equation*}
\sum ^{M_3}_{*n=-\infty} \frac{ P_{\lambda}(x, s)}{s-b_{i}}s[Y(s)]_k \bigg|_{s= q^{n}\xi}
\end{equation*}
follows from the convergence of the majorant series of the form $C''' \big(|q|^{ \varepsilon_2 /2}\big)^{-n} $.
It is also shown that $\lim_{K \to -\infty} P_{\lambda}\big(x, q^{K-1} \xi\big)Y(q^{K} \xi) =0 $.
\end{proof}

Therefore, we obtain the following theorem.

\begin{Theorem}[{cf.~\cite[Theorem 2.1]{SY}}]\label{thm:qcint}
Let $Y(x)$ be a solution of $E_{{\bf B}, {\bf b}}$ which satisfies the conditions $(a)$ and $(b)$ in Proposition~$\ref{prop:qcintconv}$.
Then the function $\widehat{Y}(x)$ defined by
\begin{align*}
& \widehat{Y}_{i}(x) = \int^{\xi \infty}_{0}\frac{P_{\lambda}(x, s)}{s-b_{i}}Y(s) \, {\rm d}_{q}s, \quad i=0,\dots,N, \qquad
\widehat{Y}(x) = \begin{pmatrix} \widehat{Y}_{0}(x) \\ \vdots \\ \widehat{Y}_{N}(x) \end{pmatrix} %\label{eq:qcint1}
\end{align*}
is convergent and it satisfies the equation $E_{c_{\lambda}({\bf B}), {\bf b}}$, where $c_{\lambda}({\bf B})= {\bf F} $ is the tuple defined by the convolution in equation~\eqref{eq:bF}.
\end{Theorem}

Note that Theorem~\ref{thm:qcint} is a one-parameter deformation of~\cite[Theorem 2.1]{SY} by Sakai and Yamaguchi with respect to the parameter $\xi $.

The $q$-middle convolution is defined by considering an appropriate quotient space.
\begin{Definition}[{$q$-middle convolution,~\cite{SY}}]\label{def:qmc}
We define the $\mathbf{F}$-invariant subspaces $\mathcal{K}$ and $\mathcal{L} $ of $(\mathbb{C}^m )^{N+1}$ as follows:
\begin{equation*}
 \mathcal{K} =
 \begin{pmatrix}
 \operatorname{ker}B_0 \\
 \vdots \\
 \operatorname{ker}B_N
 \end{pmatrix}\!, \qquad
 \mathcal{L} =
 \operatorname{ker}\big(\widehat{F} - \big(1 - q^{\lambda}\big)I_{(N+1)m}\big).
\end{equation*}
We denote the action of $F_k$ on the quotient space $(\mathbb{C}^m )^{N+1}/(\mathcal{K} + \mathcal{L})$ by $\overline{F}_k$, $k=\infty,1, \dots,N$.
Then the $q$-middle convolution $mc_\lambda$ is defined by the correspondence $ E_{{\bf B}, {\bf b}}\mapsto E_{\overline{{\bf F}}, {\bf b}} $, where $ \overline{{\bf F}} = \big( \overline{F}_{\infty}; \overline{F}_{1},\dots,\overline{F}_N \big) $.
\end{Definition}
The $q$-middle convolution $mc_\lambda$ would induce the integral transformation of the solutions by applying the integral transformation on the $q$-convolution, although it would be necessary to consider the subspace $\mathcal{K} + \mathcal{L} \subset (\mathbb{C}^m )^{N+1}$.

\section[q-integral representation of solutions to q-hypergeometric equation]{$\boldsymbol{q}$-integral representation of solutions \\ to $\boldsymbol{q}$-hypergeometric equation} \label{sec:qintqhg}

In this section, we obtain the $q$-hypergeometric equation and the $q$-Jordan Pochhammer equation by applying the $q$-convolution.

\subsection[q-convolution and q-hypergeometric equation]{$\boldsymbol{q}$-convolution and $\boldsymbol{q}$-hypergeometric equation}

Set $y(x)=x^{\mu}(\alpha x;q)_{\infty}/(\beta x;q)_{\infty}$.
Then
\begin{align*}
& y(qx) = (qx)^{\mu} \frac{(\alpha qx;q)_{\infty}}{(\beta qx;q)_{\infty}} = q^{\mu}\frac{1-\beta x}{1-\alpha x} y(x) = \bigg( q^{\mu}\frac{\beta}{\alpha}+\frac{q^{\mu}(1-\beta /\alpha)}{1-\alpha x}\bigg) y(x) .
\end{align*}
Namely, the function $y(x)$ satisfies the linear $q$-difference equation $y(qx)=B(x)y(x) $, where
\begin{align}
& B(x)=B_{\infty}+\frac{B_1}{1-x/b_1}, \qquad
B_{\infty}=q^{\mu}\frac{\beta}{\alpha}, \qquad
B_1=q^{\mu}\biggl(1-\frac{\beta}{\alpha}\biggr), \qquad
b_1=\frac{1}{\alpha}.
\label{eq:gqxBxgxhge}
\end{align}
Then $ B_0=1-B_{\infty}-B_1=1-q^{\mu}$.
We apply the convolution $c_{\lambda}$.
Then the matrices $ c_{\lambda} (\textbf{B}) = \textbf{F}=(F_1,F_{\infty})$ are written as
\begin{align*}
&F_1= \begin{pmatrix}
 0 & 0\\ B_0 & B_1-\big(1-q^{\lambda}\big)
 \end{pmatrix}
 = \begin{pmatrix}
 0 & 0\\ 1-q^{\mu} & q^{\mu} (1- \beta /\alpha )-1+q^{\lambda}
 \end{pmatrix}\!,
 \\[1mm]
&F_{\infty}= \begin{pmatrix}
 1-B_0 & -B_1\\ -B_0 & 1-B_1
 \end{pmatrix}
 = \begin{pmatrix}
 q^{\mu} & -q^{\mu} (1- \beta /\alpha ) \\ -(1-q^{\mu}) & 1-q^{\mu} (1- \beta /\alpha )
 \end{pmatrix}\! .
\end{align*}
The equation $E_{\textbf{F},b}$ is written as
\begin{align}
\widehat{Y}(qx)
 &= \bigg(F_{\infty}+\frac{F_1}{1-\alpha x}\bigg) \widehat{Y}(x) = \begin{pmatrix}
 q^{\mu} & -q^{\mu} (1-\beta /\alpha )
 \\[1mm]
 \displaystyle \frac{(1-q^{\mu})\alpha x}{1-\alpha x} & \displaystyle \frac{(-\alpha +(\alpha - \beta)q^{\mu})x+q^{\lambda}}{1-\alpha x}
 \end{pmatrix} \widehat{Y}(x). \label{eq:clahg}
\end{align}
On the system of first order linear $q$-difference equations with two components, each component satisfies the second order linear $q$-difference equation described as follows.
\begin{Proposition} \label{prop:g1g2single}
Assume that the functions $g_1(x)$ and $g_2(x)$ satisfy
\begin{align}\label{eq:g1qxg2qx}
\begin{split}
& g_1(qx)=a_{11}(x)g_1(x)+a_{12}(x)g_2(x) + b_1 (x),
\\[1mm]
& g_2(qx)=a_{21}(x)g_1(x)+a_{22}(x)g_2(x) + b_2 (x).
\end{split}
\end{align}
Then we obtain the following equation
\begin{align*}
 & \big\{ a_{11}(x/q)a_{22}(x/q) - a_{12}(x/q)a_{21}(x/q)\big\} \frac{a_{12}(x)}{a_{12}(x/q)} g_1(x/q) + g_1(qx)
 \\
 & \qquad{} -\bigg\{ a_{11}(x)+\frac{a_{12}(x)}{a_{12}(x/q)}a_{22}(x/q)\bigg\} g_1(x)
 - b_1 (x) + \frac{a_{12}(x)}{a_{12}(x/q)} a_{22}(x/q) b_1 (x/q)
 \\
 & \qquad {} - a_{12}(x) b_2 (x/q) = 0 .
\end{align*}
\end{Proposition}
\begin{proof}
We replace $x$ by $x/q$ in equation~\eqref{eq:g1qxg2qx} and eliminate the term $g_2(x/q)$.
Next, we eliminate the term $g_2 (x)$ by using the first equation of~\eqref{eq:g1qxg2qx}.
Then we obtain the proposition.
\end{proof}

We apply Proposition~\ref{prop:g1g2single} to equation~\eqref{eq:clahg} by setting
\[
\widehat{Y}(x)=\begin{pmatrix}
\widehat{y}_0(x)\\
\widehat{y}_1(x)
\end{pmatrix}.
\]
Then we have
\begin{align}
&\big(x-q^{\lambda+1}\beta^{-1}\big)\widehat{y}_0(x/q) + q^{-\mu} \beta^{-1} (\alpha x-q )\widehat{y}_0(qx) \nonumber
\\
&\qquad {}-\big\{ \big(q^{-\mu}\alpha\beta^{-1}+1\big)x- q\beta^{-1} \big(1 +q^{\lambda-\mu}\big)\big\} \widehat{y}_0(x)=0 . \label{eq:qeqhy0}
\end{align}
Set $\widehat{y}_0(x)=x^{\lambda}h(x)$.
Then the function $h(x)$ satisfies
\begin{align*}
&\big(x-q^{\lambda+1}\beta^{-1}\big)h(x/q) + \big(q^{2\lambda-\mu}\alpha\beta^{-1}x-q^{2\lambda-\mu+1}\beta^{-1}\big)h(qx) \nonumber\\
&\qquad {}-\big\{ \big(q^{\lambda-\mu}\alpha\beta^{-1}+q^{\lambda}\big)x-q^{\lambda+1}\beta^{-1} - q^{2\lambda-\mu+1}\beta^{-1}\big\}h(x)=0.
\end{align*}
In the case $q^{\lambda} = \beta $, it is written as the standard form of the $q$-hypergeometric equation
\begin{equation}
(x-q)g(x/q)+(abx-c)g(qx)-\{ (a+b)x-q-c \}g(x)=0 \label{eq:stqHGrel}
\end{equation}
with the parameters
\begin{equation*}
a=q^{-\mu}\alpha,\qquad
b=\beta,\qquad
c=q^{-\mu+1} \beta .
\end{equation*}
We can similarly obtain the $q$-difference equation for the function $\widehat{y}_1(x) $ as
\begin{align*}
&\big(x-q^{\lambda+1}\beta^{-1}\big)\widehat{y}_1(x/q) + q^{-\mu} \beta^{-1} (\alpha x- 1 )\widehat{y}_1(qx)
\\
&\qquad {}-\big\{ \big(q^{-\mu}\alpha\beta^{-1}+1\big)x- \beta^{-1} \big(q +q^{\lambda-\mu} \big)\big\} \widehat{y}_1(x)=0,
\end{align*}
and it is transformed to the standard form of the $q$-hypergeometric equation with the parameters
\begin{equation*}
a=q^{-\mu}\alpha,\qquad
b=\beta,\qquad
c=q^{-\mu}\beta .
\end{equation*}

If we apply Theorem~\ref{thm:qcint} by choosing the function $Y(x)$ as $Y(x)= x^{\mu}(\alpha x;q)_{\infty}/(\beta x;q)_{\infty}$, we obtain $q$-integral representations of solutions to the equation $E_{\textbf{F},b}$ in equation~\eqref{eq:clahg}.
We~in\-ves\-tigate the assumption of Theorem~\ref{thm:qcint} for the functions including $ x^{\mu}(\alpha x;q)_{\infty}/(\beta x;q)_{\infty}$.
\begin{Proposition} \label{prop:Yscon}
Let $\xi, \alpha_1, \dots, \alpha_N, \beta_1, \dots, \beta_N \in \Cplx \setminus \{ 0 \} $ such that $\beta_1 \xi, \dots, \beta_N \xi \not \in q^{\Zint} $.
Set
\begin{equation}
 y(s) = s^{\mu} \frac{(\alpha_1 s, \dots, \alpha_N s; q)_{\infty}}{(\beta_1 s, \dots, \beta_N s; q)_{\infty}} . \label{eq:Ysdef}
\end{equation}
\begin{enumerate}
\item [$(i)$]
If $\mu >0 $, then there exist $ \varepsilon, C_1 \in \Rea_{>0} $ and $M_1 \in \Zint $ such that $|y(s ) \mid \leq C_1 |s |^{\varepsilon} $ for any $s \in \{ q^{n}\xi \mid n \geq M_1, \, n \in \Zint \} $.

\item [$(ii)$] Let $\nu \in \Rea $. If $| q |^{\nu -\mu} |\alpha_1 \cdots \alpha_N /( \beta_1 \cdots \beta_N )| < 1$, then there exist $\varepsilon, C_2 \in \Rea_{>0} $ and $M_2 \in \Zint $ such that $|y (s ) | \leq C_2 |s | ^{\nu - \varepsilon} $ for any $s \in \{ q^{n}\xi \mid n \leq M_2, \, n \in \Zint \} $.
\end{enumerate}
\end{Proposition}
\begin{proof}
We obtain (i) by setting $\varepsilon = \mu $, because $ s^{-\mu} y(s) \to 1 $ as $s \to 0$.

We show (ii).
It follows from the assumption that there exists $\varepsilon \in \Rea_{>0} $ such that
\begin{equation}
| q | ^{-\mu} |\alpha_1 \cdots \alpha_N /( \beta_1 \cdots \beta_N )| <|q |^{2 \varepsilon} |q|^{-\nu} .
\label{eq:q-muineq}
\end{equation}
Since
\begin{equation}
y(s/q) = q^{-\mu} \frac{(1-\alpha_1 s/q) \cdots (1-\alpha_N s/q)}{(1-\beta_1 s/q) \cdots (1-\beta_N s/q)} y(s), \label{eq:Ysrel}
\end{equation}
we have $|y( s/q) / y( s) | \leq |q|^{-\mu - \varepsilon} |\alpha_1 \cdots \alpha_N /( \beta_1 \cdots \beta_N )|$ for sufficiently large $s$.
Then there exists $M' \in \Zint $ such that
\begin{equation*}
\big|y\big( q^{M' + n'} \xi \big) \big| \leq \big(|q|^{-\mu - \varepsilon} |\alpha_1 \cdots \alpha_N /( \beta_1 \cdots \beta_N )|\big)^{-n'} \big| y\big( q^{M'} \xi \big) \big|
\end{equation*}
for any non-positive integer $n' $, and it follows that there exist $C'\in \Rea_{>0} $ and $M'' \in \Zint $ such that $|y( q^{n} \xi ) | \leq C' (|q|^{-\mu - \varepsilon} |\alpha_1 \cdots \alpha_N /( \beta_1 \cdots \beta_N )|)^{-n} $ for any integer $n$ such that $n \leq M''$.
By applying equation~\eqref{eq:q-muineq}, we have
\begin{equation*}
\big|y\big( q^{n} \xi \big) \big| < C' \big(|q|^{-\nu + \varepsilon} \big)^{-n} = \big( C' | \xi|^{-\nu + \varepsilon} \big) |q ^n \xi |^{\nu - \varepsilon} .
\end{equation*}
Therefore, we obtain (ii).
\end{proof}
\begin{Remark}
The function $y(s) $ in equation~\eqref{eq:Ysdef} satisfies equation~\eqref{eq:Ysrel}.
The function
\begin{equation*}
y(s)= s^{\mu '} \frac{(q/(\beta_1 s), \dots,q/(\beta_N s);q)_{\infty}}{(q/(\alpha_1 s), \dots, q/(\alpha_N s);q)_{\infty}}, \qquad
q^{\mu '} \frac{\alpha_1 \cdots \alpha_N}{\beta_1 \cdots \beta_N} =q^{\mu}
\end{equation*}
also satisfies \eqref{eq:Ysrel}.
Proposition~\ref{prop:Yscon} is valid by replacing the assumption with $\alpha_1 \xi, \dots, \alpha_N \xi \not \in q^{\Zint} $.
\end{Remark}

We apply Theorem~\ref{thm:qcint} by picking up a solution of $y(qx)=B(x)y(x) $ determined by equation~\eqref{eq:gqxBxgxhge} as $y(x) = x^{\mu}(\alpha x;q)_{\infty}/(\beta x;q)_{\infty}$.
The assumption of Theorem~\ref{thm:qcint} is satisfied by the condition of Proposition~\ref{prop:Yscon} by setting $N=1 $ and $\nu =\lambda $.
Therefore, we obtain the following theorem.
\begin{Theorem} \label{thm:qcintqhge}
Let $\xi, \alpha, \beta \in \Cplx \setminus \{ 0 \} $ such that $ \beta \xi \not \in q^{\Zint}$.
If $\mu >0 $ and $| q |^{\lambda -\mu} |\alpha / \beta | < 1$, then the function $\widehat{Y}(x) $ defined by
\begin{align*}
& \widehat{Y}(x) = \begin{pmatrix} \widehat{y}_{0}(x) \\ \widehat{y}_{1}(x) \end{pmatrix},
 \\[1mm]
& \widehat{y}_{0}(x) = \int^{\xi \infty}_{0}\frac{P_{\lambda}(x, s)}{s} s^{\mu}\frac{(\alpha s;q)_{\infty}}{(\beta s;q)_{\infty}} \, {\rm d}_{q}s, \qquad
\widehat{y}_{1}(x) = \int^{\xi \infty}_{0}\frac{P_{\lambda}(x, s)}{s-1/\alpha}s^{\mu}\frac{(\alpha s;q)_{\infty}}{(\beta s;q)_{\infty}} \, {\rm d}_{q}s
\end{align*}
is convergent and it satisfies the equation $E_{{\bf F}, {\bf b}}$ given in equation~\eqref{eq:clahg}.
\end{Theorem}
The function $\widehat{y}_{0}(x)$ is expressed as
\begin{align*}
& \widehat{y}_{0}(x) = \int_0^{\xi \infty}s^{\mu -1} \frac{\big(q^{\lambda+1}s/x, \alpha s ; q\big)_{\infty}}{(qs/x, \beta s ; q)_{\infty}}\,{\rm d}_{q}s = (1-q)\sum_{n=-\infty}^{\infty}(q^n\xi)^{\mu}\frac{\big(q^{\lambda+n+1}\xi / x, q^n \xi\alpha ; q\big)_{\infty}}{\big(q^{n+1}\xi / x, q^n \xi\beta ; q\big)_{\infty}} .
\end{align*}
This is a bilateral basic hypergeometric series of the form $_2 \psi_2$.
If we specialize the value $\xi $, then we obtain the unilateral basic hypergeometric series.

If $\xi =1 /\alpha $, then it follows from $(q^n;q)_{\infty}=0$ for $n\in \mathbb{Z}_{\le 0}$ that
\begin{align}
 \widehat{y}_0(x) &= (1-q)\sum_{n=1}^{\infty}(q^n/\alpha)^{\mu} \frac{\big(q^{\lambda+n+1}/(\alpha x),q^n;q\big)_{\infty}}{\big(q^{n+1}/(\alpha x),q^n\beta/\alpha;q\big)_{\infty}} \nonumber
\\
& = (1-q)\alpha^{-\mu}q^{\mu}\frac{\big(q^{\lambda+2}/(\alpha x), q;q\big)_{\infty}}{\big(q^2/(\alpha x), q\beta/\alpha ;q\big)_{\infty}} \bigg\{ 1+ \sum_{n=2}^{\infty} \frac{\big(q^{2}/(\alpha x);q\big)_{n-1} (q \beta/\alpha;q)_{n-1}}{\big(q^{\lambda+2}/(\alpha x);q\big)_{n-1} (q ;q)_{n-1}} (q^{\mu})^{n-1} \bigg\} \nonumber
\\
& = (1-q)\alpha^{-\mu}q^{\mu}\frac{\big(q^{\lambda+2}/(\alpha x), q;q\big)_{\infty}}{\big(q^2/(\alpha x), q\beta/\alpha ;q\big)_{\infty}}
 \_2\phi_1 \bigg(\!\! \begin{array}{c} q^2/(\alpha x), q\beta/\alpha\\
 q^{\lambda+2}/(\alpha x) \end{array}\! ;q,q^{\mu} \bigg). \label{eq:hy0sum}
\end{align}
The assumption of Theorem~\ref{thm:qcint} is satisfied by the condition $\mu>0 $ solely, because the condition $|y(s ) | \leq C_2 |s | ^{\lambda - \varepsilon_2} $ for any $s \in \{ q^{n}\xi \mid n \leq 0, \, n \in \Zint \} $ is satisfied by the condition $y(s) = s^{\mu}(\alpha s;q)_{\infty}/(\beta s;q)_{\infty}=0$ for any $s \in \{ q^{n}\xi \mid n \leq 0, \, n \in \Zint \} $.
If $\xi=q^{-\lambda}x$, then we have
\begin{align*}
 \widehat{y}_0(x)& = (1-q)\sum_{n=0}^{\infty}\big(q^{n-\lambda}x \big)^{\mu}\frac{\big(q^{n+1}, q^{n-\lambda} \alpha x ; q\big)_{\infty}}{\big(q^{n+1-\lambda}, q^{n-\lambda}x \beta ; q\big)_{\infty}} \nonumber\\
& = (1-q)q^{-\lambda\mu}x^{\mu}\frac{\big( q^{-\lambda}\alpha x, q ;q\big)_{\infty}}{\big( q^{-\lambda}\beta x, q^{1-\lambda} ;q\big)_{\infty}}
 \_2\phi_1 \biggl(\!\! \begin{array}{c} q^{-\lambda}\beta x, q^{1-\lambda} \\
 q^{-\lambda}\alpha x \end{array}\! ;q,q^{\mu} \biggr).
\end{align*}
The assumption of Theorem~\ref{thm:qcint} is satisfied by the condition $\mu>0 $ solely.

Although the evaluation $\xi= 1/\beta $ is inconsistent with the assumption of Theorem~\ref{thm:qcintqhge}, we evaluate $\xi= A/\beta $, change the scaling of the function $\widehat{y}_0(x) $ and take the limit $A \to 1$.
Set $\xi=A /\beta $.
Under the hypothesis that the limit $A \to 1 $ commutes with the infinite summation, we have
\begin{align}
(1-A)\widehat{y}_0(x)& = (1-q)\sum_{n=-\infty}^{\infty}(1-A)(q^n\xi)^{\mu}\frac{\big(A q^{\lambda+n+1} / (\beta x),A q^n \alpha / \beta ; q\big)_{\infty}}{\big(A q^{n+1} /(\beta x ), A q^n ; q\big)_{\infty}} \nonumber
\\
&\xrightarrow[A\to 1]{} (1-q) \sum_{n=-\infty}^{ 0} (q^n/\beta )^{\mu}\frac{\big( q^{\lambda+n+1} /(\beta x ), q^n \alpha / \beta ; q\big)_{\infty}}{\big( q^{n+1}/(\beta x),q; q\big)_{\infty} (q^n;q)_{-n}} \nonumber
\\
& = (1-q)\beta^{-\mu}\frac{\big(q^{\lambda+1}/(\beta x), \alpha/\beta;q\big)_{\infty}}{(q/(\beta x), q;q)_{\infty}} \,_2\phi_1 \biggl(\!\! \begin{array}{c} q^{-\lambda}\beta x, q\beta/\alpha\\
 \beta x \end{array}\! ;q,q^{\lambda-\mu} \frac{\alpha}{\beta} \biggr).
\label{eq:qHGxiAbeta}
\end{align}
Similarly, we evaluate $\xi= Ax $, change the scaling of the function $\widehat{y}_0(x) $ and take the limit $A \to 1$.
Under the hypothesis that the limit $A \to 1 $ commutes with the infinite summation, we have
\begin{align}
& (1-A)\widehat{y}_0(x)\xrightarrow[A\to 1]{} (1-q)q^{-\mu}x^{\mu}\frac{\big( \alpha x/q, q^{\lambda} ;q\big)_{\infty}}{(\beta x/q, q ;q)_{\infty}} \, \!_2\phi_1 \biggl(\!\! \begin{array}{c} q^2/(\alpha x), q^{1-\lambda} \\
 q^2/(\beta x) \end{array}\! ;q,q^{\lambda-\mu} \frac{\alpha}{\beta} \biggr).
\label{eq:qHGxiAx}
\end{align}

To obtain alternative results corresponding to the specialization $\xi= 1/\beta $ and $\xi= x $, we replace the functions with
\begin{align*}
& P_{\lambda}(x, s) = (x/s)^{\lambda} \frac{(x/s;q)_{\infty}}{\big(q^{-\lambda} x/s ;q\big)_{\infty}}, \qquad
y(s)=s^{\mu '} \frac{(q/(\beta s);q)_{\infty}}{(q/(\alpha s);q)_{\infty}}
\end{align*}
with the condition $ q^{\mu '} \alpha /\beta =q^{\mu}$.
Then the function $y(x) $ also satisfies the $q$-difference equation $y(qx)=B(x)y(x) $ with the condition in equation~\eqref{eq:gqxBxgxhge}.
The function $\widehat{y}_{0}(x)$ obtained by the convolution is expressed as
\begin{align*}
 \widehat{y}_{0}(x) &= \int_0^{\xi \infty}s^{\mu ' -1} (x/s)^{\lambda} \frac{(x /s, q/ (\beta s ); q)_{\infty}}{\big(q^{-\lambda} x/s, q/ (\alpha s ) ; q\big)_{\infty}}\,{\rm d}_{q}s \nonumber\\
& = (1-q) x^{\lambda} \sum_{n=-\infty}^{\infty}(q^n\xi)^{\mu ' -\lambda} \frac{\big(x q^{-n} / \xi, q^{1-n} / (\beta \xi ) ; q\big)_{\infty}}{\big(x q^{-\lambda -n} / \xi, q^{1-n} / (\alpha \xi ) ; q\big)_{\infty}},
\end{align*}
and it also satisfies equation~\eqref{eq:qeqhy0} under the condition of convergence.
If $ \xi = 1/\beta $, then
\begin{align*}
\widehat{y}_{0}(x) &= (1-q)x^{\lambda} \sum_{n=-\infty}^{0}(q^n/\beta )^{\mu '-\lambda} \frac{\big(x q^{-n} \beta, q^{1-n} ; q\big)_{\infty}}{\big(q^{-\lambda} \beta x q^{-n}, q^{1-n} \beta / \alpha ; q\big)_{\infty}}
\\
& = (1-q)\beta ^{\lambda -\mu '} x^{\lambda} \frac{(\beta x, q ; q)_{\infty}}{\big(q^{-\lambda} \beta x, q \beta / \alpha ; q\big)_{\infty}} \biggl\{ 1+ \sum_{m=1}^{\infty} \frac{\big(q^{-\lambda} \beta x;q\big)_{m} (q \beta / \alpha ;q)_{m}}{(\beta x ;q)_{m} (q ;q)_{m}} \big(q^{\lambda -\mu '}\big)^{m} \biggr\}
 \\
& = (1-q)\beta ^{\lambda -\mu '} x^{\lambda} \frac{(\beta x, q ; q)_{\infty}}{\big(q^{-\lambda} \beta x, q \beta / \alpha ; q\big)_{\infty}}
 \_2\phi_1 \biggl(\!\! \begin{array}{c} q^{-\lambda} \beta x, q\beta/\alpha\\
 \beta x \end{array}\! ;q,q^{\lambda -\mu} \frac{\alpha}{\beta} \biggr),
\end{align*}
and it converges under the condition $| q |^{\lambda -\mu} |\alpha / \beta | < 1$.
If $ \xi = x$, then
\begin{align*}
\widehat{y}_{0}(x) &= (1-q) x^{\lambda} \sum_{n=-\infty}^{-1}(q^n x )^{\mu ' -\lambda} \frac{\big( q^{-n}, q^{1-n} / (\beta x) ; q\big)_{\infty}}{\big(q^{-\lambda} q^{-n}, q^{1-n} / (\alpha x) ; q\big)_{\infty}}
\\
& = (1-q) q^{\lambda -\mu} \frac{\alpha}{\beta} x^{\mu '} \frac{\big( q^2 / (\beta x), q ; q\big)_{\infty}}{\big( q^2 / (\alpha x), q^{1-\lambda} ; q\big)_{\infty}}
 \_2\phi_1 \biggl(\!\! \begin{array}{c} q^2 / (\alpha x), q^{1-\lambda},
 \\
 q^2 / (\beta x) \end{array}\! ;q,q^{\lambda -\mu} \frac{\alpha}{\beta} \biggr),
\end{align*}
and it converges under the condition $| q |^{\lambda -\mu} |\alpha / \beta | < 1$.
Although these functions do not coincides with the functions in equations~\eqref{eq:qHGxiAbeta},~\eqref{eq:qHGxiAx}, respectively, they match up to the pseudo-constant.
Namely, the functions
\begin{align*}
& \frac{\big(q^{\lambda+1}/(\beta x) ; q\big)_{\infty}}{(q/(\beta x) ;q)_{\infty}} \Big/ \bigg( x^{\lambda} \frac{(\beta x ; q)_{\infty}}{\big(q^{-\lambda} \beta x ; q\big)_{\infty}} \bigg),\qquad
x^{\mu}\frac{(\alpha x/q;q)_{\infty}}{( \beta x/q;q)_{\infty}} \Big/ \bigg(x^{\mu '} \frac{\big( q^2 / (\beta x) ; q\big)_{\infty}}{\big( q^2 / (\alpha x) ; q\big)_{\infty}} \bigg)
\end{align*}
are invariant under the transformation $x \to qx $.

It is well known that the standard form of the $q$-hypergeometric equation given in equation~\eqref{eq:stqHGrel} has the local solutions ${}_2\phi_1(a,b;c;q,x) $ and $x^{1-\gamma}\!_2\phi_1\big(qa/c,qb/c;q^2/c;q,x\big)$ ($q^{\gamma}= c$) about $x=0$ and also has the local solutions $x^{-\alpha^{\prime}}\!_2\phi_1(a,qa/c;qa/b;q,qc/(abx))$ and $x^{-\beta^{\prime}}\!_2\phi_1(b,qb/c;qb/a;\allowbreak q,qc/(abx))$, $q^{\alpha^{\prime}}=a$, $q^{\beta^{\prime}}=b$, about $x= \infty $.
We connect the functions obtained by the $q$-convolution with the local solutions.
It is known in~\cite[formula~(1.4.1)]{GR} that
\begin{equation}
{}_2\phi_1 \biggl(\!\! \begin{array}{c} a,b \\
 c \end{array}\! ;q,z \biggr)
 = \frac{(b,az;q)_{\infty}}{(c,z;q)_{\infty}} \,_2\phi_1 \biggl(\!\! \begin{array}{c} c/b,z \\
 az \end{array}\! ;q, b \biggr)
\label{eq:GR141}
\end{equation}
under the condition $|z|<1$ and $|b|<1$.
It follows from equation~\eqref{eq:GR141} that the function $\widehat{y}_0(x)$ in the case $\xi=1 /\alpha$ is written as
\begin{align*}
& \widehat{y}_0(x) =(1-q)\alpha^{-\mu}q^{\mu}\frac{\big(q, q^{\mu+1}\beta/\alpha;q\big)_{\infty}}{(q\beta/\alpha, q^{\mu};q)_{\infty}} \,_2\phi_1 \biggl(\!\! \begin{array}{c} q^{\lambda},q^{\mu}\\
 q^{\mu+1}\beta/\alpha \end{array}\! ;q, \frac{q^2}{\alpha x} \biggr).
\end{align*}
which corresponds to a local solution about $x=\infty$.
Similarly, the function $\widehat{y}_0(x)$ in the case $\xi= x $ is written as another local solution about $x=\infty $.
The function $\widehat{y}_0(x)$ in the case $\xi= 1/ \beta$ and the function $\widehat{y}_0(x)$ in the case $ \xi=q^{-\lambda}x$ are written as local solutions about $x=0$.

\subsection{\texorpdfstring{\mathversion{bold}$q$}{q}-convolution and
\texorpdfstring{\mathversion{bold}$q$}{q}-Jordan Pochhammer equation}

Set
\begin{equation*}
y(x)=x^{\mu} \prod_{j=1}^N \frac{(\alpha_j x;q)_{\infty}}{(\beta_j x;q)_{\infty}}.
\end{equation*}
Then $y(qx) / y(x) = q^{\mu} (1-\beta_1 x) \cdots (1-\beta_N x)/ \{ (1-\alpha_1 x) \cdots (1-\alpha_N x) \}$.
Namely, the function $y(x)$ satisfies the linear $q$-difference equation $y(qx)=B(x)y(x) $, where
\begin{align*}
& B(x)=B_{\infty}+ \sum_{j=1}^N \frac{B_j}{1-x/b_j}, \qquad
B_{\infty}=q^{\mu}\frac{\beta_1 \cdots \beta_N}{\alpha_1 \cdots \alpha_N},
\\
& B_k =q^{\mu} \frac{\alpha_k -\beta_k}{\alpha_k} \prod_{j=1, j\neq k}^N \frac{\alpha_k -\beta_j}{\alpha_k - \alpha_j}, \qquad
b_k=\frac{1}{\alpha_k},\qquad
k=1, \dots,N.
%\label{eq:gqxBxgxqJP}
\end{align*}
Then $ B_0=1-B_{\infty}-B_1 - \dots -B_N=1-q^{\mu}$.
We apply the convolution $c_{\lambda}$.
Then the matrices $ c_{\lambda} (\textbf{B}) = \textbf{F}=(F_{\infty}; F_1, \dots,F_N)$ are written as
\begin{align}
& F_i = \begin{pmatrix}
 {} & {} & {} & O & {} & {}
 \\
 1-q^{\mu}& B_1 & \cdots & B_i - 1 + q^\lambda & \cdots & B_N
 \\
 {} & {} & {} & O & {} & {}
 \end{pmatrix}
 ((i+1)\text{-st}), \qquad
 1\leq i \leq N, \nonumber
 \\[1mm]
& F_\infty = \begin{pmatrix}
 q^{\mu} & -B_1 & \cdots & -B_N \\
 q^{\mu} -1 & 1-B_1 & \cdots & -B_N \\
 \vdots & \vdots & \ddots & \vdots \\
 q^{\mu} -1 & -B_1 & \cdots & 1- B_N
 \end{pmatrix}\!.
\label{eq:FiJP}
\end{align}
The equation $E_{\textbf{F},b}$ is written as
\begin{align*}
 &\widehat{Y}(qx)
= \biggl(F_{\infty}+\sum_{j=1}^N \frac{F_j}{1-\alpha_j x}\biggr) \widehat{Y}(x) . %\label{eq:claqJP}
\end{align*}
By applying Proposition~\ref{prop:Yscon} and Theorem~\ref{thm:qcint}, we obtain the following theorem.

\begin{Theorem} \label{thm:qcintqJH}
Let $\xi \in \Cplx \setminus \{ 0 \} $.
If $\mu >0 $ and $\big| q ^{\lambda -\mu} \alpha_1 \cdots \alpha_N / (\beta_1 \cdots \beta_N)\big| < 1$, then the function~$\widehat{Y}(x) $ defined by
\begin{align}\label{eq:qcintqJP}
\begin{split}
& \widehat{Y}(x) = \begin{pmatrix} \widehat{y}_{0}(x) \\ \widehat{y}_{1}(x) \\ \vdots \\ \widehat{y}_{N}(x) \end{pmatrix}\!,\qquad
\\
& \widehat{y}_{0}(x) = \int^{\xi \infty}_{0}\frac{P_{\lambda}(x, s)}{s} s^{\mu} \prod_{j=1}^N \frac{(\alpha_j s;q)_{\infty}}{(\beta_j s;q)_{\infty}} \, {\rm d}_{q}s,
 \\
&\widehat{y}_{k}(x) = \int^{\xi \infty}_{0}\frac{P_{\lambda}(x, s)}{s-1/\alpha_k}s^{\mu}\prod_{j=1}^N \frac{(\alpha_j s;q)_{\infty}}{(\beta_j s;q)_{\infty}} \, {\rm d}_{q}s,\qquad
 k=1,\dots,N
\end{split}
\end{align}
is convergent and it satisfies the equation $E_{c_{\lambda}({\bf B}), {\bf b}}$, where $c_{\lambda}({\bf B})= {\bf F} $ is the tuple defined by the convolution in equation~\eqref{eq:FiJP}.
\end{Theorem}
The functions $\widehat{y}_{k}(x)$, $k=0,1,\dots,N$, are expressed as
\begin{align*}
& \widehat{y}_{0}(x) = (1-q)\sum_{n=-\infty}^{\infty}(q^n\xi)^{\mu}\frac{\big(q^{\lambda+n+1}\xi /x, q^n \alpha_1 \xi,\dots, q^n \alpha_N \xi ; q\big)_{\infty}}{\big(q^{n+1}\xi / x, q^n \beta_1 \xi,\dots, q^n \beta_N \xi ; q\big)_{\infty}},
\\
& \widehat{y}_{1}(x) = (q-1)\alpha_1 \sum_{n=-\infty}^{\infty}(q^n\xi)^{\mu+1}\frac{\big(q^{\lambda+n+1}\xi / x, q^{n+1} \alpha_1 \xi, q^n \alpha_2 \xi,\dots, q^n \alpha_N \xi ; q\big)_{\infty}}{\big(q^{n+1}\xi / x, q^n \beta_1 \xi,q^n \beta_2 \xi,\dots, q^n \beta_N \xi ; q\big)_{\infty}},
 \\[-1mm]
& \qquad \vdots
 \\[-1mm]
& \widehat{y}_{N}(x) = (q-1)\alpha_N \sum_{n=-\infty}^{\infty}(q^n\xi)^{\mu+1}\frac{\big(q^{\lambda+n+1}\xi / x, q^{n} \alpha_1 \xi, \dots, q^{n} \alpha_{N-1} \xi, q^{n+1} \alpha_N \xi ; q\big)_{\infty}}{\big(q^{n+1}\xi / x, q^n \beta_1 \xi,\dots, q^n \beta_{N-1} \xi, q^n \beta_N \xi ; q\big)_{\infty}} .
\end{align*}
These are bilateral basic hypergeometric series of the form $_{N+1} \psi_{N+1}$.
If we specialize the value~$\xi$, then we obtain the unilateral basic hypergeometric series $_{N+1} \phi_{N}$.

Mimachi obtained the $q$-difference systems associated with a Jackson integral of Jordan-Pochhammer type in~\cite{Mim} and discussed the connection problem there.
It would be interesting to combine the results and the methods in~\cite{Mim} to our ones.

\section[q-integral representation of solutions to variants of q-hypergeometric equation]{$\boldsymbol{q}$-integral representation of solutions to variants \\ of $\boldsymbol{q}$-hypergeometric equation} \label{sec:qintvarqhg}

The variants of the $q$-hypergeometric equation were introduced in~\cite{HMST}.
In this section, we obtain them by considering the middle convolution which are related with the $q$-Jordan Pochhammer equation.
Note that the $q$-Jordan Pochhammer equation was obtained by applying the convolution to the scalar $q$-difference equation which the function
\[
y(x)=x^{\mu} (\alpha_1 x, \dots, \alpha_N x ;q)_{\infty} /(\beta_1 x, \dots, \beta_N x ;q)_{\infty}
\]
satisfies.
The middle convolution was defined in Definition~\ref{def:qmc} by taking a quotient space divided by the invariant spaces $\mathcal{K}$ and $\mathcal{L} $.
We investigate the case $N=2,3$ with the condition that the space $\mathcal{K}$ or/and the space $\mathcal{L} $ is/are non-zero.

\subsection[q-middle convolution and variants of q-hypergeometric equation of degree two]{$\boldsymbol{q}$-middle convolution and variants of $\boldsymbol{q}$-hypergeometric equation\\ of degree two} \label{sec:qmcvqhgdeg2}

Assume that $\alpha_1$, $\alpha_2$, $\beta_1$, $\beta_2$ are mutually distinct and set
\begin{equation*}
y(x)=x^{\mu}\frac{(\alpha_1 x, \alpha_2 x;q)_{\infty}}{(\beta_1 x, \beta_2 x;q)_{\infty}}.
\end{equation*}
The function $y(x)$ satisfies the linear $q$-difference equation $y(qx)=B(x)y(x) $, where
\begin{align}\label{eq:gqxBxgxqJPN2}
\begin{split}
& B(x)=B_{\infty} + \frac{B_1}{1-x/b_1}+ \frac{B_2}{1-x/b_2}, \qquad
b_1=\frac{1}{\alpha_1},\; b_2=\frac{1}{\alpha_2},
\\
& B_{\infty} = q^{\mu}\frac{\beta_1\beta_2}{\alpha_1\alpha_2},\qquad
B_1 = q^{\mu}\frac{(\alpha_1-\beta_1)(\alpha_1-\beta_2)}{\alpha_1(\alpha_1-\alpha_2)},\qquad
B_2 = q^{\mu}\frac{(\alpha_2-\beta_1)(\alpha_2-\beta_2)}{\alpha_2(\alpha_2-\alpha_1)}.
\end{split}
\end{align}
Note that $ B_0=1-B_{\infty}-B_1 -B_2=1-q^{\mu}$.
We apply the convolution $c_{\lambda}$.
Then the matrices $ c_{\lambda} (\textbf{B}) = \textbf{F}=(F_{\infty}; F_1,F_2)$ are written as
\begin{align*}
& F_1 = \begin{pmatrix}
 0 & 0 & 0\\ 1-q^{\mu} & B_1-1+q^{\lambda} & B_2\\ 0 & 0 & 0
 \end{pmatrix}\!,\qquad
F_2 = \begin{pmatrix}
 0 & 0 & 0\\ 0 & 0 & 0\\ 1-q^{\mu} & B_1 & B_2-1+q^{\lambda}
 \end{pmatrix},
 \\
& F_{\infty} = \begin{pmatrix}
 q^{\mu} & -B_1 & -B_2\\ q^{\mu}-1 & 1-B_1 & -B_2\\ q^{\mu}-1 & -B_1 & 1-B_2
 \end{pmatrix}\! .
\end{align*}
The corresponding $q$-difference equation $E_{\mathbf{F},b}$ is written as
\begin{align}
\widehat{Y}(qx)
 &= \biggl(F_{\infty}+\frac{F_1}{1- \alpha_1 x}+\frac{F_2}{1- \alpha_2 x} \biggr) \widehat{Y}(x), \qquad
\widehat{Y}(x) = \begin{pmatrix} \widehat{y}_{0}(x) \\ \widehat{y}_{1}(x) \\ \widehat{y}_{2}(x) \end{pmatrix}\! .
\label{eq:Yqx33}
\end{align}
Theorem~\ref{thm:qcintqJH} states that the $q$-integral representation in equation~\eqref{eq:qcintqJP} satisfies equation~\eqref{eq:Yqx33} under the condition $\mu >0 $ and $\big| q ^{\lambda -\mu} \alpha_1 \alpha_2 / (\beta_1 \beta_2)\big| < 1$.
We can obtain the single $q$-difference equation for $\widehat{y}_{1}(x) $ from equation~\eqref{eq:Yqx33} by a straightforward calculation, and it is written as
\begin{align}
& \big(1 - q^2 x \alpha_1\big)(1 - q x \alpha_2) \widehat{y}_{1}\big(q^3 x\big) - \big\{ q^2 \big(q^{\mu +1} \beta_1 \beta_2 +\alpha_1 \alpha_2 (q+1) \big) x^2 \nonumber
\\
& \qquad{} - \big(q^{\mu +1} (\beta_1 +\beta_2) + \big(q^{\lambda} +1\big) (q \alpha_1+\alpha_2) \big) x + q^{\mu +1} + (q + 1 )q^{\lambda} \big\} \widehat{y}_{1}\big(q^2 x\big) \nonumber
\\
&\qquad {}+ \big\{ q^2 (q^{\mu} (q +1 )\beta_1 \beta_2 +\alpha_1 \alpha_2 ) x^2 -q \big(q^{\mu +1} (\beta_1+\beta_2) \big(q^{\lambda} +1\big) + q^{\lambda} (q \alpha_1+\alpha_2) \big) x \nonumber
\\
& \qquad {}+ q^{\lambda +1} \big( q^{\mu} (q +1) + q^{\lambda} \big) \big\} \widehat{y}_{1}(q x) - q^{\mu +2} \big(\beta_1 x-q^{\lambda}\big) \big(\beta_2 x-q^{\lambda}\big) \widehat{y}_{1}( x) =0.
\label{eq:hy1q3}
\end{align}

The spaces $\mathcal{K} $ and $\mathcal{L} $ were introduced in Definition~\ref{def:qmc} to formulate the $q$-middle convolution.
If $\mu =0$ (resp.~$q^{\lambda} = q^{\mu} \beta_1 \beta_2 /(\alpha_1 \alpha_2) $), then it is shown that $\dim \mathcal{K} =1$ (resp.~$\dim \mathcal{L} =1 $).

\subsubsection{The case \texorpdfstring{\mathversion{bold}$\mu =0$}{mu = 0}} \label{sec:mu0}

We discuss the case $\mu =0$.
In this case, the space $\mathcal{K} $ is spanned by the vector $^t (1,0,0)$.
Since $q^{\mu} =1 $, the equation $E_{\mathbf{F},b}$ is written as
\begin{align*}
\begin{pmatrix} \widehat{y}_{0}(qx) \\ \widehat{y}_{1}(qx) \\ \widehat{y}_{2}(qx) \end{pmatrix}
= \begin{pmatrix}
 1 & -B_1 & -B_2\\ 0 & * & * \\ 0 & * & *
 \end{pmatrix}
\begin{pmatrix} \widehat{y}_{0}(x) \\ \widehat{y}_{1}(x) \\ \widehat{y}_{2}(x) \end{pmatrix}\!,
\end{align*}
where $B_1 = (\alpha_1-\beta_1)(\alpha_1-\beta_2)/ \{ \alpha_1(\alpha_1-\alpha_2) \}$ and $B_2 = (\alpha_2-\beta_1)(\alpha_2-\beta_2)/ \{ \alpha_2(\alpha_2-\alpha_1) \} $, and we have
\begin{align*}
& \begin{pmatrix}
 \widehat{y}_1(qx)\\ \widehat{y}_2(qx)
 \end{pmatrix}
 = \biggl(\overline{F}_{\infty}+\frac{\overline{F}_1}{1- \alpha_1 x}+\frac{\overline{F}_2}{1- \alpha_2 x} \biggr)
 \begin{pmatrix}
 \widehat{y}_1(x)\\ \widehat{y}_2(x)
 \end{pmatrix}\!,
 \\[1mm]
& \overline{F}_1 = \begin{pmatrix}
 B_1-1+q^{\lambda} & B_2\\ 0 & 0
 \end{pmatrix}\!, \qquad
\overline{F}_2 = \begin{pmatrix}
 0 & 0\\ B_1 & B_2-1+q^{\lambda}
 \end{pmatrix}\!, \qquad
 \overline{F}_{\infty} = \begin{pmatrix}
 1-B_1 & -B_2\\ -B_1 & 1-B_2
 \end{pmatrix}\!.
\end{align*}
We may regard the tuple $( \overline{F}_{\infty}; \overline{F}_1,\overline{F}_2)$ as the one obtained by the $q$-middle con\-vo\-lu\-tion $ mc_{\lambda} (\textbf{B}) $.

We apply Proposition~\ref{prop:g1g2single}.
Then a $q$-difference equation for the function $\widehat{y}_1(x) $ is written as
\begin{align}
&\biggl( x-\frac{q^{\lambda+1}}{\beta_1} \biggr) \biggl( x-\frac{q^{\lambda+1}}{\beta_2} \biggr) \widehat{y}_1(x/q) + \frac{\alpha_1\alpha_2}{\beta_1\beta_2} \biggl( x-\frac{1}{\alpha_1} \biggr) \biggl(x-\frac{q}{\alpha_2} \biggr) \widehat{y}_1(qx) \nonumber
\\[1mm]
 &\qquad{}- \biggl[ \biggl( \frac{\alpha_1\alpha_2}{\beta_1\beta_2}\!+1 \biggr) x^2-\biggl\{ q \biggl( \frac{1}{\beta_1}\!+\frac{1}{\beta_2}\biggr)\! +q^{\lambda}\frac{q\alpha_1\!+\alpha_2}{\beta_1\beta_2} \biggr\} x\!+\frac{q^{\lambda+1}(1+q)}{\beta_1\beta_2} \biggr] \widehat{y}_1(x) = 0.
\label{eq:g1eqdeg2k}
\end{align}
Recall that the variant of $q$-hypergeometric equation of degree two was introduced in~\cite{HMST}, and it is written as
\begin{align}\label{eq:varqhgdeg2}
\begin{split}
& \big(x-q^{h_1 +1/2} t_1\big) \big(x - q^{h_2 +1/2} t_2\big) g(x/q)
+ q^{k_1 +k_2} \big(x - q^{l_1-1/2}t_1 \big) \big(x - q^{l_2 -1/2} t_2\big) g(q x)
\\
&\qquad{}-\big[ \big(q^{k_1} +q^{k_2} \big) x^2 +E x + p \big( q^{1/2}+ q^{-1/2}\big) t_1 t_2 \big] g(x) =0,
\\
& p= q^{(h_1 +h_2 + l_1 + l_2 +k_1 +k_2 )/2}, \qquad
E= -p \big\{ \big(q^{- h_2}+q^{-l_2}\big)t_1 + \big(q^{- h_1}+ q^{- l_1}\big) t_2 \big\} .
\end{split}
\end{align}
Equation~\eqref{eq:g1eqdeg2k} is a special case of the variant of $q$-hypergeometric equation of degree two with the constraint $k_2=0$.
However, equation~\eqref{eq:g1eqdeg2k} recovers the variant of $q$-hypergeometric equation of degree two by setting $g (x) = x ^{-k_2} \widehat{y}_1(x) $.
Namely, we obtain the following proposition directly.
\begin{Proposition} \label{prop:paramrel}
Assume that $\widehat{y}_1(x) $ satisfies equation~\eqref{eq:g1eqdeg2k}.
Set $g (x) = x ^{-k_2} \widehat{y}_1(x) $ and
\begin{align}
\label{eq:propparam}
\begin{split}
& \alpha_1 = \frac{q^{- l_1 + 1/2}}{t_1}, \qquad
\alpha_2 = \frac{q^{-l_2 +3/2}}{t_2},\qquad
q^{\lambda}=q^{(h_1+h_2-l_1-l_2-k_1 +k_2 +1)/2}
\\
& \beta_1 = \frac{q^{ (-h_1+h_2-l_1-l_2-k_1 +k_2)/2 +1}}{ t_1}, \qquad
\beta_2 = \frac{q^{(h_1- h_2-l_1-l_2-k_1 +k_2)/2 +1}}{ t_2}.
\end{split}
\end{align}
Then $g(x)$ satisfies the variant of $q$-hypergeometric equation of degree two given in equation~\eqref{eq:varqhgdeg2}.
Note that $\beta_1 = q^{\lambda -h_1 +1/2} /t_1 $ and $\beta_2 = q^{\lambda - h_2 + 1/2} /t_2 $.
\end{Proposition}
Similarly, a $q$-difference equation for the function $\widehat{y}_2(x) $ is written as{\samepage
\begin{align*}
&\biggl( x-\frac{q^{\lambda+1}}{\beta_1} \biggr) \biggl( x-\frac{q^{\lambda+1}}{\beta_2} \biggr)\widehat{y}_2(x/q) + \frac{\alpha_1\alpha_2}{\beta_1\beta_2} \biggl( x-\frac{q}{\alpha_1} \biggr) \biggl( x-\frac{1}{\alpha_2} \biggr) \widehat{y}_2(qx)
\\[1mm]
&\qquad{}- \biggl[\biggl( \frac{\alpha_1\alpha_2}{\beta_1\beta_2}+1 \biggr) x^2- \biggl\{ q \biggl( \frac{1}{\beta_1}+\frac{1}{\beta_2} \biggr) +q^{\lambda}\frac{\alpha_1+q\alpha_2}{\beta_1\beta_2} \biggr\} x+\frac{q^{\lambda+1}(1+q)}{\beta_1\beta_2} \biggr] \widehat{y}_2(x) = 0,
\end{align*}
and it is also a special case of the variant of $q$-hypergeometric equation of degree two.}

We obtained the formal $q$-difference equation for $\widehat{y}_1(x) $ corresponding to the $q$-middle convolution for the tuples of the matrices in equation~\eqref{eq:g1eqdeg2k}.
We now discuss the convergence of the $q$-integral representations obtained by the $q$-middle convolution and the actual $q$-difference equations which the $q$-integral representation satisfies.
In the case $\mu =0$, the assumptions of Corollary~\ref{cor:qcint} and Theorem~\ref{thm:qcint} do not hold true.
Therefore, we go back to Proposition~\ref{prop:convKL}.
We apply Proposition~\ref{prop:convKL} for the function $y(x)= x^{\mu} (\alpha_1 x, \alpha_2 x;q)_{\infty}/(\beta_1 x, \beta_2 x;q)_{\infty}$ with the condition $\mu =0$.
Then the functions
\begin{align*}
& \widehat{y}_{i}^{[K,L]}(x) = (1-q) \sum_{n=K}^{L} s \frac{P_{\lambda}(x, s)}{s-b_{i}} \frac{(\alpha_1 s, \alpha_2 s;q)_{\infty}}{(\beta_1 s, \beta_2 s;q)_{\infty}} |_{s= q^n \xi}, \qquad
i=0,1,2,
\\
& b_0=0, \qquad b_1 = 1/\alpha_1, \qquad b_2 = 1/\alpha_2
\end{align*}
satisfy
\begin{align*}
& \widehat{y}_0^{\,[K,L]}(qx) = \widehat{y}_0^{\,[K,L]}(x) - B_1\widehat{y}_1^{\,[K,L]}(x) - B_2\widehat{y}_2^{\,[K,L]}(x) + (1-q)Q^{[K,L]}(x), \nonumber\\
& \widehat{y}_1^{\,[K,L]}(qx) = \frac{\alpha_1(B_1-1)x+q^{\lambda}}{1-\alpha_1 x}\widehat{y}_1^{\,[K,L]}(x) + \frac{\alpha_1 B_2 x}{1-\alpha_1 x}\widehat{y}_2^{\,[K,L]}(x) - \frac{(1-q)\alpha_1 x}{1-\alpha_1 x}Q^{[K,L]}(x), \nonumber\\
& \widehat{y}_2^{\,[K,L]}(qx) = \frac{\alpha_2 B_1 x}{1-\alpha_2 x}\widehat{y}_1^{\,[K,L]}(x) + \frac{\alpha_2(B_2-1)x+q^{\lambda}}{1-\alpha_2 x}\widehat{y}_2^{\,[K,L]}(x) - \frac{(1-q)\alpha_2 x}{1-\alpha_2 x}Q^{[K,L]}(x),
\end{align*}
where
\begin{equation*}
Q^{[K,L]}(x)=P_{\lambda}\big(x, q^{K-1} \xi\big) y \big(q^{K} \xi\big) - P_{\lambda}\big(x, q^{L} \xi\big) y \big(q^{L+1} \xi\big).
\end{equation*}
Note that the convergence theorem in Theorem~\ref{thm:qcintqJH} for the functions $\widehat{y}_i^{\,[K,L]}(x)$, $i=0,1,2$, is not applicable, because of the condition $\mu =0$.
We apply Proposition~\ref{prop:g1g2single} to obtain the $q$-difference equation which the function $\widehat{y}_1^{\,[K,L]}(x) $ satisfies.
Then we have
\begin{align}
&\biggl( x-\frac{q^{\lambda+1}}{\beta_1} \biggr) \biggl( x-\frac{q^{\lambda+1}}{\beta_2} \biggr) \widehat{y}_1^{\,[K,L]}(x/q)
+ \frac{\alpha_1\alpha_2}{\beta_1\beta_2} \biggl( x-\frac{1}{\alpha_1} \biggr) \biggl( x-\frac{q}{\alpha_2} \biggr) \widehat{y}_1^{\,[K,L]}(qx) \nonumber
\\[1mm]
 &\qquad{}- \biggl[ \biggl( \frac{\alpha_1\alpha_2}{\beta_1\beta_2}+1 \biggr) x^2- \biggl\{ q \biggl( \frac{1}{\beta_1}+\frac{1}{\beta_2} \biggr) +q^{\lambda}\frac{q\alpha_1+\alpha_2}{\beta_1\beta_2} \biggr\} x+\frac{q^{\lambda+1}(1+q)}{\beta_1\beta_2} \biggr] \widehat{y}_1^{\,[K,L]}(x) \nonumber
 \\[1mm]
 &\qquad{}+ (1-q)q\frac{\alpha_1}{\beta_1\beta_2}x \biggl\{ \biggl( \frac{\alpha_2}{q}x-q^{\lambda} \biggr) Q^{[K,L]}(x/q)+ \biggl( 1-\frac{\alpha_2}{q}x \biggr) Q^{[K,L]}(x) \biggr\} = 0,
\label{eq:g1KL}
\end{align}
and the non-homogeneous term is written as
\begin{align}
& (1-q)q\frac{\alpha_1}{\beta_1\beta_2}x \biggl\{ \biggl( \frac{\alpha_2}{q}x-q^{\lambda} \biggr) Q^{[K,L]}(x/q)+ \biggl( 1-\frac{\alpha_2}{q}x \biggr) Q^{[K,L]}(x) \biggr\} \nonumber
\\
&\quad= q(1-q)\big(1-q^{\lambda}\big)\frac{\alpha_1}{\beta_1\beta_2}x \nonumber
\\
&\quad\phantom{=} \times \biggl( \frac{\big(q^{\lambda+K+1}\xi/x, q^K\xi\alpha_1, q^{K-1}\xi\alpha_2 ;q\big)_{\infty}}{\big(q^K\xi/x, q^K\xi\beta_1, q^K\xi\beta_2 ;q\big)_{\infty}} - \frac{\big(q^{\lambda+L+2}\xi/x, q^{L+1}\xi\alpha_1, q^L\xi\alpha_2 ;q\big)_{\infty}}{\big(q^{L+1}\xi/x, q^{L+1}\xi\beta_1, q^{L+1}\xi\beta_2 ;q\big)_{\infty}} \biggr).
\label{eq:nonhomKL2}
\end{align}
The function $\widehat{y}_1^{\,[K,L]}(x)$ is written as
\begin{align*}
 \widehat{y}_1^{\,[K,L]}(x) = (q-1)\alpha_1\sum_{n=K}^{L}c_n, \qquad
 c_n = q^n\xi\frac{\big(q^{\lambda+n+1}\xi / x, q^{n+1} \xi\alpha_1, q^n \xi\alpha_2 ; q\big)_{\infty}}{\big(q^{n+1}\xi / x, q^n \xi\beta_1, q^n \xi\beta_2 ; q\big)_{\infty}} .
\end{align*}
Since $c_{n+1}/c_n \to q$, $n\to +\infty$, and $c_{-(n+1)}/c_{-n} \to q^{\lambda} \alpha_1 \alpha_2 /(\beta_1 \beta_2)$, $n\to +\infty$, the function $\widehat{y}_1^{\,[K,L]}(x) $ converges as $K \to -\infty $ and $L \to +\infty $ if $\big|q^{\lambda} \alpha_1 \alpha_2 /(\beta_1 \beta_2 )\big| <1 $.
Write
\begin{equation}
\widehat{y}_1 (x)= \lim_{K \to -\infty \atop{ L \to +\infty}} \widehat{y}_1^{\,[K,L]}(x) = (q-1)\alpha_1\sum_{n=-\infty}^{+\infty} q^n\xi\frac{\big(q^{\lambda+n+1}\xi / x, q^{n+1} \xi\alpha_1, q^n \xi\alpha_2 ; q\big)_{\infty}}{\big(q^{n+1}\xi / x, q^n \xi\beta_1, q^n \xi\beta_2 ; q\big)_{\infty}} . \label{eq:hy1xivqhg2}
\end{equation}

We investigate the limit of the non-homogeneous term in equation~\eqref{eq:nonhomKL2} as $K \to -\infty$ and $L \to +\infty$.
We have
\begin{equation*}
\lim_{L \to +\infty} \frac{\big(q^{\lambda+L+2}\xi/x, q^{L+1}\xi\alpha_1, q^L\xi\alpha_2 ;q\big)_{\infty}}{\big(q^{L+1}\xi/x, q^{L+1}\xi\beta_1, q^{L+1}\xi\beta_2 ;q\big)_{\infty}} =1.
\end{equation*}
Set $\vartheta_q (t) = (t, q/t, q ;q)_{\infty}$.
It follows from the identity $\vartheta_q (q a x)/ \vartheta_q (q b x) = (b/a) \vartheta_q (a x)/ \vartheta_q (b x) $ that
\begin{align}
& \frac{\big(q^{\lambda+K+1}\xi/x, q^K\xi\alpha_1, q^{K-1}\xi\alpha_2 ;q\big)_{\infty}}{\big(q^K\xi/x, q^K\xi\beta_1, q^K\xi\beta_2 ;q\big)_{\infty}} \nonumber
\\
&\qquad{} = \frac{\vartheta_q\big(q^{\lambda+K+1}\xi/x\big) \vartheta_q\big(q^K\xi\alpha_1\big) \vartheta_q\big(q^{K-1}\xi\alpha_2\big)}{\vartheta_q\big(q^K\xi/x\big) \vartheta_q\big(q^K\xi\beta_1\big) \vartheta_q\big(q^K\xi\beta_2\big)} \frac{\big(q^{1-K}x/\xi, q^{1-K}/(\xi\beta_1), q^{1-K}/(\xi\beta_2) ;q\big)_{\infty}}{\big(q^{-\lambda-K}\xi/x, q^{1-K}/(\xi\alpha_1), q^{2-K}/(\xi\alpha_2) ;q\big)_{\infty}}\! \nonumber \\
&\qquad{}= \biggl( q^{\lambda}\frac{\alpha_1\alpha_2}{\beta_1\beta_2} \biggr)^{-K} \frac{\vartheta_q\big(q^{\lambda+1}\xi/x\big) \vartheta_q(\xi\alpha_1) \vartheta_q\big(q^{-1}\xi\alpha_2\big)}{\vartheta_q(\xi/x) \vartheta_q(\xi\beta_1) \vartheta_q(\xi\beta_2)} \nonumber
\\
&\qquad\phantom{=} \times \frac{\big(q^{1-K}x/\xi, q^{1-K}/(\xi\beta_1), q^{1-K}/(\xi\beta_2) ;q\big)_{\infty}}{\big(q^{-\lambda-K}\xi/x, q^{1-K}/(\xi\alpha_1), q^{2-K}/(\xi\alpha_2) ;q\big)_{\infty}}.
\label{eq:thetaq}
\end{align}
Therefore, if $\big| q^{\lambda} \alpha_1 \alpha_2 /( \beta_1 \beta_2 ) \big| <1$, then equation~\eqref{eq:nonhomKL2} tends to
\begin{equation}
-q(1-q)\big(1-q^{\lambda}\big)\frac{\alpha_1}{\beta_1\beta_2}x \label{eq:nonhomlimKL2}
\end{equation}
as $K \to -\infty $ and $L \to +\infty $.
Hence we obtain the following proposition.

\begin{Proposition} \label{prop:y1nonhom}
If $| q^{\lambda} \alpha_1 \alpha_2 /( \beta_1 \beta_2 ) | <1$, the function $\widehat{y}_1(x) $ in equation~\eqref{eq:hy1xivqhg2} satisfies
\begin{align}
&\biggl( x-\frac{q^{\lambda+1}}{\beta_1} \biggr) \biggl( x-\frac{q^{\lambda+1}}{\beta_2} \biggr) \widehat{y}_1(x/q) + \frac{\alpha_1\alpha_2}{\beta_1\beta_2} \biggl( x-\frac{1}{\alpha_1} \biggr) \biggl( x-\frac{q}{\alpha_2} \biggr) \widehat{y}_1(qx) \nonumber
\\
 &\qquad{}- \biggl[ \biggl( \frac{\alpha_1\alpha_2}{\beta_1\beta_2}+1 \biggr) x^2- \biggl\{ q \biggl( \frac{1}{\beta_1}+\frac{1}{\beta_2} \biggr) +q^{\lambda}\frac{q\alpha_1+\alpha_2}{\beta_1\beta_2} \biggr\} x+\frac{q^{\lambda+1}(1+q)}{\beta_1\beta_2} \biggr] \widehat{y}_1(x) \nonumber
 \\
 &\qquad{}-q(1-q)\big(1-q^{\lambda}\big)\frac{\alpha_1}{\beta_1\beta_2}x = 0 .
\label{eq:y1nonhom}
\end{align}
\end{Proposition}

Equation~\eqref{eq:y1nonhom} is an non-homogeneous extension of equation~\eqref{eq:g1eqdeg2k}.
We investigate a relationship with the third order difference equation in equation~\eqref{eq:hy1q3}.
Let $T_x$ be the operator such that $T_x f(x) = f(qx)$.
If $q^{\mu} =1$, then equation~\eqref{eq:hy1q3} is factorized as
\begin{align*}
& (T_x -q) \big[ (1 - q x \alpha_1) (1 - x \alpha_2) T_x^2 \nonumber
\\
& \qquad {} - \big\{ q ( \beta_1 \beta_2 +\alpha_1 \alpha_2 ) x^2 - (q(\beta_1+\beta_2) + q^{\lambda} (q \alpha_1+\alpha_2) ) x + q^{\lambda} (q + 1) \big\} T_x \nonumber
\\
& \qquad {} + q \big(\beta_1 x-q^{\lambda}\big) \big(\beta_2 x-q^{\lambda}\big) \big] \widehat{y}_{1}( x) =0.
\end{align*}
It follows from the relation $(T_x -q )x =0 $ that, if $\widehat{y}_1(x) $ satisfies equation~\eqref{eq:y1nonhom}, then it also satisfies equation~\eqref{eq:hy1q3} with the condition $q^{\mu} =1 $.

If $\xi= 1 /\alpha_1$, $\xi= 1 /\alpha_2 $ or $\xi=q^{-\lambda} x$, then
\begin{align*}
& \frac{\big(q^{\lambda+K+1}\xi/x, q^K\xi\alpha_1, q^{K-1}\xi\alpha_2 ;q\big)_{\infty}}{\big(q^K\xi/x, q^K\xi\beta_1, q^K\xi\beta_2 ;q\big)_{\infty}} =0
\end{align*}
for any negative integer $K$, and equation~\eqref{eq:nonhomKL2} tends to equation~\eqref{eq:nonhomlimKL2} as $K \to -\infty $ and $L \to +\infty $ without the condition for $q^{\lambda} \alpha_1 \alpha_2 /(\beta_1 \beta_2 )$.

We substitute $\xi= 1 /\alpha_1$, $\xi= 1 /\alpha_2 $ or $\xi=q^{-\lambda} x$ in $\widehat{y}_1(x) $.
If $\xi= 1 /\alpha_1 $, then it follows from a~similar calculation to equation~\eqref{eq:hy0sum} that
\begin{align}
\widehat{y}_1(x) &= (q-1)\alpha_1 \sum_{n=0}^{\infty} \frac{q^n}{\alpha_1} \frac{\big(q^{\lambda+n+1} /(\alpha_1 x ), q^{n+1}, q^n \alpha_2 /\alpha_1 ; q\big)_{\infty}}{\big(q^{n+1}/ (\alpha_1 x ), q^n \beta_1 /\alpha_1, q^n \beta_2 /\alpha_1 ; q\big)_{\infty}} \nonumber
\\
& = (q-1)\frac{\big(q^{\lambda+1}/(\alpha_1 x), \alpha_2/\alpha_1, q;q\big)_{\infty}}{\big(q/(\alpha_1 x), \beta_1/\alpha_1, \beta_2/\alpha_1 ;q\big)_{\infty}} \,_3\phi_2 \biggl(\!\! \begin{array}{c} q/(\alpha_1 x), \beta_1/\alpha_1, \beta_2/\alpha_1\\ q^{\lambda+1}/(\alpha_1 x), \alpha_2/\alpha_1 \end{array}\! ;q,q \biggr) .
\label{eq:xi1al1}
\end{align}
Note that the condition for $q^{\lambda} \alpha_1 \alpha_2 /(\beta_1 \beta_2 )$ is not necessary.
If $\xi= 1 /\alpha_2 $, then
\begin{align}
 \widehat{y}_1(x)={}& (q-1)q\frac{\alpha_1}{\alpha_2}\frac{\big(q^{\lambda+2}/(\alpha_2 x), q^2\alpha_1/\alpha_2, q ;q\big)_{\infty}}{\big(q^2/(\alpha_2 x), q\beta_1/\alpha_2, q\beta_2/\alpha_2 ;q\big)_{\infty}} \nonumber
 \\
& \times {}_3\phi_2 \biggl(\!\! \begin{array}{c} q^2/(\alpha_2 x), q\beta_1/\alpha_2, q\beta_2/\alpha_2\\
 q^{\lambda+2}/(\alpha_2 x), q^2\alpha_1/\alpha_2 \end{array}\! ;q,q \biggr).
\label{eq:xi1al2}
\end{align}
If $\xi=q^{-\lambda}x$, then
\begin{align}
 &\widehat{y}_1(x)\!= (q\!-\!1)q^{-\lambda}\alpha_1 x\frac{\big(q^{1-\lambda}\alpha_1 x, q^{-\lambda}\alpha_2 x, q ;q\big)_{\infty}}{\big( q^{-\lambda}\beta_1 x, q^{-\lambda}\beta_2 x, q^{1 -\lambda} ;q\big)_{\infty}} \,{}_3\phi_2 \biggl(\!\! \begin{array}{c} q^{-\lambda}\beta_1 x, q^{-\lambda}\beta_2 x, q^{1-\lambda} \\
 q^{1 -\lambda}\alpha_1 x, q^{-\lambda}\alpha_2 x \end{array}\!\! ;q,q \biggr) .\!\!\!\!
\label{eq:xiqlax}
\end{align}
By Proposition~\ref{prop:y1nonhom}, we obtain the following assertion.
\begin{Proposition}
The functions in equations~\eqref{eq:xi1al1}, \eqref{eq:xi1al2}, \eqref{eq:xiqlax} satisfy equation~\eqref{eq:y1nonhom}.
\end{Proposition}

To obtain results corresponding to the specialization $\xi= 1/\beta_1$, $\xi= 1/\beta_2$ and $\xi= x $, we replace the functions with
\begin{align}
& P_{\lambda}(x, s) = (x/s)^{\lambda} \frac{(x/s;q)_{\infty}}{\big(q^{-\lambda} x/s ;q\big)_{\infty}}, \qquad y(x)=x^{\mu '} \frac{(q/(\beta_1 x), q/(\beta_2 x) ;q)_{\infty}}{(q/(\alpha_1 x), q/(\alpha_2 x) ;q)_{\infty}} \label{eq:Ply}
\end{align}
with the condition $ q^{\mu '} \alpha_1 \alpha_2 /(\beta_1 \beta_2 ) =1$.
Then the function $y(x) $ also satisfies the $q$-difference equation $y(qx)=B(x)y(x) $, where $B(x)$ is given as equation~\eqref{eq:gqxBxgxqJPN2} with the condition $\mu =0$.
The function $\widehat{y}_1^{\,[K,L]}(x) $ is written as
\begin{align*}
 \widehat{y}_1^{\,[K,L]}(x) = (1-q) x^{\lambda} \sum_{n=K}^{L} (q^n\xi )^{\mu' -\lambda} \frac{\big(x q^{-n} / \xi, q^{1-n} /( \beta_1 \xi ), q^{1-n} /( \beta_2 \xi ) ; q\big)_{\infty}}{\big(x q^{-n-\lambda} / \xi, q^{-n} / (\alpha_1 \xi),q^{1-n} / (\alpha_2 \xi) ; q\big)_{\infty}},
\end{align*}
and it converges as $K \to -\infty $ and $L \to +\infty $ if $\big|q^{\lambda} \alpha_1 \alpha_2 /(\beta_1 \beta_2 )\big|<1 $.
Write the limit by $\widehat{y}_1 (x) $.
Note that the function $\widehat{y}_1^{\,[K,L]}(x)$ satisfies equation~\eqref{eq:g1KL}, where
\begin{align}
 Q^{[K,L]}(x)&=P_{\lambda}\big(x, q^{K-1} \xi\big) y\big(q^{K} \xi\big) - P_{\lambda}\big(x, q^{L} \xi\big) y\big(q^{L+1} \xi\big) \nonumber\\
 &= (q x )^{\lambda} \biggl\{ \big(q^{K} \xi\big)^{\mu ' -\lambda} \frac{\big(x q^{1-K}/\xi, q^{1-K}/(\beta_1 \xi ), q^{1-K}/(\beta_2 \xi ) ;q\big)_{\infty}}{\big(x q^{-\lambda +1-K} / \xi, q^{1-K}/(\alpha_1 \xi ), q^{1-K}/(\alpha_2 \xi ) ;q\big)_{\infty}} \nonumber \\
& \hphantom{= (q x )^{\lambda} \biggl\{} - \big(q^{L+1} \xi\big)^{\mu ' -\lambda} \frac{\big(x q^{-L}/\xi, q^{-L}/(\beta_1 \xi ), q^{-L}/(\beta_2 \xi ) ;q\big)_{\infty}}{\big(x q^{-\lambda -L} / \xi, q^{-L}/(\alpha_1 \xi ), q^{-L}/(\alpha_2 \xi ) ;q\big)_{\infty}} \biggr\} .
\label{eq:QKL}
\end{align}
The non-homogeneous term in equation~\eqref{eq:g1KL} is written as
\begin{align}
& (1-q)q\frac{\alpha_1}{\beta_1\beta_2}x \biggl\{ \biggl( \frac{\alpha_2}{q}x-q^{\lambda} \biggr) Q^{[K,L]}(x/q)+ \biggl( 1-\frac{\alpha_2}{q}x \biggr) Q^{[K,L]}(x) \biggr\}\nonumber \\
&\qquad= (1-q)\big(1-q^{\lambda}\big)\frac{\alpha_1 \alpha_2}{\beta_1\beta_2}x^{\lambda +2} \biggl( (q^{K} \xi)^{\mu ' -\lambda} \frac{(x q^{1-K}/\xi, q^{1-K}/(\beta_1 \xi ), q^{1-K}/(\beta_2 \xi ) ;q)_{\infty}}{(x q^{-\lambda -K} / \xi, q^{1-K}/(\alpha_1 \xi ), q^{2-K}/(\alpha_2 \xi ) ;q)_{\infty}} \nonumber \\
&\qquad\phantom{=} - \big(q^{L+1} \xi\big)^{\mu ' -\lambda} \frac{\big(x q^{-L}/\xi, q^{-L}/(\beta_1 \xi ), q^{-L}/(\beta_2 \xi ) ;q\big)_{\infty}}{\big(x q^{-\lambda -1-L} / \xi, q^{-L}/(\alpha_1 \xi ), q^{1-L}/(\alpha_2 \xi ) ;q\big)_{\infty}} \biggr).
\label{eq:nonhomKL22}
\end{align}
We investigate the limit of equation~\eqref{eq:nonhomKL22} as $K \to -\infty $ and $L \to +\infty $.
We have
\begin{align*}
& \lim_{K \to -\infty} \big(q^{K} \xi\big)^{\mu ' -\lambda} \frac{\big(x q^{1-K}/\xi, q^{1-K}/(\beta_1 \xi ), q^{1-K}/(\beta_2 \xi ) ;q\big)_{\infty}}{\big(x q^{-\lambda -K} / \xi, q^{1-K}/(\alpha_1 \xi ), q^{2-K}/(\alpha_2 \xi ) ;q\big)_{\infty}}
\\
& \qquad{}= \lim_{K \to -\infty} \xi ^{\mu ' -\lambda} \biggl( \frac{\beta_1 \beta_2}{\alpha_1 \alpha_2} q^{ -\lambda} \biggr)^K =0,
\end{align*}
under the condition $\big| q^{\lambda} \alpha_1 \alpha_2 /(\beta_1 \beta_2 ) \big| <1$.
If $\xi=1 /\beta_1$, $\xi=1/ \beta_2$ or $\xi= x$, then
\begin{align*}
& \frac{\big(x q^{-L}/\xi, q^{-L}/(\beta_1 \xi ), q^{-L}/(\beta_2 \xi ) ;q\big)_{\infty}}{\big(x q^{-\lambda -1-L} / \xi, q^{-L}/(\alpha_1 \xi ), q^{1-L}/(\alpha_2 \xi ) ;q\big)_{\infty}} =0
\end{align*}
for any positive integer $L$, and equation~\eqref{eq:nonhomKL22} tends to $0$ as $K \to -\infty $ and $L \to +\infty $ under the condition $\big| q^{\lambda} \alpha_1 \alpha_2 /(\beta_1 \beta_2 ) \big| <1 $.

If $\xi= 1 /\beta_1$, then
\begin{align}
 \widehat{y}_1(x) & = (1-q) x^{\lambda} \sum_{n=-\infty}^{0} (q^n/ \beta_1)^{\mu' -\lambda} \frac{\big(x \beta_1 q^{-n}, q^{1-n}, q^{1-n} \beta_1 / \beta_2 ; q\big)_{\infty}}{\big(x q^{-n-\lambda} \beta_1, q^{-n} \beta_1 / \alpha_1,q^{1-n} \beta_1 / \alpha_2 ; q\big)_{\infty}} \nonumber
 \\
& = (1-q) \beta_1^{\lambda -\mu '} x^{\lambda} \frac{( \beta_1 x, q, q \beta_1 / \beta_2 ; q)_{\infty}}{\big( q^{-\lambda} x \beta_1, \beta_1 / \alpha_1,q \beta_1 / \alpha_2 ; q\big)_{\infty}} \nonumber \\
& \hphantom{=}\times {}_3\phi_2 \biggl(\!\! \begin{array}{c} q^{-\lambda}\beta_1 x, \beta_1/\alpha_1, q\beta_1/\alpha_2\\
 \beta_1 x, q\beta_1/\beta_2 \end{array}\! ;q,q^{\lambda}\frac{\alpha_1\alpha_2}{\beta_1\beta_2} \biggr),
\label{eq:xi1be1}
\end{align}
where the convergence condition is given by $\big| q^{\lambda} \alpha_1 \alpha_2 /(\beta_1 \beta_2 ) \big| <1 $.
The case $\xi= 1 /\beta_2 $ is obtained from the case $\xi=1 /\beta_1 $ by replacing $\beta_1$ with $\beta_2$.

If $\xi=x$, then
\begin{align}
 \widehat{y}_1(x) & = (1-q) x^{\lambda} \sum_{n=-\infty}^{-1} (q^n x )^{\mu' -\lambda} \frac{\big( q^{-n}, q^{1-n} /( \beta_1 x ), q^{1-n} /( \beta_2 x ) ; q\big)_{\infty}}{\big(q^{-n-\lambda}, q^{-n} / (\alpha_1 x ),q^{1-n} / (\alpha_2 x ) ; q\big)_{\infty}} \nonumber
 \\
& = (1-q) q^{\lambda}\frac{\alpha_1\alpha_2}{\beta_1\beta_2} x^{\mu '} \frac{\big( q^{2} /( \beta_1 x ), q^{2} /( \beta_2 x ), q ; q\big)_{\infty}}{\big(q / (\alpha_1 x ),q^{2} / (\alpha_2 x ), q^{1-\lambda} ; q\big)_{\infty}} \nonumber
\\
& \hphantom{=}\times {}_3\phi_2 \biggl(\!\! \begin{array}{c} q/(\alpha_1 x), q^2/(\alpha_2 x), q^{1-\lambda}\\
 q^2/(\beta_1 x), q^2/(\beta_2 x) \end{array}\! ;q,q^{\lambda}\frac{\alpha_1\alpha_2}{\beta_1\beta_2} \biggr),
\label{eq:xi1x}
\end{align}
where the convergence condition is given by $\big| q^{\lambda} \alpha_1 \alpha_2 /(\beta_1 \beta_2 ) \big| <1 $.

Since the non-homogeneous term in equation~\eqref{eq:g1KL} vanishes in these cases, we have
\begin{Proposition}
If $\big| q^{\lambda} \alpha_1 \alpha_2 /(\beta_1 \beta_2 ) \big| <1 $, then the functions in equations~\eqref{eq:xi1be1}, \eqref{eq:xi1x} satisfy equation~\eqref{eq:g1eqdeg2k}.
\end{Proposition}

We replace the parameters by using Proposition~\ref{prop:paramrel} to fit results on $q$-integrals with the variant of $q$-hypergeometric equation of degree two given in equation~\eqref{eq:varqhgdeg2}.
We recall Proposition~\ref{prop:paramrel}.
If the function $\widehat{y}_1(x) $ is a solution to equation~\eqref{eq:g1eqdeg2k}, then the function $g (x) = x ^{-k_2} \widehat{y}_1(x) $ satisfies equation~\eqref{eq:varqhgdeg2} by setting the parameters as equation~\eqref{eq:propparam}.
We apply it to equations~\eqref{eq:xi1be1}, \eqref{eq:xi1x} and we multiply relevant constants.
Then we have the following theorem.
\begin{Theorem} \label{thm:deg2-1}
Set $\lambda = (h_1+h_2-l_1-l_2-k_1 +k_2 +1)/2 $ and assume $\lambda +k_1 -k_2 >0 $.
\begin{enumerate}
\item [$(i)$] Let $i \in \{ 1,2 \}$ and set $i'=3-i$.
The function
\begin{align*}
 g(x) & = x^{\lambda -k_2} \frac{\big( q^{\lambda -h_i +1/2} x /t_i ; q\big)_{\infty}}{\big( q^{ -h_i +1/2} x /t_i ; q\big)_{\infty}}
 \\
& \hphantom{=}\times {}_3\phi_2 \biggl(\!\! \begin{array}{c} q^{ -h_i +1/2} x /t_i, q^{\lambda -h_i + l_i},q^{\lambda -h_i + l_{i'}} t_{i'} /t_i \\
 q^{\lambda -h_i +1/2} x /t_i, q^{ 1 -h_i + h_{i'}} t_{i'} / t_i \end{array}\! ;q,q^{\lambda +k_1 -k_2} \biggr)
\end{align*}
satisfies the variant of $q$-hypergeometric equation of degree two given in equation~\eqref{eq:varqhgdeg2}.

\item [$(ii)$] The function
\begin{align*}
 g(x) & = x^{ -k_1} \frac{\big( q^{- \lambda +h_1 +3/2} t_1 / x, q^{ -\lambda + h_2 + 3/2}t_2 / x ; q\big)_{\infty}}{\big( q^{l_1 + 1/2} t_1 / x,q^{l_2 +1 /2} t_2 / x ; q\big)_{\infty}} \nonumber
 \\
& \hphantom{=}\times {}_3\phi_2 \biggl(\!\! \begin{array}{c} q^{l_1 + 1/2} t_1 / x,q^{l_2 +1 /2} t_2 / x, q^{1-\lambda} \\
 q^{- \lambda +h_1 +3/2} t_1 / x, q^{ -\lambda + h_2 + 3/2}t_2 / x \end{array}\! ;q,q^{\lambda + k_1 -k_2} \biggr)
\end{align*}
satisfies the variant of $q$-hypergeometric equation of degree two given in equation~\eqref{eq:varqhgdeg2}.
\end{enumerate}
\end{Theorem}

We replace the parameters in Proposition~\ref{prop:y1nonhom} by using Proposition~\ref{prop:paramrel}.
By setting $g (x) = x ^{-k_2} \widehat{y}_1(x) /\{ \alpha_1 (q-1) \} $, equation~\eqref{eq:y1nonhom} is replaced with
\begin{align}
& \big(x-q^{h_1 +1/2} t_1\big) \big(x - q^{h_2 +1/2} t_2\big) g(x/q) + q^{k_1 +k_2} \big(x - q^{l_1-1/2}t_1\big) \big(x - q^{l_2 -1/2} t_2\big) g(q x) \nonumber
\\
&\qquad{} -\big[ \big(q^{k_1} +q^{k_2} \big) x^2 +E x + p \big( q^{1/2}+ q^{-1/2}\big) t_1 t_2 \big] g(x) \nonumber
\\
&\qquad{} + \big(1-q^{(h_1+h_2-l_1-l_2-k_1 +k_2 +1)/2}\big) q^{l_1 + l_2 + k_1 -1} t_1 t_2 x^{1-k_2} =0, \nonumber
\\
& p= q^{(h_1 +h_2 + l_1 + l_2 +k_1 +k_2 )/2}, \qquad
E= -p \big\{ \big(q^{- h_2}+q^{-l_2}\big)t_1 + \big(q^{- h_1}+ q^{- l_1}\big) t_2 \big\},
\label{eq:varqhgdeg2nonhom}
\end{align}
which is a non-homogeneous version of the variant of $q$-hypergeometric equation of degree two.
Equation~\eqref{eq:hy1xivqhg2} is replaced with
\begin{equation}
g (x)= x ^{-k_2} \sum_{n=-\infty}^{+\infty} q^n\xi\frac{\big(q^{n + \lambda +1} \xi / x, q^{n- l_1 + 3/2} \xi /t_1, q^{n-l_2 +3/2} \xi /t_2 ; q\big)_{\infty}}{\big(q^{n+1} \xi / x, q^{n + \lambda -h_1 +1/2} \xi /t_1, q^{n + \lambda -h_2 +1/2} \xi /t_2 ; q\big)_{\infty}}, \label{eq:gxivqhg2}
\end{equation}
where $\lambda = (h_1+h_2-l_1-l_2-k_1 +k_2 +1)/2 $.
Then it follows from Proposition~\ref{prop:y1nonhom} that, if $\lambda +k_1 -k_2 >0 $, then the function $g(x)$ in equation~\eqref{eq:gxivqhg2} is a solution to equation~\eqref{eq:varqhgdeg2nonhom}.
By~rewriting equations~\eqref{eq:xi1al1}--\eqref{eq:xiqlax}, we have the following theorem.
\begin{Theorem} \label{thm:deg2-2}
Set $\lambda = (h_1+h_2-l_1-l_2-k_1 +k_2 +1)/2 $.
\begin{enumerate}
\item [$(i)$] Let $i \in \{ 1,2 \}$ and set $i'=3-i$.
The function
\begin{align*}
& g_i(x) = x ^{-k_2} q^{l_i - 1/2}t_i \frac{\big(q^{\lambda + l_i + 1/2}t_i / x, q^{l_i -l_{i'} +1} t_i / t_{i'}, q;q\big)_{\infty}}{\big(q^{ l_i + 1/2}t_i / x, q^{\lambda -h_i + l_i}, q^{\lambda - h_{i'} + l_i} t_i /t_{i'} ;q\big)_{\infty}} \nonumber
\\
& \hphantom{=}\times {}_3\phi_2 \biggl(\!\! \begin{array}{c} q^{ l_i + 1/2}t_i / x, q^{\lambda -h_i + l_i}, q^{\lambda - h_{i'} + l_i} t_i / t_{i'} \\ q^{\lambda + l_i + 1/2} t_i / x, q^{ l_i -l_{i'} +1} t_i / t_{i'} \end{array}\! ;q,q \biggr)
\end{align*}
satisfies the non-homogeneous version of the variant of $q$-hypergeometric equation of degree two given in equation~\eqref{eq:varqhgdeg2nonhom}.

\item [$(ii)$] The function
\begin{align*}
 g_0 (x) &= x ^{1-k_2} q^{-\lambda} \frac{\big(q^{-\lambda - l_1 + 3/2} x/ t_1, q^{-\lambda -l_2 +3/2} x / t_2, q ;q\big)_{\infty}}{\big( q^{-h_1 +1/2} x /t_1, q^{ - h_2 + 1/2} x /t_2, q^{1 -\lambda} ;q\big)_{\infty}} \nonumber\\
& \hphantom{=}\times {}_3\phi_2 \biggl(\!\! \begin{array}{c} q^{ -h_1 +1/2} x /t_1, q^{- h_2 + 1/2} x / t_2, q^{1-\lambda} \\
 q^{ -\lambda - l_1 + 3/2} x/ t_1, q^{-\lambda -l_2 +3/2} x /t_2 \end{array}\! ;q,q \biggr)
\end{align*}
satisfies the non-homogeneous version of the variant of $q$-hypergeometric equation of degree two given in equation~\eqref{eq:varqhgdeg2nonhom}.
\end{enumerate}
\end{Theorem}

By taking the difference of two solutions of the non-homogeneous equation, we obtain a~solution of the homogeneous equation.
\begin{Corollary}
Let $g_i (x)$, $i=0,1,2$, be the functions in Theorem~$\ref{thm:deg2-2}$.
For each $i,j \in \{ 0,1,2 \}$, the function $g_i (x) - g_j (x) $ satisfies the variant of the $q$-hypergeometric equation of degree two given in equation~\eqref{eq:varqhgdeg2}.
\end{Corollary}

\subsubsection[The case q\^\{lambda\} = q\^\{mu\} beta\_1 beta\_2 /(alpha\_1 alpha\_2)]{The case $\boldsymbol{q^{\lambda} = q^{\mu} \beta_1 \beta_2 /(\alpha_1 \alpha_2)}$} \label{sec:deg2p2}

We discuss the case $\mu \neq 0 $ and $\dim \mathcal{L} \geq 1$.
Recall that
\begin{align*}
\widehat{F}-\big(1-q^{\lambda}\big)I_3
 = \begin{pmatrix}
 B_0-\big(1-q^{\lambda}\big) & B_1 & B_2\\
 B_0 & B_1-\big(1-q^{\lambda}\big) & B_2\\
 B_0 & B_1 & B_2-\big(1-q^{\lambda}\big)
 \end{pmatrix}
\end{align*}
and $ B_0 + B_1 + B_2 =1-B_{\infty} = 1- q^{\mu} \beta_1\beta_2 /(\alpha_1\alpha_2 )$.
Therefore, if $q^{\lambda}=q^{\mu} \beta_1\beta_2 /(\alpha_1\alpha_2 )$, we can take a basis of the vector space $\operatorname{ker} \big(\widehat{F}-\big(1-q^{\lambda}\big)I_3\big)$ as $ ^t(1,1,1)$ and $\dim\ \mathcal{L} = 1$.
We take an invertible matrix as
\begin{equation}
P=\begin{pmatrix}
 0 & 0 & 1\\
 1 & 0 & 1\\
 0 & 1 & 1
\end{pmatrix}\!. \label{eq:P2}
\end{equation}
Then we have
\begin{align*}
& P^{-1}F_1P = \begin{pmatrix}
 B_1-\big(1-q^{\lambda}\big) & B_2 & 0\\
 0 & 0 & 0\\
 0 & 0 & 0
 \end{pmatrix}\!, \qquad
 P^{-1}F_2P = \begin{pmatrix}
 0 & 0 & 0\\
 B_1 & B_2-\big(1-q^{\lambda}\big) & 0\\
 0 & 0 & 0
 \end{pmatrix}\!, \nonumber
 \\
& P^{-1}F_{\infty}P = \begin{pmatrix}
 1 & 0 & 0\\
 0 & 1 & 0\\
 -B_1 & -B_2 & 1-B_0-B_1-B_2
 \end{pmatrix}\!.
\end{align*}
The upper-left $2\times 2$ submatrices of $P^{-1}F_{\infty} P$, $P^{-1} F_1 P$, $P^{-1}F_2 P$ are matrix representations of the transformations $F_{\infty}$, $F_1$, $F_2 $ on the quotient space $\Cplx ^3 / \mathcal{L} $.
Thus, $mc_{\lambda} (\textbf{B}) = \overline{\textbf{F}} = \big( \overline{F}_{\infty}; \overline{F}_1,\overline{F}_2\big)$, where
\begin{align*}
&
\overline{F}_1 = \begin{pmatrix}
 B_1-\big(1-q^{\lambda}\big) & B_2\\
 0 & 0
 \end{pmatrix}\!, \qquad
\overline{F}_2 = \begin{pmatrix}
 0 & 0\\
 B_1 & B_2-\big(1-q^{\lambda}\big)
 \end{pmatrix}\!, \qquad
\overline{F}_{\infty} = \begin{pmatrix}
 1 & 0\\
 0 & 1
 \end{pmatrix}\!.
\end{align*}
The $q$-difference equation $E_{\overline{\mathbf{F}},b}$ is written as
\begin{align}
& \begin{pmatrix} \bar{g}_1(qx) \\ \bar{g}_2(qx) \end{pmatrix}
 = \biggl( \overline{F}_{\infty}+\frac{\overline{F}_1}{1- \alpha_1 x}+\frac{\overline{F}_2}{1- \alpha_2 x} \biggr)
\begin{pmatrix} \bar{g}_1(x) \\ \bar{g}_2(x) \end{pmatrix}\! . \label{eq:g1g2sec412}
\end{align}
We apply Proposition~\ref{prop:g1g2single}.
Then a $q$-difference equation for $\bar{g}_1(x) $ is written as
\begin{align}
&\biggl(x-\frac{q^{\mu +1}\beta_1}{\alpha_1\alpha_2}\biggr)\biggl(x-\frac{q^{\mu +1}\beta_2}{\alpha_1\alpha_2}\biggr)\bar{g}_1(x/q) + q\biggl(x-\frac{1}{\alpha_1}\biggr)\biggl(x-\frac{q}{\alpha_2}\biggr)\bar{g}_1(qx) \nonumber
\\
&\qquad{} -\biggl\{(1+q)x^2-\biggl(q^{\mu+1}\frac{\beta_1+\beta_2}{\alpha_1\alpha_2} +\frac{q}{\alpha_1}+\frac{q^2}{\alpha_2}\biggr)x+ \biggl(1+\frac{\beta_1\beta_2}{\alpha_1\alpha_2}\biggr) \frac{q^{\mu+2}}{\alpha_1\alpha_2}\biggr\}\bar{g}_1(x)=0 .
\label{eq:g1eqdeg2m}
\end{align}
Note that this equation is transformed to the variant of $q$-hypergeometric equation of degree two given in equation~\eqref{eq:varqhgdeg2} by setting $x=1/z $.
See Proposition~\ref{prop:d2zparam} in the appendix for details.
The function $\bar{g}_2(x) $ in equation~\eqref{eq:g1g2sec412} satisfies the $q$-difference equation written by replacing $\alpha_1$ and $\alpha_2$ in equation~\eqref{eq:g1eqdeg2m}.

We can also discuss the convergence of the $q$-integral representations obtained by the $q$-middle convolution and the actual $q$-difference equations which the $q$-integral representations satisfy, that are described in the appendix.
See also~\cite{Ar}.

As a result, it is shown in the appendix that the function
\begin{align*}
\bar{g}_1(x) &= (q-1)\alpha_1^{-\mu}\frac{\big(q^{\lambda+1}/(\alpha_1 x), q, \alpha_2/\alpha_1;q\big)_{\infty}}{\big(q/(\alpha_1 x), \beta_1/\alpha_1, \beta_2/\alpha_1 ;q\big)_{\infty}} \,{}_3\phi_2 \left(\!\!\begin{array}{c} q/(\alpha_1 x), \beta_1/\alpha_1, \beta_2/\alpha_1\\
 q^{\lambda+1}/(\alpha_1 x), \alpha_2/\alpha_1 \end{array}\! ;q,q^{\mu} \right)
%\label{eq:xial1mdeg20}
\end{align*}
satisfies equation~\eqref{eq:g1eqdeg2m}, if $\mu > 0$.
It is obtained as a part of Proposition~\ref{prop:mdeg2homo}.
On the other hand, the function
\begin{align*}
& \bar{g}_1(x) = \frac{(1-q)\beta_1}{\alpha_1} x^{\lambda} \frac{( \beta_1 x, q, q \beta_1 / \beta_2 ; q)_{\infty}}{\big( q^{-\lambda} x \beta_1, \beta_1 / \alpha_1,q \beta_1 / \alpha_2 ; q\big)_{\infty}} \,{}_3\phi_2 \biggl(\!\! \begin{array}{c} q^{-\lambda}\beta_1 x, \beta_1/\alpha_1, q\beta_1/\alpha_2\\
 \beta_1 x, q\beta_1/\beta_2 \end{array}\! ;q,q \biggr)
\end{align*}
satisfies the equation
\begin{align*}
& \bigg( x-\frac{q^{\mu+1}\beta_1}{\alpha_1\alpha_2} \bigg) \bigg( x-\frac{q^{\mu+1}\beta_2}{\alpha_1\alpha_2} \bigg) \bar{g}_1(x/q) + q \bigg( x-\frac{1}{\alpha_1} \bigg) \bigg( x-\frac{q}{\alpha_2} \bigg) \bar{g}_1(qx) \nonumber
\\
&\qquad - \bigg\{ (1+q)x^2 - \bigg(q^{\mu+1}\frac{\beta_1+\beta_2}{\alpha_1\alpha_2} +\frac{q}{\alpha_1}+\frac{q^2}{\alpha_2}\bigg) x + \bigg(1+\frac{\beta_1\beta_2}{\alpha_1\alpha_2} \bigg) \frac{q^{\mu+2}}{\alpha_1\alpha_2} \bigg\} \bar{g}_1(x) \nonumber
\\
&\qquad + \frac{q(1-q)\big(1-q^{\lambda}\big)}{\alpha_1} x^{\lambda+1} = 0,
%\label{eq:nhm2m10}
\end{align*}
which is the non-homogeneous version of equation~\eqref{eq:g1eqdeg2m}.
It is obtained as a part of Proposition~\ref{prop:mdeg2nonhomo}.

\subsection[q-middle convolution and variants of q-hypergeometric equation of degree three]{$\boldsymbol{q}$-middle convolution and variants of $\boldsymbol{q}$-hypergeometric equation\\ of degree three} \label{sec:deg3}

Assume that $\alpha_1$, $\alpha_2$, $\alpha_3$, $\beta_1$, $\beta_2$, $\beta_3$ are mutually distinct and set
\begin{equation*}
y(x)=x^{\mu}\frac{(\alpha_1 x, \alpha_2 x, \alpha_3 x;q)_{\infty}}{(\beta_1 x, \beta_2 x, \beta_3 x;q)_{\infty}}.
\end{equation*}
The function $y(x)$ satisfies the linear $q$-difference equation $y(qx)=B(x)y(x)$, where
\begin{align}\label{eq:gqxBxgxqJPN3}
\begin{split}
& B(x)=B_{\infty} + \frac{B_1}{1-x/b_1}+ \frac{B_2}{1-x/b_2}+ \frac{B_3}{1-x/b_3}, \quad b_1=\frac{1}{\alpha_1},\quad
b_2=\frac{1}{\alpha_2},\quad
b_3 =\frac{1}{\alpha_3},
\\
& B_{\infty} = q^{\mu}\frac{\beta_1\beta_2\beta_3}{\alpha_1\alpha_2\alpha_3},\qquad
B_1 = q^{\mu}\frac{(\alpha_1-\beta_1)(\alpha_1-\beta_2)(\alpha_1-\beta_3)} {\alpha_1(\alpha_1-\alpha_2)(\alpha_1-\alpha_3)},
\\
& B_2 = q^{\mu} \frac{(\alpha_2-\beta_1)(\alpha_2-\beta_2)(\alpha_2-\beta_3)} {\alpha_2(\alpha_2-\alpha_3)(\alpha_2-\alpha_1)},\qquad
B_3 = q^{\mu} \frac{(\alpha_3-\beta_1)(\alpha_3-\beta_2)(\alpha_3-\beta_3)} {\alpha_3(\alpha_3-\alpha_2)(\alpha_3-\alpha_1)}.
\end{split}
\end{align}
We apply the convolution $c_{\lambda}$.
Then the tuple of $4\times 4$ matrices $c_{\lambda} (\textbf{B}) = \textbf{F}=(F_{\infty}; F_1, F_2,F_3)$ are written as equation~\eqref{eq:FiJP} for $N=3$ and the equation $E_{\textbf{F},b}$ is written as
\begin{align}\label{eq:claN3}
\begin{split}
& \widehat{Y}(qx) = H \widehat{Y}(x), \qquad
\widehat{Y}(x) = \begin{pmatrix} \widehat{y}_0(x) \\ \widehat{y}_1(x) \\ \widehat{y}_2(x) \\ \widehat{y}_3(x) \end{pmatrix}\!,
\\
&H = F_{\infty}+ \frac{F_1}{1-\alpha_1 x}+ \frac{F_2}{1-\alpha_2 x}+ \frac{F_3}{1-\alpha_3 x} .
\end{split}
\end{align}
The middle convolution was obtained by considering the quotient space divided by the space $\mathcal{K} +\mathcal{L} $.
Recall that the space $\mathcal{L} ( \subset \Cplx ^4)$ was defined by $\mathcal{L} = \operatorname{ker} (\widehat{F}-(1-q^{\lambda})I_4)$,
\begin{align*}
\widehat{F}-\big(1-q^{\lambda}\big)I_4
 = \begin{pmatrix}
 B_0 -\big(1-q^{\lambda}\big) & B_1 & B_2 & B_3\\
 B_0 & B_1-\big(1-q^{\lambda}\big) & B_2 & B_3\\
 B_0 & B_1 & B_2-\big(1-q^{\lambda}\big) & B_3\\
 B_0 & B_1 & B_2 & B_3-\big(1-q^{\lambda}\big)
 \end{pmatrix}\!,
\end{align*}
and $ B_0=1-B_{\infty}-B_1 -B_2-B_3=1-q^{\mu}$.
If $B_0 +B_1+ B_2 + B_3 =1-q^{\lambda}$ $(\Leftrightarrow q^{\lambda}=B_{\infty} = q^{\mu} \beta_1\beta_2\beta_3 /(\alpha_1\alpha_2\alpha_3 )$, then $\dim\ \mathcal{L} = 1$ and the vector $ ^t(1,1,1,1)$ is a basis of $\mathcal{L} $.
On the other hand, if $\mu=0$, then $B_0=0$, $\dim \mathcal{K} = 1$ and the vector $ ^t(1,0,0,0)$ is a basis of $\mathcal{K} $.

In order to write down the $q$-difference equation corresponding to the $q$-middle convolution, we take an invertible matrix $P$ as
\begin{equation}
P=\begin{pmatrix}
 0 & 0 & 1 & 1\\
 1 & 0 & 1 & 0\\
 0 & 1 & 1 & 0\\
 0 & 0 & 1 & 0
\end{pmatrix}
\label{eq:P}
\end{equation}
and set
\begin{equation*}
	\widetilde{F}_{1} = P^{-1}F_{1}P, \qquad
	\widetilde{F}_{2} = P^{-1}F_{2}P, \qquad
	\widetilde{F}_{3} = P^{-1}F_{3}P, \qquad
	\widetilde{F}_{\infty} = P^{-1}{F}_{\infty}P.
\end{equation*}
The function $\check{Y}(x) = P ^{-1} \widehat{Y}(x)$ satisfies
\begin{align}
& \check{Y}(qx)
= \biggl(\widetilde{F}_{\infty}+ \frac{\widetilde{F}_1}{1-\alpha_1 x}+ \frac{\widetilde{F}_2}{1-\alpha_2 x}+ \frac{\widetilde{F}_3}{1-\alpha_3 x}\biggr) \check{Y}(x), \qquad
\check{Y}(x) = \begin{pmatrix} \bar{g}_1(x) \\ \bar{g}_2(x) \\ \bar{g}_3(x) \\ \bar{g}_4 (x) \end{pmatrix}\! . \label{eq:claPN3}
\end{align}
Then it follows from a lengthy calculation that a $q$-difference equation for the function $\bar{g}_1(x)$ is written as
\begin{align}
& \big( q^3\alpha_1 x -1\big) \big(q^2 \alpha_2 x -1\big) \big(q^3 \alpha_3 x -1\big) \bar{g}_1\big(q^4 x\big) +c_3(x) \bar{g}_1\big(q^3 x\big) +c_2(x) \bar{g}_1\big(q^2 x\big) \nonumber
\\
&\qquad +c_1(x) \bar{g}_1(q x) + q^{\mu +4} \big( \beta_1 x - q^{\lambda} \big) \big(\beta_2 x-q^{\lambda} \big) \big(\beta_3 x-q^{\lambda} \big) \bar{g}_1( x) =0,
\label{eq:g1q4}
\end{align}
where $c_3(x) $, $c_2(x)$ and $c_1(x)$ are some polynomials of degree three.

We now impose the condition
\begin{equation}
\mu=0\qquad \mbox{and}\qquad q^{\lambda} = \frac{\beta_1\beta_2\beta_3}{\alpha_1\alpha_2\alpha_3}. \label{eq:condmula}
\end{equation}
Then we have
\begin{align*}
& \widetilde{F}_{1} = \begin{pmatrix}
 B_1-\big(1-q^{\lambda}\big) & B_2 & 0 & 0\\
 0 & 0 & 0 & 0\\
 0 & 0 & 0 & 0\\
 0 & 0 & 0 & 0
 \end{pmatrix}\!,\qquad
 \widetilde{F}_{2} = \begin{pmatrix}
 0 & 0 & 0 & 0\\
 B_1 & B_2-\big(1-q^{\lambda}\big) & 0 & 0\\
 0 & 0 & 0 & 0\\
 0 & 0 & 0 & 0
 \end{pmatrix}\!, \nonumber
 \\
& \widetilde{F}_{3} = \begin{pmatrix}
 -B_1 & -B_2 & 0 & 0\\
 -B_1 & -B_2 & 0 & 0\\
 B_1 & B_2 & 0 & 0\\
 -B_1 & -B_2 & 0 & 0
 \end{pmatrix}\!,\qquad
 \widetilde{F}_{\infty} = \begin{pmatrix}
 1 & 0 & 0 & 0\\
 0 & 1 & 0 & 0\\
 -B_1 & -B_2 & q^{\lambda} & 0\\
 0 & 0 & 0 & 1
 \end{pmatrix}\!.
\end{align*}
The upper-left $2\times 2$ submatrices of $\widetilde{F}_{\infty}$, $\widetilde{F}_{1}$, $\widetilde{F}_{2}$, $\widetilde{F}_{3}$ are matrix representations of the transformations $F_{\infty}$, $F_1$, $F_2$, $F_3$ on the quotient space $\Cplx ^4 / (\mathcal{K} + \mathcal{L})$.
Thus, the $q$-middle convolution $ mc_{\lambda} (\textbf{B})$ is realized as the tuple of these $2\times 2$ matrices.
We can restrict equation~\eqref{eq:claPN3} by choosing the first two components, and it is written as
\begin{align*}
& \begin{pmatrix} \bar{g}_1(qx) \\ \bar{g}_2(qx) \end{pmatrix}
 = \begin{pmatrix}
 \displaystyle 1+\frac{B_1-1+q^{\lambda}}{1-\alpha_1 x}-\frac{B_1}{1-\alpha_3 x} & \displaystyle \frac{B_2}{1-\alpha_1 x}-\frac{B_2}{1-\alpha_3 x}\\
 \displaystyle \frac{B_1}{1-\alpha_2 x}-\frac{B_1}{1-\alpha_3 x} & \displaystyle 1+\frac{B_2-1+q^{\lambda}}{1-\alpha_2 x}-\frac{B_2}{1-\alpha_3 x}
 \end{pmatrix}
\begin{pmatrix}
\bar{g}_1(x) \\ \bar{g}_2(x) \end{pmatrix}\! .
\end{align*}
By applying Proposition~\ref{prop:g1g2single}, a $q$-difference equation for the function $\bar{g}_1(x)$ is written as
\begin{align}
& \biggl( x-q\frac{\beta_1\beta_2}{\alpha_1\alpha_2\alpha_3} \biggr) \biggl( x-q\frac{\beta_2\beta_3}{\alpha_1\alpha_2\alpha_3} \biggr) \biggl( x-q\frac{\beta_3\beta_1}{\alpha_1\alpha_2\alpha_3} \biggr) \bar{g}_1(x/q) \nonumber
\\
&\qquad +q \biggl( x-\frac{1}{\alpha_1} \biggr) \biggl( x-\frac{q}{\alpha_2} \biggr) \biggl( x-\frac{1}{\alpha_3} \biggr) \bar{g}_1(qx) \nonumber
\\
&\qquad + \biggl\{ -(1+q)x^3 + q \biggl( \frac{1}{\alpha_1}+\frac{q}{\alpha_2}+\frac{1}{\alpha_3} +\frac{\beta_1\beta_2+\beta_2\beta_3+\beta_3\beta_1}{\alpha_1\alpha_2\alpha_3} \biggr) x^2 \nonumber
\\
& \qquad \hphantom{+ \biggl\{} - \frac{q^2}{\alpha_1\alpha_2\alpha_3} \biggl( \beta_1+\beta_2+\beta_3 +\frac{\beta_1\beta_2\beta_3}{\alpha_1 \alpha_2 \alpha_3} \biggl( \alpha_1 +\frac{\alpha_2}{q} + \alpha_3 \biggr) \biggr) x \nonumber
\\
& \qquad \hphantom{+ \biggl\{} +q^2(1+q)\frac{\beta_1\beta_2\beta_3}{{\alpha_1}^2{\alpha_2}^2{\alpha_3}^2}\biggr\} \bar{g}_1(x) =0 .
\label{eq:g1eqdeg3}
\end{align}
The function $\bar{g}_2(x) $ satisfies the $q$-difference equation written by replacing $\alpha_1$ and $\alpha_2$ in equation~\eqref{eq:g1eqdeg3}.

Recall that the variant of $q$-hypergeometric equation of degree three was introduced in~\cite{HMST}, and it is written as
\begin{align}
& \big(x-q^{h_1 +1/2} t_1\big) \big(x- q^{h_2 +1/2} t_2\big) \big(x- q^{h_3 +1/2} t_3\big) g(x/q) \nonumber
\\
& \qquad + q^{2\alpha +1} \big(x - q^{l_1-1/2}t_1 \big) \big(x - q^{l_2 -1/2} t_2\big) \big(x - q^{l_3 -1/2} t_3\big) g(qx) \nonumber
\\
& \qquad + q^{\alpha} \bigl[ - (q + 1 ) x^3 + q^{1/2} \big\{ \big(q^{h_1} + q^{l_1}\big)t_1 + \big(q^{h_2} + q^{l_2}\big)t_2 + \big(q^{h_3} + q^{l_3}\big)t_3 \big\} x^2 \nonumber
\\
& \qquad\phantom{+ q^{\alpha} \bigl[} - q^{(h_1+h_2+h_3+l_1+l_2+l_3 +1)/2} \big\{ \big(q^{- h_1}+q^{-l_1}\big)t_2 t_3 + \big(q^{- h_2}+ q^{- l_2}\big) t_1 t_3 \nonumber
\\
& \qquad\phantom{+ q^{\alpha} \big[} + \big(q^{- h_3}+ q^{- l_3}\big) t_1 t_2 \big\} x + q^{(h_1 +h_2 + h_3 + l_1 + l_2 + l_3 )/2} ( q + 1 ) t_1 t_2t_3 \bigr] g(x) =0 .
\label{eq:varqhgdeg3}
\end{align}
Equation~\eqref{eq:g1eqdeg3} is a special case of the variant of $q$-hypergeometric equation of degree three with the constraint $\alpha =0$.
Equation~\eqref{eq:g1eqdeg3} recovers the variant of $q$-hypergeometric equation of degree three by setting $g (x) = x ^{-\alpha} \bar{g}_1(x) $.
Namely, we obtain the following proposition directly.

\begin{Proposition} \label{prop:paramreldeg3}
Assume that $\bar{g}_1(x) $ satisfies equation~\eqref{eq:g1eqdeg3}.
Set $g (x) = x ^{-\alpha} \bar{g}_1(x) $ and
\begin{align*}
& \alpha_1 = \frac{q^{- l_1 + 1/2}}{t_1}, \quad \
\alpha_2 = \frac{q^{-l_2 +3/2}}{t_2},\quad \
\alpha_3 = \frac{q^{- l_3 + 1/2}}{t_3}, \quad \
\beta_1 = \frac{q^{ (-h_1+h_2+h_3-l_1-l_2-l_3)/2 +1}}{ t_1}, \nonumber
\\
& \beta_2 = \frac{q^{ (h_1-h_2+h_3-l_1-l_2-l_3)/2 +1}}{ t_2}, \qquad
\beta_3 = \frac{q^{ (h_1+h_2-h_3-l_1-l_2-l_3)/2 +1}}{ t_3} .
%\label{eq:propparamdeg3}
\end{align*}
Then $g(x)$ satisfies the variant of $q$-hypergeometric equation of degree three given in equation~\eqref{eq:varqhgdeg3}.
Note that $\beta_j = q^{\lambda -h_j +1/2} /t_j$, $j=1,2,3$, where $q^{\lambda} = q^{ (h_1+h_2+h_3-l_1-l_2-l_3 +1 )/2} $.
\end{Proposition}

We obtained the formal $q$-difference equation for $\bar{g}_1(x)$ corresponding to the $q$-middle convolution for the tuples of the matrices in equation~\eqref{eq:g1eqdeg3}.
In the appendix, we discuss the convergence of the $q$-integral representations obtained by the $q$-middle convolution and the actual $q$-difference equations which the $q$-integral representation satisfy.
In conclusion, we obtain the following theorems.

\begin{Theorem} \label{thm:deg3-1}
Set $\lambda = (h_1+h_2+h_3-l_1-l_2-l_3 +1 )/2 $.
\begin{enumerate}
 \item [$(i)$] Let $(i,i',i'')$ be a permutation of $(1,2,3)$.
The function
\begin{align*}
 g_i^{\langle 1 \rangle} (x) &= x^{-\alpha} q^{l_i - 1/2} t_i \frac{\big(q^{\lambda + l_i + 1/2} t_i / x, q^{l_i -l_{i'} +1} t_i /t_{i'}, q^{l_i - l_{i''} + 1} t_i /t_{i''}, q;q\big)_{\infty}}{\big(q^{ l_i + 1/2} t_i / x, q^{\lambda +l_i -h_i}, q^{\lambda + l_i -h_{i'}} t_i /t_{i'}, q^{\lambda + l_i -h_{i''}} t_i /t_{i''} ;q\big)_{\infty}}
\\
& \hphantom{=}\times {}_4\phi_3 \biggl(\!\! \begin{array}{c} q^{ l_i + 1/2} t_i / x, q^{\lambda + l_i -h_i}, q^{\lambda + l_i -h_{i'}} t_i /t_{i'}, q^{\lambda + l_i -h_{i''}} t_i /t_{i''} \\ q^{\lambda + l_i + 1/2} t_i /x, q^{ l_i -l_{i'} +1} t_i /t_{i'}, q^{ l_i - l_{i''} + 1} t_i /t_{i''} \end{array};q,q \biggr)
\end{align*}
satisfies the non-homogeneous version of the variant of $q$-hypergeometric equation of degree three given by
\begin{align}
& \big(x-q^{h_1 +1/2} t_1\big) \big(x- q^{h_2 +1/2} t_2\big) \big(x- q^{h_3 +1/2} t_3\big) g^{\langle 1 \rangle} (x/q) \nonumber\\
& + q^{2\alpha +1} \big(x - q^{l_1-1/2}t_1 \big) \big(x - q^{l_2 -1/2} t_2\big) \big(x - q^{l_3 -1/2} t_3\big) g^{\langle 1 \rangle} (qx) \nonumber \\
& + q^{\alpha} \bigl[ - (q + 1 ) x^3 + q^{1/2} \{ (q^{h_1} + q^{l_1})t_1 + (q^{h_2} + q^{l_2})t_2 + (q^{h_3} + q^{l_3})t_3 \} x^2 \nonumber\\
& \qquad - q^{(h_1+h_2+h_3+l_1+l_2+l_3 +1)/2} \big\{ \big(q^{- h_1}+q^{-l_1}\big)t_2 t_3 + \big(q^{- h_2}+ q^{- l_2}\big) t_1 t_3 \nonumber \\
& \qquad + \big(q^{- h_3}+ q^{- l_3}\big) t_1 t_2 \big\} x + q^{(h_1 +h_2 + h_3 + l_1 + l_2 + l_3 )/2} ( q + 1 ) t_1 t_2t_3 \bigr] g^{\langle 1 \rangle} (x) \nonumber \\
& - q^{\alpha +l_1 +l_2 +l_3 -1/2} t_1 t_2 t_3 \big(1-q^{\lambda}\big) x^{1-\alpha} =0 .
\label{eq:varqhgdeg3nonhom1}
\end{align}

\item[$(ii)$] The function
\begin{align*}
 g^{\langle 1 \rangle}_0 (x) &= x^{1-\alpha} q^{-\lambda} \frac{\big( x q^{-\lambda - l_1 + 3/2} /t_1, x q^{-\lambda -l_2 +3/2} /t_2, x q^{-\lambda - l_3 + 3/2} /t_3, q ;q\big)_{\infty}}{\big(x q^{-h_1 +1/2} /t_1, x q^{ -h_2 +1/2} /t_2, x q^{ -h_3 +1/2} /t_3, q^{-\lambda+1} ;q\big)_{\infty}} \nonumber
 \\
& \hphantom{=}\times {}_4\phi_3 \biggl(\!\! \begin{array}{c} x q^{-h_1 +1/2} /t_1, x q^{-h_2 +1/2} /t_2, x q^{ -h_3 +1/2} /t_3, q^{-\lambda+1} \\
 x q^{-\lambda - l_1 + 3/2} /t_1, x q^{-\lambda -l_2 +3/2} /t_2, x q^{-\lambda - l_3 + 3/2}/t_3 \end{array};q,q\biggr)
\end{align*}
also satisfies equation~\eqref{eq:varqhgdeg3nonhom1}.
\end{enumerate}
\end{Theorem}

\begin{Theorem} \label{thm:deg3-2}
Set $\lambda = (h_1+h_2+h_3-l_1-l_2-l_3 +1 )/2 $.
\begin{enumerate}
\item [$(i)$] Let $(i,i',i'')$ be a permutation of $(1,2,3)$.
The function
\begin{align*}
 g_i ^{\langle 2 \rangle} (x) & = x^{\lambda - \alpha} \frac{q^{\lambda -h_i -1/2}}{t_i} \frac{\big(x q^{\lambda -h_i +1/2} /t_i, q^{ h_{i'} -h_i +1} t_{i'} /t_i, q^{h_{i''} -h_i +1} t_{i''} /t_i, q ; q\big)_{\infty}}{\big(x q^{ -h_i +1/2} /t_i, q^{\lambda + l_i -h_i}, q^{\lambda + l_{i'} -h_i} t_{i'} /t_i, q^{\lambda + l_{i''} -h_i}t_{i''} /t_i ; q\big)_{\infty}} \nonumber
 \\
& \hphantom{=}\times {}_4\phi_3 \biggl(\!\! \begin{array}{c} x q^{ -h_i +1/2} /t_i, q^{\lambda + l_i -h_i}, q^{\lambda + l_{i'} -h_i} t_{i'} /t_i, q^{\lambda + l_{i''} -h_i}t_{i''} /t_i \\
 x q^{\lambda -h_i +1/2} /t_i, q^{ h_{i'} -h_i +1} t_{i'} /t_i, q^{h_{i''} -h_i +1} t_{i''} /t_i \end{array}\! ;q,q \biggr)
\end{align*}
satisfies the non-homogeneous version of the variant of $q$-hypergeometric equation of degree three given by
\begin{align}
& \big(x-q^{h_1 +1/2} t_1\big) \big(x- q^{h_2 +1/2} t_2\big) \big(x- q^{h_3 +1/2} t_3\big) g^{\langle 2 \rangle}(x/q) \nonumber\\
& + q^{2\alpha +1} \big(x - q^{l_1-1/2}t_1 \big) \big(x - q^{l_2 -1/2} t_2\big) \big(x - q^{l_3 -1/2} t_3\big) g^{\langle 2 \rangle}(qx) \nonumber \\
& + q^{\alpha} \bigl[ - (q + 1 ) x^3 + q^{1/2} \big\{ \big(q^{h_1} + q^{l_1}\big)t_1 + \big(q^{h_2} + q^{l_2}\big)t_2 + \big(q^{h_3} + q^{l_3}\big)t_3 \big\} x^2 \nonumber\\
& \qquad - q^{(h_1+h_2+h_3+l_1+l_2+l_3 +1)/2} \big\{ \big(q^{- h_1}+q^{-l_1}\big)t_2 t_3 + \big(q^{- h_2}+ q^{- l_2}\big) t_1 t_3 \nonumber \\
& \qquad + \big(q^{- h_3}+ q^{- l_3}\big) t_1 t_2 \big\} x + q^{(h_1 +h_2 + h_3 + l_1 + l_2 + l_3 )/2} ( q + 1 ) t_1 t_2t_3 \bigr] g^{\langle 2 \rangle}(x) \nonumber \\
& + q^{\alpha} \big(1-q^{\lambda}\big) x^{\lambda - \alpha +2} = 0.
\label{eq:varqhgdeg3nonhom2}
\end{align}

\item[$(ii)$] The function
\begin{align*}
 g_0 ^{\langle 2 \rangle}(x) & = x^{\lambda - \alpha -1} \frac{\big( q^{-\lambda +h_1 +3/2} t_1 / x, q^{-\lambda +h_2 +3/2} t_2 / x, q^{-\lambda +h_3 +3/2} t_3 / x, q ; q\big)_{\infty}}{\big( q^{ l_1 + 1/2} t_1 / x,q^{l_2 +1/2} t_2 / x, q^{ l_3 + 1/2} t_3 / x, q^{1-\lambda} ; q\big)_{\infty}} \nonumber
 \\
& \hphantom{=}\times {}_4 \phi_3 \biggl(\!\! \begin{array}{c} q^{ l_1 + 1/2} t_1 / x,q^{l_2 +1/2} t_2 / x, q^{ l_3 + 1/2} t_3 / x, q^{1-\lambda} \\
 q^{-\lambda +h_1 +3/2} t_1 / x, q^{-\lambda +h_2 +3/2} t_2 / x, q^{-\lambda +h_3 +3/2} t_3 / x \end{array}\! ;q,q \biggr)
\end{align*}
also satisfies equation~\eqref{eq:varqhgdeg3nonhom2}.
\end{enumerate}
\end{Theorem}
By taking the difference of two solutions of the non-homogeneous equation, we obtain a~solution of the homogeneous equation.
\begin{Corollary}\quad
\begin{enumerate}
\item [$(i)$] Let $g^{\langle 1 \rangle}_i (x)$, $i=0,1,2,3$, be the functions in Theorem~$\ref{thm:deg3-1}$.
For each $i,j \in \{ 0,1,2,3 \}$, the function $g^{\langle 1 \rangle}_i (x) - g^{\langle 1 \rangle}_j (x) $ satisfies the variant of the $q$-hypergeometric equation of degree three given in equation~\eqref{eq:varqhgdeg3}.

\item [$(ii)$] Let $g^{\langle 2 \rangle}_i (x)$, $i=0,1,2,3$, be the functions in Theorem~$\ref{thm:deg3-2}$.
For each $i,j \in \{ 0,1,2,3 \}$, the function $g^{\langle 2 \rangle}_i (x) - g^{\langle 2 \rangle}_j (x) $ also satisfies in equation~\eqref{eq:varqhgdeg3}.
\end{enumerate}
\end{Corollary}

\section{Concluding remarks} %\label{sec:CR}

In this paper, we reformulated the $q$-integral transformations associated with the $q$-convolution, which was previously formulated by Sakai and Yamaguchi~\cite{SY}.
Namely, we added the parameter~$\xi $ in the Jackson integral and we took care of the convergence of the $q$-integration.
As an application of our reformulation, we obtained $q$-integral representations of solutions to the variants of the $q$-hypergeometric equation by applying the $q$-middle convolution.

On the hypergeometric differential equation \eqref{eq:GaussHGE}, global behaviour of solutions had been studied well.
In particular, the $2\times 2$ matrix which connects the local solutions at $z=0$ with those at $z=\infty $ was calculated, and integral representations of solutions were applied for the calculation~\cite{WW}.
The connection matrix for the $q$-hypergeometric equation was also calculated~\cite{GR}.
We hope to study global behaviour of solutions to the variants of the $q$-hypergeometric equation.
For this purpose, we need to develop the local theory of linear $q$-difference equations at the point other than the origin and the infinity.

When the authors were preparing this manuscript, Fujii and Nobukawa presented the pre\-print~\cite{FN}.
They also investigated $q$-integral solutions to the variants of the $q$-hypergeometric equation.
In particular, they found solutions in terms of the very-well-poised-balanced $q$-hy\-per\-ge\-o\-met\-ric series ${}_8 W_7$.
Fujii also obtained several results on this direction in his master's thesis~\cite{Fuj}.
It seems that their results do not rely on the $q$-middle convolution, and it would be interesting to compare them with our ones for further study.

A referee pointed out that the results in Section~\ref{sec:qmc} will open a possibility to give a recursive formula of the connection coefficients, if a local solution is sent to a local solution of the $q$-middle convolution equation by choosing an appropriate value of $\xi $.

\appendix

\section{Supplement to Section~\ref{sec:deg2p2}}\allowdisplaybreaks

We provide additional discussions on the $q$-middle convolution related to the function $x^{\mu} (\alpha_1 x, \allowbreak\alpha_2 x;q)_{\infty}/(\beta_1 x, \beta_2 x;q)_{\infty} $ with the condition $q^{\lambda} = q^{\mu} \beta_1 \beta_2 /(\alpha_1 \alpha_2) $.
To begin with, we discuss a~relationship between equation~\eqref{eq:g1eqdeg2m} and the variant of $q$-hypergeometric equation of degree two.
Recall that the $q$-difference equation for $\bar{g}_1(x) $ given in equation~\eqref{eq:g1eqdeg2m} is
\begin{align*}
&\biggl(x-\frac{q^{\mu +1}\beta_1}{\alpha_1\alpha_2}\biggr)\biggl(x-\frac{q^{\mu +1}\beta_2}{\alpha_1\alpha_2}\biggr)\bar{g}_1(x/q) + q\biggl(x-\frac{1}{\alpha_1}\biggr)\biggl(x-\frac{q}{\alpha_2}\biggr)\bar{g}_1(qx) \nonumber
\\
&\qquad{} -\biggl\{(1+q)x^2-\biggl(q^{\mu+1}\frac{\beta_1+\beta_2}{\alpha_1\alpha_2}+\frac{q}{\alpha_1}+\frac{q^2}{\alpha_2}\biggr)x+ \biggl(1+\frac{\beta_1\beta_2}{\alpha_1\alpha_2}\biggr)\frac{q^{\mu+2}}{\alpha_1\alpha_2}\biggr\}\bar{g}_1(x)=0 .
%\label{eq:g1eqdeg2m+}
\end{align*}
Set $x=1/z$, and write $\bar{g}_1(x) = x^{\mu}f(1/x)$.
Then $\bar{g}_1(x)=z^{-\mu}f(z)$, $\bar{g}_1(qx)= q ^{\mu} z^{-\mu}f(z/q)$, $\bar{g}_1(x/q)= q ^{-\mu} z^{-\mu}f(qz)$, and we have
\begin{align}
& ( z-\alpha_1 ) \biggl( z-\frac{\alpha_2}{q} \biggr) f(z/q) +\frac{\beta_1\beta_2}{\alpha_1\alpha_2}\biggl(z-q^{-\mu-1}\frac{\alpha_1\alpha_2}{\beta_1}\biggr)\biggl(z-q^{-\mu-1}\frac{\alpha_1\alpha_2}{\beta_2}\biggr)f(qz) \label{eq:varqhgd2z}
\\
&-\biggl\{ \!\biggl( 1+\frac{\beta_1\beta_2}{\alpha_1\alpha_2} \biggr) z^2-\big(q^{-1}(\beta_1+\beta_2)+q^{-\mu}\alpha_1+q^{-\mu-1}\alpha_2\big)z
+q^{-\mu-2}(q+1)\alpha_1\alpha_2\!\biggl\}f(z)=0.\nonumber
\end{align}
This equation is a special case of the variant of $q$-hypergeometric equation of degree two, which was given in equation~\eqref{eq:varqhgdeg2}, with the constraint $k_1=0$.
However, equation~\eqref{eq:varqhgd2z} recovers the variant of $q$-hypergeometric equation of degree two by setting $g(z)=z^{-k_1}f(z)$.
Namely, we obtain the following proposition directly.
\begin{Proposition} \label{prop:d2zparam}
Assume that $f(z)$ satisfies equation~\eqref{eq:varqhgd2z}. Set $g(x)=x^{-k_1}f(x)$ and
\begin{align*}
&\alpha_1=q^{h_1+1/2}t_1, \qquad
\alpha_2=q^{h_2+3/2}t_2, \qquad
q^{\lambda}=q^{(h_1+h_2-l_1-l_2-k_1+k_2+1)/2},
\\
&\beta_1=q^{(h_1+h_2+l_1-l_2-k_1+k_2)/2+1}t_1, \qquad
\beta_2=q^{(h_1+h_2-l_1+l_2-k_1+k_2)/2+1}t_2.
%\label{eq:d2zparam}
\end{align*}
Then $g(x)$ satisfies the variant of $q$-hypergeometric equation of degree two given in equation~\eqref{eq:varqhgdeg2}.
Note that $\beta_1=q^{\lambda+l_1+1/2}t_1$, $\beta_2=q^{\lambda+l_2+1/2}t_2$ and $q^{\mu}=q^{\lambda}\alpha_1\alpha_2/(\beta_1\beta_2)=q^{\lambda+k_1-k_2}$.
\end{Proposition}

We can also discuss the convergence of the $q$-integral representation obtained by the $q$-middle convolution and the actual $q$-difference equation which the $q$-integral representation satisfies as the case $\mu =0$ in Section~\ref{sec:mu0}.
We apply Proposition~\ref{prop:convKL} for the case $y(x)= x^{\mu}(\alpha_1 x, \alpha_2 x;q)_{\infty}/(\beta_1 x, \beta_2 x;q)_{\infty}$.
Then the functions
\begin{align}\label{eq:hyint2m}
\begin{split}
& \widehat{y}_{i}^{[K,L]}(x) = (1-q) \sum_{n=K}^{L} s^{\mu+1} \frac{P_{\lambda}(x, s)}{s-b_{i}} \frac{(\alpha_1 s, \alpha_2 s;q)_{\infty}}{(\beta_1 s, \beta_2 s;q)_{\infty}} \bigg|_{s= q^n \xi}, \qquad i=0,1,2,
\\
& b_0=0, \qquad b_1 = 1/\alpha_1, \qquad b_2 = 1/\alpha_2
\end{split}
\end{align}
satisfy
\begin{align*}
& \begin{pmatrix} \widehat{y}_0^{\,[K,L]}(qx) \\[1mm] \widehat{y}_1^{\,[K,L]}(qx) \\[1mm] \widehat{y}_2^{\,[K,L]}(qx) \end{pmatrix}
= H \begin{pmatrix} \widehat{y}_0^{\,[K,L]}(x) \\[1mm] \widehat{y}_1^{\,[K,L]}(x) \\[1mm] \widehat{y}_2^{\,[K,L]}(x) \end{pmatrix} + (1-q) Q^{[K,L]}(x) \begin{pmatrix} 1 \\[1mm] -\alpha_1 x /(1-\alpha_1 x) \\[1mm] -\alpha_2 x /(1-\alpha_2 x) \end{pmatrix}\!,
\end{align*}
where $H=F_{\infty}+ F_1/(1-\alpha_1 x ) + F_2/(1-\alpha_2 x)$ is the matrix determined in equation~\eqref{eq:Yqx33} with the condition $q^{\lambda} = q^{\mu} \beta_1\beta_2 /(\alpha_1\alpha_2) $, and
\begin{equation*}
Q^{[K,L]}(x)=P_{\lambda}\big(x, q^{K-1} \xi\big) y \big(q^{K} \xi\big) - P_{\lambda}\big(x, q^{L} \xi\big) y \big(q^{L+1} \xi\big).
\end{equation*}
Note that the convergence theorem in Theorem~\ref{thm:qcintqJH} for the functions $\widehat{y}_i^{\,[K,L]}(x)$, $i=0,1,2$, is not applicable by the condition $q^{\lambda}=q^{\mu} \beta_1\beta_2 /(\alpha_1\alpha_2 )$.
Write
\begin{align}
\begin{pmatrix}
 \bar{g}_1^{\,[K,L]}(x)\\[1mm]
 \bar{g}_2^{\,[K,L]}(x)\\[1mm]
 \bar{g}_3^{\,[K,L]}(x)
 \end{pmatrix}
=
P^{-1} \begin{pmatrix}
 \widehat{y}_0^{\,[K,L]}(x)\\[1mm]
 \widehat{y}_1^{\,[K,L]}(x)\\[1mm]
 \widehat{y}_2^{\,[K,L]}(x)
 \end{pmatrix}
=\begin{pmatrix}
 -1 & 1 & 0 \\
 -1 & 0 & 1 \\
 1 & 0 & 0
 \end{pmatrix}
 \begin{pmatrix}
 \widehat{y}_0^{\,[K,L]}(x)\\[1mm]
 \widehat{y}_1^{\,[K,L]}(x)\\[1mm]
 \widehat{y}_2^{\,[K,L]}(x)
 \end{pmatrix}\!,
\label{eq:gyPinv2}
\end{align}
where $P$ is the matrix determined in equation~\eqref{eq:P2}.
Then we have
\begin{align*}
\begin{pmatrix}
 \bar{g}_1^{\,[K,L]}(qx)\\[1mm]
 \bar{g}_2^{\,[K,L]}(qx)\\[1mm]
 \bar{g}_3^{\,[K,L]}(qx)
 \end{pmatrix}
= P^{-1}HP \begin{pmatrix}
 \bar{g}_1^{\,[K,L]}(x)\\[1mm]
 \bar{g}_2^{\,[K,L]}(x)\\[1mm]
 \bar{g}_3^{\,[K,L]}(x)
 \end{pmatrix}
- (1-q) Q^{[K,L]}(x) P^{-1} \begin{pmatrix}
 -1\\[1mm]
 \alpha_1 x /(1-\alpha_1 x)\\[1mm]
 \alpha_2 x /(1-\alpha_2 x)
\end{pmatrix}\!.
\end{align*}
The first two equations are written as
\begin{align*}
& \bar{g}_1^{\,[K,L]}(qx) = \frac{B_1+q^{\lambda}-\alpha_1 x}{1-\alpha_1 x}\bar{g}_1^{\,[K,L]}(x)
+ \frac{B_2}{1-\alpha_1 x}\bar{g}_2^{\,[K,L]}(x) - \frac{1-q}{1-\alpha_1 x} Q^{[K,L]}(x),
\\
&\bar{g}_2^{\,[K,L]}(qx) = \frac{B_1}{1-\alpha_2 x}\bar{g}_1^{\,[K,L]}(x)
+ \frac{B_2+q^{\lambda}-\alpha_2 x}{1-\alpha_2 x}\bar{g}_2^{\,[K,L]}(x) - \frac{1-q}{1-\alpha_2 x}Q^{[K,L]}(x).
\end{align*}
We apply Proposition~\ref{prop:g1g2single} to obtain the $q$-difference equation which the function $\bar{g}_1^{\,[K,L]}(x) $ satisfies.
Then we have
\begin{align}
& \bigg( x-\frac{q^{\mu+1}\beta_1}{\alpha_1\alpha_2} \bigg) \bigg( x-\frac{q^{\mu+1}\beta_2}{\alpha_1\alpha_2} \bigg) \bar{g}_1^{\,[K,L]}(x/q) + q \bigg( x-\frac{1}{\alpha_1} \bigg) \bigg( x-\frac{q}{\alpha_2} \bigg) \bar{g}_1^{\,[K,L]}(qx) \nonumber
\\
&\qquad{}- \bigg\{ (1+q)x^2 - \bigg( q^{\mu+1}\frac{\beta_1+\beta_2}{\alpha_1\alpha_2}+\frac{q}{\alpha_1}+\frac{q^2}{\alpha_2}\bigg) x + \bigg( 1+\frac{\beta_1\beta_2}{\alpha_1\alpha_2} \bigg) \frac{q^{\mu+2}}{\alpha_1\alpha_2} \bigg\} \bar{g}_1^{\,[K,L]}(x) \nonumber
\\
&\qquad{}+ \frac{q^2(1-q)}{\alpha_1\alpha_2} \bigg\{ \bigg( \frac{\alpha_2}{q}x-q^{\lambda} \bigg) Q^{[K,L]}(x/q)+ \bigg( 1-\frac{\alpha_2}{q}x \bigg) Q^{[K,L]}(x) \bigg\} = 0,
\label{eq:nhmdeg2}
\end{align}
and the non-homogeneous term is written as
\begin{align}
&\frac{q^2(1-q)}{\alpha_1\alpha_2} \bigg\{ \bigg( \frac{\alpha_2}{q}x-q^{\lambda} \bigg) Q^{[K,L]}(x/q) + \bigg( 1-\frac{\alpha_2}{q}x \bigg) Q^{[K,L]}(x) \bigg\} \nonumber
\\
&\qquad{}= \frac{q^2(1-q)(1-q^{\lambda})}{\alpha_1\alpha_2}
\biggl( (q^K\xi)^{\mu} \frac{(q^{\lambda+K+1}\xi/x, q^K\xi\alpha_1, q^{K-1}\xi\alpha_2 ;q)_{\infty}}{(q^K\xi/x, q^K\xi\beta_1, q^K\xi\beta_2 ;q)_{\infty}} \nonumber
\\
&\qquad\hphantom{= \frac{q^2(1-q)(1-q^{\lambda})}{\alpha_1\alpha_2}
\biggl(} -(q^{L+1}\xi)^{\mu} \frac{(q^{\lambda+L+2}\xi/x, q^{L+1}\xi\alpha_1, q^L\xi\alpha_2 ;q)_{\infty}}{(q^{L+1}\xi/x, q^{L+1}\xi\beta_1, q^{L+1}\xi\beta_2 ;q)_{\infty}} \biggr).
\label{eq:nhm2m}
\end{align}

On the other hand, it follows from equations~\eqref{eq:hyint2m}, \eqref{eq:gyPinv2} that the function $\bar{g}_1^{\,[K,L]}(x) $ is written as
\begin{align*}
&\bar{g}_1^{\,[K,L]}(x) = -\widehat{y}_0^{\,[K,L]}(x) + \widehat{y}_1^{\,[K,L]}(x)
\\
&\hphantom{\bar{g}_1^{\,[K,L]}(x)} = (1-q) \sum_{n=K}^{L} \frac{-s^{\mu}}{1-\alpha_1 s} \frac{(q^{\lambda+1} s /x,\alpha_1 s, \alpha_2 s;q)_{\infty}}{(q s /x,\beta_1 s, \beta_2 s;q)_{\infty}} \bigg|_{s= q^n \xi}
= (q-1)\sum_{n=K}^{L}c_n,
\\
&c_n = (q^n\xi)^{\mu} \frac{\big(q^{\lambda+n+1}\xi /x, q^{n+1} \xi\alpha_1, q^n \xi\alpha_2; q\big)_{\infty}}{\big(q^{n+1}\xi /x, q^n \xi\beta_1, q^n \xi\beta_2 ;q\big)_{\infty}} .
\end{align*}
Since $c_{n+1}/c_n \to q^{\mu}$, $n\to +\infty$, and $c_{-(n+1)}/c_{-n} \to q^{\lambda - \mu +1} \alpha_1 \alpha_2 /(\beta_1 \beta_2 ) = q$, $n\to +\infty$, the function $\bar{g}_1^{\,[K,L]}(x) $ converges as $K \to -\infty $ and $L \to +\infty $ under the condition $\mu > 0$.
Note that the functions $ \widehat{y}_0^{\,[K,L]}(x)$ and $\widehat{y}_1^{\,[K,L]}(x)$ do not converge as $K \to -\infty $.
Write
\begin{align}
& \bar{g}_1 (x)= \lim_{K \to -\infty \atop{ L \to +\infty}} \bar{g}_1^{\,[K,L]}(x)
 = (q-1) \sum_{n=-\infty}^{+\infty} (q^n\xi)^{\mu} \frac{\big(q^{\lambda+n+1}\xi /x, q^{n+1} \xi\alpha_1, q^n \xi\alpha_2; q\big)_{\infty}}{\big(q^{n+1}\xi /x, q^n \xi\beta_1, q^n \xi\beta_2 ;q\big)_{\infty}}.
\label{eq:bg1xi2m}
\end{align}
We investigate the limit of the non-homogeneous term in equation~\eqref{eq:nhm2m} as $K \to -\infty $ and $L \to +\infty $.
If $\mu > 0$, then
\begin{align*}
\lim_{L \to +\infty} \big(q^{L+1}\xi\big)^{\mu} \frac{\big(q^{\lambda+L+2}\xi/x, q^{L+1}\xi\alpha_1, q^L\xi\alpha_2 ;q\big)_{\infty}}{\big(q^{L+1}\xi/x, q^{L+1}\xi\beta_1, q^{L+1}\xi\beta_2 ;q\big)_{\infty}} = 0 .
\end{align*}
As equation~\eqref{eq:thetaq}, we have
\begin{align*}
&\big(q^K\xi\big)^{\mu} \frac{\big(q^{\lambda+K+1}\xi/x, q^K\xi\alpha_1, q^{K-1}\xi\alpha_2 ;q\big)_{\infty}}{\big(q^K\xi/x, q^K\xi\beta_1, q^K\xi\beta_2 ;q\big)_{\infty}}
= \big(q^K\xi\big)^{\mu} \bigg( q^{\lambda}\frac{\alpha_1\alpha_2}{\beta_1\beta_2} \bigg)^{-K}
\\
&\qquad{}\times
\frac{\vartheta_q\big(q^{\lambda+1}\xi/x\big) \vartheta_q(\xi\alpha_1) \vartheta_q\big(q^{-1}\xi\alpha_2\big)}{\vartheta_q(\xi/x) \vartheta_q(\xi\beta_1) \vartheta_q(\xi\beta_2)}
 \frac{\big(q^{1-K}x/\xi, q^{1-K}/(\xi\beta_1), q^{1-K}/(\xi\beta_2) ;q\big)_{\infty}}{\big(q^{-\lambda-K}\xi/x, q^{1-K}/(\xi\alpha_1), q^{2-K}/(\xi\alpha_2) ;q\big)_{\infty}} .
\end{align*}
It follows from the condition $q^{\lambda}=q^{\mu} \beta_1\beta_2/(\alpha_1\alpha_2)$ that
\begin{align*}
&\lim_{K \to -\infty} \big(q^K\xi\big)^{\mu} \frac{\big(q^{\lambda+K+1}\xi/x, q^K\xi\alpha_1, q^{K-1}\xi\alpha_2 ;q\big)_{\infty}}{\big(q^K\xi/x, q^K\xi\beta_1, q^K\xi\beta_2 ;q\big)_{\infty}}
= \xi^{\mu}\frac{\vartheta_q\big(q^{\lambda+1}\xi/x\big) \vartheta_q(\xi\alpha_1) \vartheta_q\big(q^{-1}\xi\alpha_2\big)}{\vartheta_q(\xi/x) \vartheta_q(\xi\beta_1) \vartheta_q(\xi\beta_2)}.
\end{align*}
Therefore, equation~\eqref{eq:nhm2m} tends to
\begin{equation*}
\frac{q^2(1-q)\big(1-q^{\lambda}\big)}{\alpha_1\alpha_2} \xi^{\mu}\frac{\vartheta_q\big(q^{\lambda+1}\xi/x\big) \vartheta_q(\xi\alpha_1) \vartheta_q\big(q^{-1}\xi\alpha_2\big)}{\vartheta_q(\xi/x) \vartheta_q(\xi\beta_1) \vartheta_q(\xi\beta_2)}
\end{equation*}
as $K \to -\infty $ and $L \to +\infty $. Hence we obtain the following proposition.

\begin{Proposition}% \label{prop:nhmeq2m}
If $\mu > 0$, then the function $\bar{g}_1(x)$ in equation~\eqref{eq:bg1xi2m} satisfies
\begin{align}
& \bigg( x-\frac{q^{\mu+1}\beta_1}{\alpha_1\alpha_2} \bigg) \bigg( x-\frac{q^{\mu+1}\beta_2}{\alpha_1\alpha_2} \bigg) \bar{g}_1(x/q) + q \bigg( x-\frac{1}{\alpha_1} \bigg) \bigg( x-\frac{q}{\alpha_2} \bigg) \bar{g}_1(qx) \nonumber
\\
&\qquad{}- \bigg\{ (1+q)x^2 - \bigg( q^{\mu+1}\frac{\beta_1+\beta_2}{\alpha_1\alpha_2}+\frac{q}{\alpha_1}+\frac{q^2}{\alpha_2}\bigg) x + \bigg( 1+\frac{\beta_1\beta_2}{\alpha_1\alpha_2} \bigg) \frac{q^{\mu+2}}{\alpha_1\alpha_2} \bigg\} \bar{g}_1(x) \nonumber
\\
&\qquad{}+ \frac{q^2(1-q)\big(1-q^{\lambda}\big)}{\alpha_1\alpha_2} \xi^{\mu}\frac{\vartheta_q(q^{\lambda+1}\xi/x) \vartheta_q(\xi\alpha_1) \vartheta_q(q^{-1}\xi\alpha_2)}{\vartheta_q(\xi/x) \vartheta_q(\xi\beta_1) \vartheta_q(\xi\beta_2)} = 0,
\label{eq:nhmdeg2m}
\end{align}
where $q^{\lambda}=q^{\mu}\beta_1\beta_2/(\alpha_1\alpha_2)$.
\end{Proposition}
Equation~\eqref{eq:nhmdeg2m} is a non-homogeneous extension of equation~\eqref{eq:g1eqdeg2m}. If $\bar{g}_1(x)$ satisfies equation~\eqref{eq:nhmdeg2m}, then it also satisfies the third order difference equation which is obtained from
\begin{equation*}
 \check{Y}(qx)
= \biggl(P^{-1}F_{\infty}P+ \frac{P^{-1}F_1P}{1-\alpha_1 x}+ \frac{P^{-1}F_2P}{1-\alpha_2 x} \biggr) \check{Y}(x), \qquad
\check{Y}(x) = \begin{pmatrix} \bar{g}_1(x) \\ \bar{g}_2(x) \\ \bar{g}_3(x) \end{pmatrix} \!.
\end{equation*}

If $\xi = 1/\alpha_1$, $\xi = 1/\alpha_2$ or $\xi = q^{-\lambda}x$, then
\begin{equation*}
\big(q^K\xi\big)^{\mu} \frac{\big(q^{\lambda+K+1}\xi/x, q^K\xi\alpha_1, q^{K-1}\xi\alpha_2 ;q\big)_{\infty}}{\big(q^K\xi/x, q^K\xi\beta_1, q^K\xi\beta_2 ;q\big)_{\infty}} = 0
\end{equation*}
for any negative integer $K$, and equation~\eqref{eq:nhm2m} tends to $0$ as $K \to -\infty$ and $L \to +\infty$ under the condition $\mu > 0$.

We substitute $\xi = 1/\alpha_1$, $\xi = 1/\alpha_2$ or $\xi = q^{-\lambda}x$ in $\bar{g}_1(x)$. If $\xi = 1/\alpha_1$, then
\begin{align}
\bar{g}_1(x) &= (q\!-\!1)\alpha_1^{-\mu}\frac{\big(q^{\lambda+1}/(\alpha_1 x), q, \alpha_2/\alpha_1;q\big)_{\infty}}{\big(q/(\alpha_1 x), \beta_1/\alpha_1, \beta_2/\alpha_1 ;q\big)_{\infty}} \,{}_3\phi_2 \left(\!\! \begin{array}{c} q/(\alpha_1 x), \beta_1/\alpha_1, \beta_2/\alpha_1\\
 q^{\lambda+1}/(\alpha_1 x), \alpha_2/\alpha_1 \end{array}\! ;q,q^{\mu} \right)\! .
\label{eq:xial1mdeg2}
\end{align}
If $\xi = 1/\alpha_2$, then
\begin{align}
\bar{g}_1(x) ={}& (q-1)q^{\mu}\alpha_2^{-\mu}\nonumber
\\
&\times\frac{\big(q^{\lambda+2}/(\alpha_2 x), q^2\alpha_1/\alpha_2, q ;q\big)_{\infty}}{\big(q^2/(\alpha_2 x), q\beta_1/\alpha_2, q\beta_2/\alpha_2 ;q\big)_{\infty}} \,{}_3\phi_2 \left(\!\! \begin{array}{c} q^2/(\alpha_2 x), q\beta_1/\alpha_2, q\beta_2/\alpha_2\\
 q^{\lambda+2}/(\alpha_2 x), q^2\alpha_1/\alpha_2 \end{array}\! ;q,q^{\mu}\right) \!.
\label{eq:xial2mdeg2}
\end{align}
If $\xi=q^{-\lambda}x$, then
\begin{align}
\bar{g}_1(x) ={}& (q-1)q^{-\lambda\mu} x^{\mu}\nonumber
\\
&\times\frac{\big(q, q^{-\lambda+1}\alpha_1 x, q^{-\lambda}\alpha_2 x;q\big)_{\infty}}{\big(q^{-\lambda+1}, q^{-\lambda}\beta_1 x, q^{-\lambda}\beta_2 x;q\big)_{\infty}}\,{}_3\phi_2 \left(\!\! \begin{array}{c} q^{-\lambda+1}, q^{-\lambda}\beta_1 x, q^{-\lambda}\beta_2 x\\
 q^{-\lambda+1}\alpha_1 x, q^{-\lambda}\alpha_2 x \end{array}\! ;q,q^{\mu} \right)\!.
\label{eq:xilxmdeg2}
\end{align}
Since the non-homogeneous term in equation~\eqref{eq:nhmdeg2m} vanishes in this case, we have

\begin{Proposition} \label{prop:mdeg2homo}
If $\mu > 0$, then the functions in equations~\eqref{eq:xial1mdeg2}, \eqref{eq:xial2mdeg2}, \eqref{eq:xilxmdeg2} satisfy equation~\eqref{eq:g1eqdeg2m}.
\end{Proposition}

To obtain results corresponding to the specialization $\xi= 1/\beta_1$, $\xi= 1/\beta_2$ and $\xi= x $, we replace the functions with equation~\eqref{eq:Ply} with the condition $q^{\mu '} \alpha_1 \alpha_2 /(\beta_1 \beta_2) = q^{\mu}$, i.e., $q^{\mu '}=q^{\lambda}$.
Then the function $y(x) $ also satisfies the $q$-difference equation $y(qx)=B(x)y(x) $, where $B(x)$ is given as equation~\eqref{eq:gqxBxgxqJPN2} with the condition $q^{\lambda}= q^{\mu}\beta_1\beta_2/(\alpha_1\alpha_2)$.
The function $\bar{g}_1^{\,[K,L]}(x) $ is written as
\begin{align*}
 \bar{g}_1^{\,[K,L]}(x) = \frac{1-q}{\alpha_1} x^{\lambda} \sum_{n=K}^{L} (q^n\xi )^{-1} \frac{\big(x q^{-n} / \xi, q^{1-n} /( \beta_1 \xi ), q^{1-n} /( \beta_2 \xi ) ; q\big)_{\infty}}{\big(x q^{-n-\lambda} / \xi, q^{-n} / (\alpha_1 \xi),q^{1-n} / (\alpha_2 \xi) ; q\big)_{\infty}} .
\end{align*}
It converges as $K \to -\infty $ and $L \to +\infty $, if $\mu > 0$.
Write the limit by $\bar{g}_1 (x) $.
Note that the function $\bar{g}_1^{\,[K,L]}(x)$ satisfies equation~\eqref{eq:nhmdeg2}, where the function $Q^{[K,L]}(x) $ is determined by equation~\eqref{eq:QKL}.
The non-homogeneous term in equation~\eqref{eq:nhmdeg2} is written as
\begin{align}
&\frac{q^2(1-q)}{\alpha_1\alpha_2} \bigg\{ \bigg( \frac{\alpha_2}{q}x-q^{\lambda} \bigg) Q^{[K,L]}(x/q) + \bigg( 1-\frac{\alpha_2}{q}x \bigg) Q^{[K,L]}(x) \bigg\} \nonumber
\\
&\qquad = \frac{q(1-q)\big(1-q^{\lambda}\big)}{\alpha_1} x^{\lambda+1}
\biggl( \frac{\big(x q^{1-K}/\xi, q^{1-K}/(\beta_1 \xi ), q^{1-K}/(\beta_2 \xi ) ;q\big)_{\infty}}{\big(x q^{-\lambda -K} / \xi, q^{1-K}/(\alpha_1 \xi ), q^{2-K}/(\alpha_2 \xi ) ;q\big)_{\infty}} \nonumber
\\
&\qquad \hphantom{ = \frac{q(1-q)(1-q^{\lambda})}{\alpha_1} x^{\lambda+1}
\biggl( } - \frac{\big(x q^{-L}/\xi, q^{-L}/(\beta_1 \xi ), q^{-L}/(\beta_2 \xi ) ;q\big)_{\infty}}{\big(x q^{-\lambda -1-L} / \xi, q^{-L}/(\alpha_1 \xi ), q^{1-L}/(\alpha_2 \xi ) ;q\big)_{\infty}} \biggr).
\label{eq:Plnhm}
\end{align}
We investigate the limit of equation~\eqref{eq:Plnhm} as $K \to -\infty$ and $L \to +\infty$. We have
\begin{align*}
\lim_{K \to -\infty} \frac{\big(x q^{1-K}/\xi, q^{1-K}/(\beta_1 \xi ), q^{1-K}/(\beta_2 \xi ) ;q\big)_{\infty}}{\big(x q^{-\lambda -K} / \xi, q^{1-K}/(\alpha_1 \xi ), q^{2-K}/(\alpha_2 \xi ) ;q\big)_{\infty}} = 1.
\end{align*}
If $\xi = 1/\beta_1$, $\xi = 1/\beta_2$ or $\xi = x$, then
\begin{equation*}
\frac{\big(x q^{-L}/\xi, q^{-L}/(\beta_1 \xi ), q^{-L}/(\beta_2 \xi ) ;q\big)_{\infty}}{\big(x q^{-\lambda -1-L} / \xi, q^{-L}/(\alpha_1 \xi ), q^{1-L}/(\alpha_2 \xi ) ;q\big)_{\infty}} = 0
\end{equation*}
for any positive integer $L$, and equation~\eqref{eq:Plnhm} tends to
\begin{equation*}
\frac{q(1-q)\big(1-q^{\lambda}\big)}{\alpha_1} x^{\lambda+1}
\end{equation*}
as $K \to -\infty$ and $L \to +\infty$.

If $\xi = 1/\beta_1$, then
\begin{align}
& \bar{g}_1(x)\! = \frac{(1\!-\!q)\beta_1}{\alpha_1} x^{\lambda} \frac{( \beta_1 x, q, q \beta_1 / \beta_2 ; q)_{\infty}}{\big( q^{-\lambda} x \beta_1, \beta_1 / \alpha_1,q \beta_1 / \alpha_2 ; q\big)_{\infty}} \,{}_3\phi_2 \biggl(\!\! \begin{array}{c} q^{-\lambda}\beta_1 x, \beta_1/\alpha_1, q\beta_1/\alpha_2\\
 \beta_1 x, q\beta_1/\beta_2 \end{array}\!\! ;q,q \biggr) .\!\!\label{eq:xiPlbe1}
\end{align}
The case $\xi= 1 /\beta_2 $ is obtained from the case $\xi=1 /\beta_1 $ by replacing $\beta_1$ with $\beta_2$.
If $\xi=x$, then
\begin{align}
 \bar{g}_1(x)& = \frac{(1-q)q}{\alpha_1} x^{\lambda -1}\nonumber
 \\
 &\phantom{=}\times\frac{\big(q, q^{2} /( \beta_1 x ), q^{2} /( \beta_2 x ) ; q\big)_{\infty}}{\big(q / (\alpha_1 x ),q^{2} / (\alpha_2 x ), q^{1-\lambda} ; q\big)_{\infty}}
\,{}_3\phi_2 \biggl(\!\! \begin{array}{c} q/(\alpha_1 x), q^2/(\alpha_2 x), q^{1-\lambda}\\
 q^2/(\beta_1 x), q^2/(\beta_2 x) \end{array}\! ;q,q \biggr) .
\label{eq:xiPlx}
\end{align}
\begin{Proposition} \label{prop:mdeg2nonhomo}
The functions equations~\eqref{eq:xiPlbe1}, \eqref{eq:xiPlx} satisfy
\begin{align}
& \bigg( x-\frac{q^{\mu+1}\beta_1}{\alpha_1\alpha_2} \bigg) \bigg( x-\frac{q^{\mu+1}\beta_2}{\alpha_1\alpha_2} \bigg) \bar{g}_1(x/q) + q \bigg( x-\frac{1}{\alpha_1} \bigg) \bigg( x-\frac{q}{\alpha_2} \bigg) \bar{g}_1(qx) \nonumber
\\
&\qquad{}- \bigg\{ (1+q)x^2 - \bigg( q^{\mu+1}\frac{\beta_1+\beta_2}{\alpha_1\alpha_2}+\frac{q}{\alpha_1}+\frac{q^2}{\alpha_2}\bigg) x + \bigg( 1+\frac{\beta_1\beta_2}{\alpha_1\alpha_2} \bigg) \frac{q^{\mu+2}}{\alpha_1\alpha_2} \bigg\} \bar{g}_1(x) \nonumber
\\
&\qquad{}+ \frac{q(1-q)\big(1-q^{\lambda}\big)}{\alpha_1} x^{\lambda+1} = 0.
\label{eq:nhm2m1}
\end{align}
\end{Proposition}
Note that, if $\bar{g}_1(x)$ satisfies equation~\eqref{eq:nhm2m1}, then it also satisfies the third order difference equation with the condition $\mu > 0$ and $q^{\lambda}=q^{\mu}\beta_1\beta_2/(\alpha_1\alpha_2)$.

We can replace the parameters by using Proposition~\ref{prop:d2zparam} to fit results on $q$-integrals with the variant of $q$-hypergeometric equation of degree two given in equation~\eqref{eq:varqhgdeg2}.
Then we can obtain the solutions of the variant of $q$-hypergeometric equation of degree two and those of a~non-homogeneous version of the variant of $q$-hypergeometric equation of degree two, which are similar to the ones in Theorems~\ref{thm:deg2-1} and~\ref{thm:deg2-2}.
We omit the details in this paper.

\section{Supplement to Section~\ref{sec:deg3}}

We investigate the $q$-integral representations obtained by the $q$-middle convolution related to the function $(\alpha_1 x, \alpha_2 x,\alpha_3 x;q)_{\infty}/(\beta_1 x, \beta_2 x, \beta_3 x;q)_{\infty}$ with the condition $q^{\lambda} = \beta_1\beta_2\beta_3 /(\alpha_1\alpha_2\alpha_3) $.

We apply Proposition~\ref{prop:convKL} for the case $y(x)= (\alpha_1 x, \alpha_2 x,\alpha_3 x;q)_{\infty}/(\beta_1 x, \beta_2 x, \beta_3 x;q)_{\infty}$.
Then the functions
\begin{align}\label{eq:hyint3}
\begin{split}
& \widehat{y}_{i}^{[K,L]}(x) = (1-q) \sum_{n=K}^{L} s \frac{P_{\lambda}(x, s)}{s-b_{i}} \frac{(\alpha_1 s, \alpha_2 s, \alpha_3 s;q)_{\infty}}{(\beta_1 s, \beta_2 s, \beta_3 s;q)_{\infty}} \bigg|_{s= q^n \xi}, \qquad i=0,1,2,3,
\\
& b_0=0, \qquad
b_1 = 1/\alpha_1, \qquad
b_2 = 1/\alpha_2, \qquad
b_3 = 1/\alpha_3
\end{split}
\end{align}
satisfy
\begin{align*}
& \begin{pmatrix} \widehat{y}_0^{\,[K,L]}(qx) \\[1mm] \widehat{y}_1^{\,[K,L]}(qx) \\[1mm] \widehat{y}_2^{\,[K,L]}(qx) \\[1mm] \widehat{y}_3^{\,[K,L]}(qx) \end{pmatrix}
= H \begin{pmatrix} \widehat{y}_0^{\,[K,L]}(x) \\[1mm] \widehat{y}_1^{\,[K,L]}(x) \\[1mm] \widehat{y}_2^{\,[K,L]}(x) \\[1mm] \widehat{y}_3^{\,[K,L]}(x) \end{pmatrix} + (1-q) Q^{[K,L]}(x) \begin{pmatrix} 1 \\[1mm] -\alpha_1 x /(1-\alpha_1 x) \\[1mm] -\alpha_2 x /(1-\alpha_2 x) \\[1mm] -\alpha_3 x /(1-\alpha_3 x) \end{pmatrix}\!,
\end{align*}
where $H$ is the matrix determined in equation~\eqref{eq:claN3} with the condition $\mu=0 $ and $q^{\lambda} = \beta_1\beta_2\beta_3 /(\alpha_1\alpha_2\alpha_3) $, and
\begin{equation*}
Q^{[K,L]}(x)=P_{\lambda}\big(x, q^{K-1} \xi\big) y \big(q^{K} \xi\big) - P_{\lambda}\big(x, q^{L} \xi\big) y \big(q^{L+1} \xi\big).
\end{equation*}
Note that the convergence theorem in Theorem~\ref{thm:qcintqJH} for the functions $\widehat{y}_i^{\,[K,L]}(x)$, $i=0,1,2,3$, is not applicable.
Write
\begin{align}
\begin{pmatrix}
 \bar{g}_1^{\,[K,L]}(x)\\[1mm]
 \bar{g}_2^{\,[K,L]}(x)\\[1mm]
 \bar{g}_3^{\,[K,L]}(x)\\[1mm]
 \bar{g}_4^{\,[K,L]}(x)
 \end{pmatrix}
=
P^{-1} \begin{pmatrix}
 \widehat{y}_0^{\,[K,L]}(x)\\[1mm]
 \widehat{y}_1^{\,[K,L]}(x)\\[1mm]
 \widehat{y}_2^{\,[K,L]}(x)\\[1mm]
 \widehat{y}_3^{\,[K,L]}(x)
 \end{pmatrix}
=\begin{pmatrix}
 0 & 1 & 0 & -1\\
 0 & 0 & 1 & -1\\
 0 & 0 & 0 & 1\\
 1 & 0 & 0 & -1
 \end{pmatrix}
 \begin{pmatrix}
 \widehat{y}_0^{\,[K,L]}(x)\\[1mm]
 \widehat{y}_1^{\,[K,L]}(x)\\[1mm]
 \widehat{y}_2^{\,[K,L]}(x)\\[1mm]
 \widehat{y}_3^{\,[K,L]}(x)
 \end{pmatrix}\!,
\label{eq:gyPinv}
\end{align}
where $P$ is the matrix determined in equation~\eqref{eq:P}.
Then we have
\begin{align*}
\begin{pmatrix}
 \bar{g}_1^{\,[K,L]}(qx)\\[1mm]
 \bar{g}_2^{\,[K,L]}(qx)\\[1mm]
 \bar{g}_3^{\,[K,L]}(qx)\\[1mm]
 \bar{g}_4^{\,[K,L]}(qx)
 \end{pmatrix}
= P^{-1}HP \begin{pmatrix}
 \bar{g}_1^{\,[K,L]}(x)\\[1mm]
 \bar{g}_2^{\,[K,L]}(x)\\[1mm]
 \bar{g}_3^{\,[K,L]}(x)\\[1mm]
 \bar{g}_4^{\,[K,L]}(x)
 \end{pmatrix}
- (1-q) Q^{[K,L]}(x) P^{-1} \begin{pmatrix}
 -1\\[1mm]
 \alpha_1 x /(1-\alpha_1 x)\\[1mm]
 \alpha_2 x /(1-\alpha_2 x)\\[1mm]
 \alpha_3 x /(1-\alpha_3 x)
\end{pmatrix}\!.
\end{align*}
The first two equations are written as
\begin{align*}
& \bar{g}_1^{\,[K,L]}(qx) = \frac{\alpha_1\alpha_3 x^2+ \big\{\alpha_1(B_1-1)- \alpha_3\big(B_1+q^{\lambda}\big)\big\} x+q^{\lambda}}{(1-\alpha_1 x)(1-\alpha_3 x)}\bar{g}_1^{\,[K,L]}(x)
\\
& \hphantom{\bar{g}_1^{\,[K,L]}(qx) =} + \frac{(\alpha_1-\alpha_3)B_2 x}{(1-\alpha_1 x)(1-\alpha_3 x)}\bar{g}_2^{\,[K,L]}(x) + \frac{(1-q)(\alpha_3-\alpha_1)x}{(1-\alpha_1 x)(1-\alpha_3 x)} Q^{[K,L]}(x),
 \\
&\bar{g}_2^{\,[K,L]}(qx) = \frac{\alpha_2\alpha_3 x^2+ \big\{\alpha_2(B_2-1)-\alpha_1\big(B_2+q^{\lambda}\big) \big\}x +q^{\lambda}}{(1-\alpha_2 x)(1-\alpha_3 x)}\bar{g}_2^{\,[K,L]}(x)
 \\
& \hphantom{\bar{g}_2^{\,[K,L]}(qx) =} + \frac{(\alpha_2-\alpha_3)B_1 x}{(1-\alpha_2 x)(1-\alpha_3 x)}\bar{g}_1^{\,[K,L]}(x) + \frac{(1-q)(\alpha_3-\alpha_2)x}{(1-\alpha_2 x)(1-\alpha_3 x)}Q^{[K,L]}(x).
\end{align*}
We apply Proposition~\ref{prop:g1g2single} to obtain the $q$-difference equation which the function $\bar{g}_1^{\,[K,L]}(x) $ satisfies.
Then we have
\begin{align}
& \biggl( x-q\frac{\beta_1\beta_2}{\alpha_1\alpha_2\alpha_3} \biggr) \biggl( x-q\frac{\beta_2\beta_3}{\alpha_1\alpha_2\alpha_3} \biggr) \biggl( x-q\frac{\beta_3\beta_1}{\alpha_1\alpha_2\alpha_3} \biggr) \bar{g}_1^{\,[K,L]}(x/q) \nonumber
\\
&\qquad{} +q \biggl( x-\frac{1}{\alpha_1} \biggr) \biggl( x-\frac{q}{\alpha_2} \biggr) \biggl( x-\frac{1}{\alpha_3} \biggr) \bar{g}_1^{\,[K,L]}(qx) \nonumber
\\
&\qquad{} +\biggl\{ -(1+q)x^3 + \biggl( \frac{q}{\alpha_1}+\frac{q^2}{\alpha_2}+\frac{q}{\alpha_3}+q\frac{\beta_1\beta_2+\beta_2\beta_3+\beta_3\beta_1}{\alpha_1\alpha_2\alpha_3} \biggr) x^2 \nonumber
\\
&\qquad\hphantom{+\biggl\{} -\biggl( q^2\frac{\beta_1+\beta_2+\beta_3}{\alpha_1\alpha_2\alpha_3}+q^2\frac{\beta_1\beta_2\beta_3}{\alpha_1{\alpha_2}^2{\alpha_3}^2}+q^2\frac{\beta_1\beta_2\beta_3}{{\alpha_1}^2{\alpha_2}^2 \alpha_3}+q\frac{\beta_1\beta_2\beta_3}{{\alpha_1}^2 \alpha_2 {\alpha_3}^2} \biggr) x \nonumber
\\
&\qquad\hphantom{+\biggl\{} +q^2(1+q)\frac{\beta_1\beta_2\beta_3}{{\alpha_1}^2{\alpha_2}^2{\alpha_3}^2} \biggr\} \bar{g}_1^{\,[K,L]}(x) \nonumber
\\
&\qquad{} +q(q-1)\frac{\alpha_1-\alpha_3}{\alpha_1\alpha_3}x \biggl\{\! \biggl( \! x-q\frac{\beta_1\beta_2\beta_3}{\alpha_1{\alpha_2}^2\alpha_3} \biggr) Q^{[K,L]}(x/q)-\!\biggl( \! x-\frac{q}{\alpha_2} \biggr) Q^{[K,L]}(x) \biggr\} = 0,
\label{eq:g1KL3}
\end{align}
and the non-homogeneous term is written as
\begin{align}
& q(q-1)\frac{\alpha_1-\alpha_3}{\alpha_1\alpha_3}x \biggl\{ \biggl( x-q\frac{\beta_1\beta_2\beta_3}{\alpha_1{\alpha_2}^2\alpha_3} \biggr) Q^{[K,L]}(x/q)-\biggl( x-\frac{q}{\alpha_2} \biggr) Q^{[K,L]}(x) \biggr\} \nonumber
\\
& \qquad{}= q^2(q-1)\big(1-q^{\lambda}\big)\frac{\alpha_1-\alpha_3}{\alpha_1\alpha_2\alpha_3}x \biggl( \frac{\big(q^{\lambda+K+1}\xi/x, q^K\xi\alpha_1, q^{K-1}\xi\alpha_2, q^K\xi\alpha_3 ;q\big)_{\infty}}{\big(q^K\xi/x, q^K\xi\beta_1, q^K\xi\beta_2, q^K\xi\beta_3 ;q\big)_{\infty}} \nonumber
\\
&\qquad \hphantom{= q^2(q-1)\big(1-q^{\lambda}\big)\frac{\alpha_1-\alpha_3}{\alpha_1\alpha_2\alpha_3}x \biggl(} -\frac{\big(q^{\lambda+L+2}\xi/x, q^{L+1}\xi\alpha_1, q^L\xi\alpha_2, q^{L+1}\xi\alpha_3 ;q\big)_{\infty}}{\big(q^{L+1}\xi/x, q^{L+1}\xi\beta_1, q^{L+1}\xi\beta_2, q^{L+1}\xi\beta_3 ;q\big)_{\infty}} \biggr).\!\!\!\!
\label{eq:nonhomKL3}
\end{align}
On the other hand, it follows from equations~\eqref{eq:hyint3}, \eqref{eq:gyPinv} that the function $\bar{g}_1^{\,[K,L]}(x) $ is written~as
\begin{align*}
& \bar{g}_1^{\,[K,L]}(x) = \widehat{y}_1^{\,[K,L]}(x)-\widehat{y}_3^{\,[K,L]}(x)
\\
&\qquad{} = (1-q) \sum_{n=K}^{L} \frac{s ( \alpha_3 -\alpha_1)}{(1-\alpha_1 s)(1-\alpha_3 s)} \frac{\big(q^{\lambda+1} s /x,\alpha_1 s, \alpha_2 s, \alpha_3 s;q\big)_{\infty}}{\big(q s /x,\beta_1 s, \beta_2 s, \beta_3 s;q\big)_{\infty}}\bigg|_{s= q^n \xi}
 \\
&\qquad{} = (1-q)(\alpha_3 -\alpha_1 )\sum_{n=K}^{L}c_n,
\\
&c_n = q^n\xi \frac{\big(q^{\lambda+n+1}\xi /x, q^{n+1} \xi\alpha_1, q^n \xi\alpha_2, q^{n+1} \xi\alpha_3; q\big)_{\infty}}{\big(q^{n+1}\xi /x, q^n \xi\beta_1, q^n \xi\beta_2, q^n \xi\beta_3 ;q\big)_{\infty}} .
\end{align*}
Since $c_{n+1}/c_n \to q$, $n\to +\infty$, and $c_{-(n+1)}/c_{-n} \to q^{\lambda +1} \alpha_1 \alpha_2 \alpha_3 /(\beta_1 \beta_2 \beta_3 )= q$, $n\to +\infty$, the function $\bar{g}_1^{\,[K,L]}(x) $ converges as $K \to -\infty $ and $L \to +\infty $.
Write
\begin{align}
 \bar{g}_1 (x)&= \lim_{K \to -\infty \atop{ L \to +\infty}} \bar{g}_1^{\,[K,L]}(x) \nonumber
 \\
& = (1-q)(\alpha_3 -\alpha_1 )\sum_{n=-\infty}^{+\infty} q^n\xi \frac{\big(q^{\lambda+n+1}\xi /x, q^{n+1} \xi\alpha_1, q^n \xi\alpha_2, q^{n+1} \xi\alpha_3; q\big)_{\infty}}{\big(q^{n+1}\xi /x, q^n \xi\beta_1, q^n \xi\beta_2, q^n \xi\beta_3 ;q\big)_{\infty}} .
\label{eq:bg1xi}
\end{align}
Note that the functions $ \widehat{y}_1^{\,[K,L]}(x)$ and $\widehat{y}_3^{\,[K,L]}(x)$ do not converge as $K \to -\infty $.

We investigate the limit of the non-homogeneous term in equation~\eqref{eq:nonhomKL3} as $K \to -\infty $ and $L \to +\infty $.
We have
\begin{equation*}
\lim_{L \to +\infty} \frac{\big(q^{\lambda+L+2}\xi/x, q^{L+1}\xi\alpha_1, q^L\xi\alpha_2, q^{L+1}\xi\alpha_3 ;q\big)_{\infty}}{\big(q^{L+1}\xi/x, q^{L+1}\xi\beta_1, q^{L+1}\xi\beta_2, q^{L+1}\xi\beta_3 ;q\big)_{\infty}} =1 .
\end{equation*}
It follows from the identity
\begin{align*}
& \frac{\big(q^{\lambda+K+1}\xi/x, q^K\xi\alpha_1, q^{K-1}\xi\alpha_2, q^K\xi\alpha_3 ;q\big)_{\infty}}{\big(q^K\xi/x, q^K\xi\beta_1, q^K\xi\beta_2, q^K\xi\beta_3 ;q\big)_{\infty}}
\\
& \qquad = \frac{\vartheta_q\big(q^{\lambda+K+1}\xi/x\big) \vartheta_q\big(q^K\xi\alpha_1\big) \vartheta_q\big(q^{K-1}\xi\alpha_2\big)\vartheta_q\big(q^K\xi\alpha_3\big)}{\vartheta_q\big(q^K\xi/x\big) \vartheta_q\big(q^K\xi\beta_1\big) \vartheta_q\big(q^K\xi\beta_2\big)\vartheta_q\big(q^K\xi\beta_3\big)}
 \\
&\qquad\phantom{=} \times\frac{\big(q^{1-K}x/\xi, q^{1-K}/(\xi\beta_1), q^{1-K}/(\xi\beta_2), q^{1-K}/(\xi\beta_3) ;q\big)_{\infty}}{\big(q^{-\lambda-K}x/\xi, q^{1-K}/(\xi\alpha_1), q^{2-K}/(\xi\alpha_2), q^{1-K}/(\xi\alpha_3) ;q\big)_{\infty}},
\end{align*}
which is obtained similarly to equation~\eqref{eq:thetaq}, and the condition $q^{\lambda} = \beta_1 \beta_2 \beta_3 /(\alpha_1 \alpha_2 \alpha_3 ) $ that
\begin{align*}
 \lim_{K \to -\infty} &\frac{\big(q^{\lambda+K+1}\xi/x, q^K\xi\alpha_1, q^{K-1}\xi\alpha_2, q^K\xi\alpha_3 ;q\big)_{\infty}}{\big(q^K\xi/x, q^K\xi\beta_1, q^K\xi\beta_2, q^K\xi\beta_3 ;q\big)_{\infty}}
\\
& = \frac{\vartheta_q\big(q^{\lambda+1}\xi/x\big) \vartheta_q(\xi\alpha_1) \vartheta_q\big(q^{-1}\xi\alpha_2\big) \vartheta_q(\xi\alpha_3)}{\vartheta_q(\xi/x) \vartheta_q(\xi\beta_1) \vartheta_q(\xi\beta_2) \vartheta_q(\xi\beta_3)}.
\end{align*}
Therefore, equation~\eqref{eq:nonhomKL3} tends to
\begin{equation*}
q^2(q-1)\big(1-q^{\lambda}\big)\frac{\alpha_1-\alpha_3}{\alpha_1\alpha_2\alpha_3}x \bigg( \frac{\vartheta_q\big(q^{\lambda+1}\xi/x\big) \vartheta_q(\xi\alpha_1) \vartheta_q\big(q^{-1}\xi\alpha_2\big) \vartheta_q(\xi\alpha_3)}{\vartheta_q(\xi/x) \vartheta_q(\xi\beta_1) \vartheta_q(\xi\beta_2) \vartheta_q(\xi\beta_3)} -1 \bigg)
 %\label{eq:nonhomlimKL3}
\end{equation*}
as $K \to -\infty $ and $L \to +\infty $.
Hence we obtain the following proposition.

\begin{Proposition} \label{prop:g1nonhom}
The function $ \bar{g}_1 (x)$ in equation~\eqref{eq:bg1xi} satisfies
\begin{align}
& \biggl( x-q\frac{\beta_1\beta_2}{\alpha_1\alpha_2\alpha_3} \biggr) \biggl( x-q\frac{\beta_2\beta_3}{\alpha_1\alpha_2\alpha_3} \biggr) \biggl( x-q\frac{\beta_3\beta_1}{\alpha_1\alpha_2\alpha_3} \biggr) \bar{g}_1(x/q) \nonumber
\\
& \qquad{}+q \biggl( x-\frac{1}{\alpha_1} \biggr) \biggl( x-\frac{q}{\alpha_2} \biggr) \biggl( x-\frac{1}{\alpha_3} \biggr) \bar{g}_1(qx) \nonumber
\\
&\qquad{} +\biggl\{ -(1+q)x^3 + \biggl( \frac{q}{\alpha_1}+\frac{q^2}{\alpha_2}+\frac{q}{\alpha_3} +q\frac{\beta_1\beta_2+\beta_2\beta_3+\beta_3\beta_1}{\alpha_1\alpha_2\alpha_3} \biggr) x^2 \nonumber
\\
&\qquad\hphantom{ +\biggl\{} -\biggl( q^2\frac{\beta_1+\beta_2+\beta_3}{\alpha_1\alpha_2\alpha_3} +q^2\frac{\beta_1\beta_2\beta_3}{\alpha_1{\alpha_2}^2{\alpha_3}^2}+q^2\frac{\beta_1\beta_2\beta_3}{{\alpha_1}^2{\alpha_2}^2 \alpha_3}+q\frac{\beta_1\beta_2\beta_3}{{\alpha_1}^2 \alpha_2 {\alpha_3}^2} \biggr) x \nonumber
\\
&\qquad\hphantom{ +\biggl\{} +q^2(1+q)\frac{\beta_1\beta_2\beta_3}{{\alpha_1}^2{\alpha_2}^2{\alpha_3}^2} \biggr\} \bar{g}_1(x) \nonumber
\\
& \qquad{} +\!q^2(q-\!1)\big(1\!-q^{\lambda}\big)\frac{\alpha_1-\alpha_3}{\alpha_1\alpha_2\alpha_3}x \biggl( \frac{\vartheta_q(q^{\lambda+1}\xi/x) \vartheta_q(\xi\alpha_1) \vartheta_q(q^{-1}\xi\alpha_2) \vartheta_q(\xi\alpha_3)}{\vartheta_q(\xi/x) \vartheta_q(\xi\beta_1) \vartheta_q(\xi\beta_2) \vartheta_q(\xi\beta_3)} -1\! \biggr)\! = 0,\!\!
\label{eq:g1nonhom}
\end{align}
where $q^{\lambda}= \beta_1\beta_2\beta_3 /(\alpha_1 \alpha_2 \alpha_3)$.
\end{Proposition}

Equation~\eqref{eq:g1nonhom} is a non-homogeneous extension of equation~\eqref{eq:g1eqdeg3}.
We investigate a relationship with the fourth order difference equation in equation~\eqref{eq:g1q4}.
By imposing the condition in equation~\eqref{eq:condmula}, it follows that the $q$-difference equation of $\bar{g}_1( x) $ in equation~\eqref{eq:g1q4} is factorized~as
\begin{align}
& \biggl( T_x - q^2 \frac{\beta_1 \beta_2 \beta_3}{\alpha_1 \alpha_2 \alpha_3} \biggr) (T_x -q) \biggl[ (q \alpha_1 x -1) (\alpha_2 x -1) (q \alpha_3 x -1) T_x^2 \nonumber
\\
& \qquad - \biggl\{ q(q+1) \alpha_1 \alpha_2 \alpha_3 x^3 - q (\alpha_1 \alpha_2 + \alpha_2 \alpha_3 + q \alpha_3 \alpha_1 +\beta_1 \beta_2 + \beta_2 \beta_3 +\beta_3 \beta_1) x^2 \nonumber
\\
& \qquad\hphantom{ - \biggl\{} + q \biggl( \beta_1 + \beta_2 + \beta_3 +\frac{\beta_1 \beta_2 \beta_3}{\alpha_1 \alpha_2 \alpha_3} \biggl( \alpha_1 + \frac{\alpha_2}{q} + \alpha_3 \biggr) \biggr) x - (q+1)\frac{\beta_1 \beta_2 \beta_3}{\alpha_1 \alpha_2 \alpha_3} \biggr\} T_x \nonumber
\\
& \qquad + \frac{q (\alpha_1 \alpha_2 \alpha_3 x -\beta_1 \beta_2) (\alpha_1 \alpha_2 \alpha_3 x- \beta_2 \beta_3) (\alpha_1 \alpha_2 \alpha_3 x -\beta_3 \beta_1)}{\alpha_1^2 \alpha_2^2 \alpha_3^2} \biggr] \bar{g}_1( x) =0.
\label{eq:deg3facto}
\end{align}
Therefore, if $\bar{g}_1(x) $ satisfies equation~\eqref{eq:g1nonhom}, then it also satisfies equation~\eqref{eq:g1q4} with the condition in equation~\eqref{eq:condmula}.

If $\xi= 1 /\alpha_1$, $\xi= 1 /\alpha_2 $, $\xi= 1 /\alpha_3 $ or $\xi=q^{-\lambda} x$, then
\begin{align*}
\frac{\big(q^{\lambda+K+1}\xi/x, q^K\xi\alpha_1, q^{K-1}\xi\alpha_2, q^K\xi\alpha_3 ;q\big)_{\infty}}{\big(q^K\xi/x, q^K\xi\beta_1, q^K\xi\beta_2, q^K\xi\beta_3 ;q\big)_{\infty}} =0
\end{align*}
for any negative integer $K$, and equation~\eqref{eq:nonhomKL3} tends to
\begin{equation*}
-q^2(q-1)\big(1-q^{\lambda}\big)\frac{\alpha_1-\alpha_3}{\alpha_1\alpha_2\alpha_3}x
\end{equation*}
as $K \to -\infty $ and $L \to +\infty $.

We substitute $\xi= 1 /\alpha_1$, $\xi= 1 /\alpha_2 $, $\xi= 1 /\alpha_3 $ or $\xi=q^{-\lambda} x$ in $\bar{g}_1(x) $.
If $\xi= 1 /\alpha_1 $, then
\begin{align}
\bar{g}_1(x) &= (1-q)\frac{\alpha_3 -\alpha_1}{\alpha_1}\frac{\big(q^{\lambda+1}/(\alpha_1 x), \alpha_2/\alpha_1, q\alpha_3/\alpha_1, q;q\big)_{\infty}}{\big(q/(\alpha_1 x), \beta_1/\alpha_1, \beta_2/\alpha_1, \beta_3/\alpha_1;q\big)_{\infty}} \nonumber
\\
&\phantom{=} \times {}_4\phi_3 \biggl(\!\! \begin{array}{c}q/(\alpha_1 x), \beta_1/\alpha_1, \beta_2/\alpha_1, \beta_3/\alpha_1\\ q^{\lambda+1}/(\alpha_1 x), \alpha_2/\alpha_1, q\alpha_3/\alpha_1 \end{array};q,q \biggr) .
\label{eq:xi1al1deg3}
\end{align}
If $\xi=1 /\alpha_2 $, then
\begin{align}
\bar{g}_1(x) &= (1-q)q\frac{\alpha_3 -\alpha_1}{\alpha_2}\frac{\big(q^{\lambda+2}/(\alpha_2 x), q^2\alpha_1/\alpha_2, q^2\alpha_3/\alpha_2, q;q\big)_{\infty}}{\big(q^2/(\alpha_2 x), q\beta_1/\alpha_2, q\beta_2/\alpha_2, q\beta_3/\alpha_2 ;q\big)_{\infty}} \nonumber
\\
&\phantom{=} \times {}_4\phi_3 \biggl(\!\! \begin{array}{c}q^2/(\alpha_2 x), q\beta_1/\alpha_2, q\beta_2/\alpha_2, q\beta_3/\alpha_2 \\ q^{\lambda+2}/(\alpha_2 x), q^2\alpha_1/\alpha_2, q^2\alpha_3/\alpha_2 \end{array};q,q\biggr) .
\label{eq:xi1al2deg3}
\end{align}
If $\xi=1 /\alpha_3$, then
\begin{align}
\bar{g}_1(x) &= (1-q)\frac{\alpha_3 -\alpha_1}{\alpha_3}\frac{\big(q^{\lambda+1}/(\alpha_3 x), q\alpha_1/\alpha_3, \alpha_2/\alpha_3, q ;q\big)_{\infty}}{\big(q/(\alpha_3 x), \beta_1/\alpha_3, \beta_2/\alpha_3, \beta_3/\alpha_3 ;q\big)_{\infty}} \nonumber
\\
&\phantom{=} \times {}_4\phi_3 \biggl(\!\! \begin{array}{c}q/(\alpha_3 x), \beta_1/\alpha_3, \beta_2/\alpha_3, \beta_3/\alpha_3\\ q^{\lambda+1}/(\alpha_3 x), q\alpha_1/\alpha_3, \alpha_2/\alpha_3 \end{array};q,q \biggr) .
\label{eq:xi1al3deg3}
\end{align}
If $\xi=q^{-\lambda}x$, then
\begin{align}
\bar{g}_1(x) &= (1-q)q^{-\lambda}(\alpha_3 -\alpha_1)x\frac{\big( q^{-\lambda+1}\alpha_1 x, q^{-\lambda}\alpha_2 x, q^{-\lambda+1}\alpha_3 x, q ;q\big)_{\infty}}{\big(q^{-\lambda}\beta_1 x, q^{-\lambda}\beta_2 x, q^{-\lambda}\beta_3 x, q^{-\lambda+1} ;q\big)_{\infty}} \nonumber
\\
&\phantom{=} \times {}_4\phi_3 \biggl(\!\! \begin{array}{c} q^{-\lambda}\beta_1 x, q^{-\lambda}\beta_2 x, q^{-\lambda}\beta_3 x, q^{-\lambda+1} \\
 q^{-\lambda+1}\alpha_1 x, q^{-\lambda}\alpha_2 x, q^{-\lambda+1}\alpha_3 x \end{array};q,q\biggr) .
\label{eq:xi1al4deg3}
\end{align}
\begin{Proposition} \label{prop:deg3nonhom1}
The functions in equations~\eqref{eq:xi1al1deg3}, \eqref{eq:xi1al2deg3}, \eqref{eq:xi1al3deg3}, \eqref{eq:xi1al4deg3} satisfy
\begin{align*}
& \biggl( x-q\frac{\beta_1\beta_2}{\alpha_1\alpha_2\alpha_3} \biggr) \biggl( x-q\frac{\beta_2\beta_3}{\alpha_1\alpha_2\alpha_3} \biggr) \biggl( x-q\frac{\beta_3\beta_1}{\alpha_1\alpha_2\alpha_3} \biggr) \bar{g}_1(x/q)
\\
&\qquad +q \biggl( x-\frac{1}{\alpha_1} \biggr) \biggl( x-\frac{q}{\alpha_2} \biggr) \biggl( x-\frac{1}{\alpha_3} \biggr) \bar{g}_1(qx)
 \\
&\qquad +\biggl\{ -(1+q)x^3 + \biggl( \frac{q}{\alpha_1}+\frac{q^2}{\alpha_2}+\frac{q}{\alpha_3}+q\frac{\beta_1\beta_2+\beta_2\beta_3+\beta_3\beta_1}{\alpha_1\alpha_2\alpha_3} \biggr) x^2
 \\
&\qquad\hphantom{ +\biggl\{} -\biggl( q^2\frac{\beta_1+\beta_2+\beta_3}{\alpha_1\alpha_2\alpha_3}+q^2\frac{\beta_1\beta_2\beta_3}{\alpha_1{\alpha_2}^2{\alpha_3}^2}+q^2\frac{\beta_1\beta_2\beta_3}{{\alpha_1}^2{\alpha_2}^2 \alpha_3}+q\frac{\beta_1\beta_2\beta_3}{{\alpha_1}^2 \alpha_2 {\alpha_3}^2} \biggr) x
 \\
&\qquad\hphantom{ +\biggl\{} +q^2(1+q)\frac{\beta_1\beta_2\beta_3}{{\alpha_1}^2{\alpha_2}^2{\alpha_3}^2} \biggr\} \bar{g}_1(x) -q^2(q-1)\big(1-q^{\lambda}\big)\frac{\alpha_1-\alpha_3}{\alpha_1\alpha_2\alpha_3}x = 0.
\end{align*}
\end{Proposition}

To obtain results corresponding to the specialization $\xi= 1/\beta_1$, $\xi= 1/\beta_2$, $\xi= 1/\beta_3$ and $\xi= x $, we replace the functions with
\begin{align*}
& P_{\lambda}(x, s) = (x/s)^{\lambda} \frac{(x/s;q)_{\infty}}{(q^{-\lambda} x/s ;q)_{\infty}}, \qquad
y(x)=x^{\mu '} \frac{(q/(\beta_1 x), q/(\beta_2 x), q/(\beta_3 x) ;q)_{\infty}}{(q/(\alpha_1 x), q/(\alpha_2 x), q/(\alpha_3 x) ;q)_{\infty}}
\end{align*}
with the condition $ q^{\mu '} \alpha_1 \alpha_2 \alpha_3 /(\beta_1 \beta_2 \beta_3 ) =1$, i.e.,~$q^{\mu '} =q^{\lambda} $.
Then the function $y(x) $ also satisfies the $q$-difference equation $y(qx)=B(x)y(x) $, where $B(x)$ is given as equation~\eqref{eq:gqxBxgxqJPN3} with the condition $\mu =0$.
The function $\bar{g}_1^{\,[K,L]}(x) $ is written as
\begin{align*}
 \bar{g}_1^{\,[K,L]}(x) & \!=\! (1\!-\!q)\frac{\alpha_3\!-\!\alpha_1}{\alpha_1 \alpha_3} x^{\lambda}
 \! \!\sum_{n=K}^{L}\! (q^n\xi )^{-1} \frac{\big(x q^{-n} / \xi, q^{1-n} /( \beta_1 \xi ), q^{1-n} /( \beta_2 \xi ), q^{1-n} /( \beta_3 \xi ) ; q\big)_{\infty}}{\big(x q^{-n-\lambda} / \xi, q^{-n} / (\alpha_1 \xi),q^{1-n} / (\alpha_2 \xi), q^{-n} / (\alpha_3 \xi\big) ; q)_{\infty}},
\end{align*}
and it converges as $K \to -\infty $ and $L \to +\infty $.
Write the limit by $\bar{g}_1 (x) $.
Note that the function~$\bar{g}_1^{\,[K,L]}(x)$ satisfies equation~\eqref{eq:g1KL3} where
\begin{align*}
 Q^{[K,L]}(x)&=P_{\lambda}\big(x, q^{K-1} \xi\big) y\big(q^{K} \xi\big) - P_{\lambda}\big(x, q^{L} \xi\big) y\big(q^{L+1} \xi\big) \nonumber\\
& = (q x )^{\lambda} \biggl\{ \frac{\big(x q^{1-K}/\xi, q^{1-K}/(\beta_1 \xi ), q^{1-K}/(\beta_2 \xi ), q^{1-K}/(\beta_3 \xi ) ;q\big)_{\infty}}{\big(x q^{-\lambda +1-K} / \xi, q^{1-K}/(\alpha_1 \xi ), q^{1-K}/(\alpha_2 \xi ), q^{1-K}/(\alpha_3 \xi ) ;q\big)_{\infty}} \nonumber \\
& \qquad \qquad - \frac{\big(x q^{-L}/\xi, q^{-L}/(\beta_1 \xi ), q^{-L}/(\beta_2 \xi ), q^{-L}/(\beta_3 \xi ) ;q\big)_{\infty}}{\big(x q^{-\lambda -L} / \xi, q^{-L}/(\alpha_1 \xi ), q^{-L}/(\alpha_2 \xi ), q^{-L}/(\alpha_3 \xi \big) ;q)_{\infty}} \biggr\} .
\end{align*}
The non-homogeneous term in equation~\eqref{eq:g1KL3} is written as
\begin{align}
& q(q-1)\frac{\alpha_1-\alpha_3}{\alpha_1\alpha_3}x \biggl\{ \biggl( x-q\frac{\beta_1\beta_2\beta_3}{\alpha_1{\alpha_2}^2\alpha_3} \biggr) Q^{[K,L]}(x/q)-\biggl( x-\frac{q}{\alpha_2} \biggr) Q^{[K,L]}(x) \biggr\} \nonumber
\\
&\qquad = q(q-1)(1-q^{\lambda}) \frac{\alpha_1-\alpha_3}{\alpha_1\alpha_3} x^{\lambda +2} \nonumber
\\
&\qquad \hphantom{=} \times\biggl( \frac{\big(x q^{1-K}/\xi, q^{1-K}/(\beta_1 \xi ), q^{1-K}/(\beta_2 \xi ), q^{1-K}/(\beta_3 \xi ) ;q\big)_{\infty}}{\big(x q^{-\lambda -K} / \xi, q^{1-K}/(\alpha_1 \xi ), q^{2-K}/(\alpha_2 \xi ), q^{1-K}/(\alpha_3 \xi \big) ;q)_{\infty}} \nonumber
\\
&\qquad \hphantom{=\times\biggl(} - \frac{\big(x q^{-L}/\xi, q^{-L}/(\beta_1 \xi ), q^{-L}/(\beta_2 \xi ), q^{-L}/(\beta_3 \xi ) ;q\big)_{\infty}}{\big(x q^{-\lambda -1 -L} / \xi, q^{-L}/(\alpha_1 \xi ), q^{1-L}/(\alpha_2 \xi ), q^{-L}/(\alpha_3 \xi ) ;q\big)_{\infty}} \biggr).
\label{eq:nonhomKL32}
\end{align}
We investigate the limit of equation~\eqref{eq:nonhomKL32} as $K \to -\infty $ and $L \to +\infty $.
We have
\begin{align*}
& \lim_{K \to -\infty} \frac{\big(x q^{1-K}/\xi, q^{1-K}/(\beta_1 \xi ), q^{1-K}/(\beta_2 \xi ), q^{1-K}/(\beta_3 \xi ) ;q\big)_{\infty}}{\big(x q^{-\lambda -K} / \xi, q^{1-K}/(\alpha_1 \xi ), q^{2-K}/(\alpha_2 \xi ), q^{1-K}/(\alpha_3 \xi ) ;q\big)_{\infty}} =1 .
\end{align*}
If $\xi=1 /\beta_1$, $\xi=1/ \beta_2$, $\xi=1/ \beta_3$ or $\xi= x$, then
\begin{align*}
& \frac{\big(x q^{-L}/\xi, q^{-L}/(\beta_1 \xi ), q^{-L}/(\beta_2 \xi ), q^{-L}/(\beta_3 \xi ) ;q\big)_{\infty}}{\big(x q^{-\lambda -1 -L} / \xi, q^{-L}/(\alpha_1 \xi ), q^{1-L}/(\alpha_2 \xi ), q^{-L}/(\alpha_3 \xi ) ;q\big)_{\infty}} =0
\end{align*}
for any positive integer $L$, and equation~\eqref{eq:nonhomKL32} tends to
\begin{align*}
& q(q-1)\big(1-q^{\lambda}\big) \frac{\alpha_1-\alpha_3}{\alpha_1\alpha_3} x^{\lambda +2}
\end{align*}
as $K \to -\infty $ and $L \to +\infty $.

If $\xi= 1 /\beta_1$, then
\begin{align}
\bar{g}_1(x) & = (1-q) \frac{\alpha_3-\alpha_1}{\alpha_1 \alpha_3} \beta_1 x^{\lambda} \frac{(\beta_1 x, q \beta_1 / \beta_2, q \beta_1 / \beta_3, q ; q)_{\infty}}{\big(q^{-\lambda} \beta_1 x, \beta_1 / \alpha_1,q \beta_1 / \alpha_2, \beta_1 / \alpha_3 ; q\big)_{\infty}} \nonumber
\\
&\phantom{=} \times {}_4\phi_3 \biggl(\!\! \begin{array}{c} q^{-\lambda} \beta_1 x, \beta_1 / \alpha_1,q \beta_1 / \alpha_2, \beta_1 / \alpha_3 \\
 \beta_1 x, q \beta_1 / \beta_2, q \beta_1 / \beta_3 \end{array}\! ;q,q \biggr) .
\label{eq:xi1be1deg3}
\end{align}
The case $\xi= 1 /\beta_ j $, $j=2,3$, is obtained from the case $\xi=1 /\beta_1 $ by replacing $\beta_1$ with $\beta_j$.
If $\xi=x$, then
\begin{align}
 \bar{g}_1(x) & = (1-q)q \frac{\alpha_3-\alpha_1}{\alpha_1 \alpha_3} x^{\lambda -1} \frac{\big( q^{2} /( \beta_1 x ), q^{2} /( \beta_2 x ), q^{2} /( \beta_3 x ), q ; q\big)_{\infty}}{\big( q / (\alpha_1 x ),q^{2} / (\alpha_2 x ), q / (\alpha_3 x ), q^{1-\lambda} ; q\big)_{\infty}} \nonumber
 \\
&\phantom{=} \times {}_4 \phi_3 \biggl(\!\! \begin{array}{c} q / (\alpha_1 x ),q^{2} / (\alpha_2 x ), q / (\alpha_3 x ), q^{1-\lambda} \\
 q^{2} /( \beta_1 x ), q^{2} /( \beta_2 x ), q^{2} /( \beta_3 x ) \end{array}\! ;q,q \biggr) .
\label{eq:xi1be4deg3}
\end{align}

\begin{Proposition} \label{prop:deg3nonhom2}
The functions in equations~\eqref{eq:xi1be1deg3}, \eqref{eq:xi1be4deg3} satisfy
\begin{align}
& \biggl( x-q\frac{\beta_1\beta_2}{\alpha_1\alpha_2\alpha_3} \biggr) \biggl( x-q\frac{\beta_2\beta_3}{\alpha_1\alpha_2\alpha_3} \biggr) \biggl( x-q\frac{\beta_3\beta_1}{\alpha_1\alpha_2\alpha_3} \biggr) \bar{g}_1(x/q) \nonumber
\\
&\qquad +q \biggl( x-\frac{1}{\alpha_1} \biggr) \biggl( x-\frac{q}{\alpha_2} \biggr) \biggl( x-\frac{1}{\alpha_3} \biggr) \bar{g}_1(qx) \nonumber
\\
&\qquad +\biggl\{ -(1+q)x^3 + \biggl( \frac{q}{\alpha_1}+\frac{q^2}{\alpha_2}+\frac{q}{\alpha_3} +q\frac{\beta_1\beta_2+\beta_2\beta_3+\beta_3\beta_1}{\alpha_1\alpha_2\alpha_3} \biggr) x^2 \nonumber
\\
&\qquad\hphantom{+\biggl\{} -\biggl( q^2\frac{\beta_1+\beta_2+\beta_3}{\alpha_1\alpha_2\alpha_3} +q^2\frac{\beta_1\beta_2\beta_3}{\alpha_1{\alpha_2}^2{\alpha_3}^2} +q^2\frac{\beta_1\beta_2\beta_3}{{\alpha_1}^2{\alpha_2}^2 \alpha_3}+q\frac{\beta_1\beta_2\beta_3}{{\alpha_1}^2 \alpha_2 {\alpha_3}^2} \biggr) x \nonumber
\\
&\qquad\hphantom{+\biggl\{} +q^2(1+q)\frac{\beta_1\beta_2\beta_3}{{\alpha_1}^2{\alpha_2}^2{\alpha_3}^2} \biggr\} \bar{g}_1(x) +q(q-1)\big(1-q^{\lambda}\big) \frac{\alpha_1-\alpha_3}{\alpha_1\alpha_3} x^{\lambda +2} = 0.
\label{eq:g1nonhom2}
\end{align}
\end{Proposition}
It follows from equation~\eqref{eq:deg3facto} that, if $\bar{g}_1(x) $ satisfies equation~\eqref{eq:g1nonhom2}, then it also satisfies equation~\eqref{eq:g1q4} with the condition in equation~\eqref{eq:condmula}.

We replace the parameters by using Proposition~\ref{prop:paramreldeg3} to fit results on $q$-integrals with the variant of $q$-hypergeometric equation of degree three given in equation~\eqref{eq:varqhgdeg3}.
We now replace the parameters in Proposition~\ref{prop:g1nonhom} by using Proposition~\ref{prop:paramreldeg3}.
Set $g ^{\langle 1 \rangle}(x) = x^{-\alpha} \bar{g}_1(x) /\{(1-q)(\alpha_3 -\alpha_1) \}$.
Then equation~\eqref{eq:g1nonhom} is replaced by
\begin{align}
& \big(x-q^{h_1 +1/2} t_1\big) \big(x- q^{h_2 +1/2} t_2\big) \big(x- q^{h_3 +1/2} t_3\big) g ^{\langle 1 \rangle} (x/q) \nonumber
\\
&\qquad + q^{2\alpha +1} \big(x - q^{l_1-1/2}t_1 \big) \big(x - q^{l_2 -1/2} t_2\big) \big(x - q^{l_3 -1/2} t_3\big) g ^{\langle 1 \rangle} (qx) \nonumber
\\
&\qquad + q^{\alpha} \bigl[ - (q + 1 ) x^3 + q^{1/2} \big\{ \big(q^{h_1} + q^{l_1}\big)t_1 + \big(q^{h_2} + q^{l_2}\big)t_2 + \big(q^{h_3} + q^{l_3}\big)t_3\big \} x^2 \nonumber
\\
& \qquad\hphantom{+ q^{\alpha} \bigl[} - q^{(h_1+h_2+h_3+l_1+l_2+l_3 +1)/2} \big\{ \big(q^{- h_1}+q^{-l_1}\big)t_2 t_3 + \big(q^{- h_2}+ q^{- l_2}\big) t_1 t_3 \nonumber
\\
& \qquad\hphantom{+ q^{\alpha} \bigl[} + \big(q^{- h_3}+ q^{- l_3}\big) t_1 t_2 \big\} x + q^{(h_1 +h_2 + h_3 + l_1 + l_2 + l_3 )/2} ( q + 1 ) t_1 t_2t_3 \bigr] g ^{\langle 1 \rangle} (x) \nonumber
\\
&\qquad +q^{\alpha +l_1 +l_2 +l_3 -1/2} t_1 t_2 t_3 \big(1-q^{\lambda}\big) x^{1-\alpha} \nonumber
\\
& \qquad\hphantom{+}\times\! \biggl( \frac{\vartheta_q\big(q^{\lambda+1}\xi/x\big) \vartheta_q\big(\xi q^{- l_1 + 1/2}/t_1\big) \vartheta_q\big( \xi q^{-l_2 +1/2} /t_2 \big) \vartheta_q\big(\xi q^{- l_3 + 1/2}/t_3 \big)}{\vartheta_q(\xi/x) \vartheta_q\big(\xi q^{\lambda -h_1 +1/2} /t_1 \big) \vartheta_q\big(\xi q^{\lambda -h_2 +1/2} /t_2 \big) \vartheta_q\big(\xi q^{\lambda -h_3 +1/2} /t_3 \big)} -1\! \biggr)=0,\!\!\!\!
\label{eq:varqhgdeg3nonhom0}
\end{align}
where $\lambda = (h_1+h_2+h_3-l_1-l_2-l_3 +1 )/2 $.
It is a non-homogeneous version of the variant of $q$-hypergeometric equation of degree three.
Equation~\eqref{eq:bg1xi} is replaced with
\begin{align*}
& g ^{\langle 1 \rangle} (x)= x^{-\alpha} \sum_{n=-\infty}^{+\infty} \!\! q^n\xi \frac{\big(\xi q^{\lambda+n+1} /x, \xi q^{n - l_1 + 3/2} /t_1, \xi q^{n -l_2 +3/2} /t_2, \xi q^{n - l_3 + 3/2}/t_3 ; q\big)_{\infty}}{\big(\xi q^{n+1} /x, \xi q^{n+ \lambda -h_1 +1/2} /t_1, \xi q^{n+ \lambda -h_2 +1/2} /t_2, \xi q^{n+ \lambda -h_3 +1/2} /t_3 ;q\big)_{\infty}}, %\label{eq:gxivqhg3}
\end{align*}
where $\lambda = (h_1+h_2+h_3-l_1-l_2-l_3 +1 )/2 $.
Then it follows from Proposition~\ref{prop:g1nonhom} that the function $g ^{\langle 1 \rangle} (x)$ in the above equation is a solution to equation~\eqref{eq:varqhgdeg3nonhom0}.
By rewriting Proposition~\ref{prop:deg3nonhom1}, we obtain Theorem~\ref{thm:deg3-1}.

To show Theorem~\ref{thm:deg3-2}, we replace the parameters in Proposition~\ref{prop:deg3nonhom2} by using Proposition~\ref{prop:paramreldeg3}.
Set $g^{\langle 2 \rangle} (x) = x^{-\alpha} \bar{g}_1(x) \alpha_1 \alpha_3 /\{q(1-q)(\alpha_3 -\alpha_1) \}$.
Then we obtain Theorem~\ref{thm:deg3-2} by rewriting equations~\eqref{eq:xi1be1deg3}, \eqref{eq:xi1be4deg3}.

\subsection*{Acknowledgements}
The authors are grateful to the referees for the valuable comments.
The second author was supported by JSPS KAKENHI Grant Number JP22K03368.

\pdfbookmark[1]{References}{ref}
\LastPageEnding

\end{document}